\documentclass[11pt]{amsart}
\usepackage{amssymb,amsthm,amsmath}
\usepackage[numbers,sort&compress]{natbib}
\usepackage{color}
\usepackage{graphicx}
\usepackage{tikz}
\usepackage{enumerate}
\usepackage{float}
\numberwithin{equation}{section}
\title[ ]{Quasi-periodic solution of quasi-linear fifth-order K{d}V equation }

\author{Yingte Sun}
\address{School of Mathematical Sciences, Fudan University, Shanghai 200433, P. R. China} \email{sunyt15@fudan.edu.cn}
\author{Xiaoping Yuan}
\address{School of Mathematical Sciences, Fudan University, Shanghai 200433, P. R. China}
\email{xpyuan@fudan.edu.cn}

\newtheorem{thm}{Theorem}[section]
 \newtheorem{coro}[thm]{Corollary}

  \newtheorem{defi}[thm]{Definition}
 \theoremstyle{definition}
 \theoremstyle{remark}
 \newtheorem{rem}[thm]{Remark}
 
 \numberwithin{equation}{section}

\theoremstyle{plain}
\newtheorem{theorem}{Theorem}[section]

\newtheorem{lemma}[theorem]{Lemma}

\theoremstyle{definition}
\newtheorem{definition}[theorem]{Definition}

\begin{document}


\begin{abstract}
In this paper, we prove the existence of quasi-periodic small-amplitude solutions for quasi-linear Hamiltonian perturbation of the fifth-order KdV equation on the torus in presence of a quasi-periodic forcing.
\end{abstract}
\maketitle
\tableofcontents

\maketitle

\tableofcontents

\section{Introduction and main result}
The existence of quasi-periodic solution of Hamiltonian partial differential equations (HPDEs) has been studied for a long time.  The considered HPDEs are usually some linear or nonlinear integrable equations with perturbations.   According to the feature of  them, the perturbations can be classified into bounded and unbounded ones. The HPDEs with bounded perturbations have been firstly studied by Kuksin,Wayne and Bourgain in   \cite{KuKsin0},\cite{Wayne1} and \cite{Bourgain1}. In this direction there are too many references for us not to list them here.  In the present paper, we focus on the HPDEs with unbounded perturbations.

If the perturbation is unbounded, the homological equation in the KAM iteration reads as follows:
\begin{equation}\label{0.01}
-\mathbf{i}\omega \cdot\partial_{\varphi}u+\lambda u +\mu(\varphi)u=p(\varphi),\quad \varphi\in \mathbb{T}^v,
\end{equation}
where $\mu(\varphi)$ has zero average, and $\mu(\varphi)\approx \gamma$ is  usually of large magnitude. The equation of this type  is called a ¡°small-denominator equation with large variable coefficient¡±. It is crucial to get appropriate estimation of such equation. Assuming $\lambda \geq |\omega| \gamma ^{1+\beta}$ for some $\beta > 0$, Kuksin gave a valid estimate of the solution in \cite{KuKsin1},  which is applied to the KdV equation and a whole hierarchy of so called higher order KdV equations, see \cite{Bambusi1}, \cite{KuKsin2} and \cite{Kappler1}.
Subsequently, Liu-Yuan \cite{Liu1} gave a new estimate, including both $\beta >0$ and $\beta = 0$, which extends the application of KAM theory to 1-dimensional derivative NLS (DNLS)
and Benjamin-Ono equations. See \cite{Liu2}, \cite{Zhang1} and \cite{Yuan1}. The case $\beta<0$ is corresponding to quasi-linear or fully nonlinear equations, for which there has not yet any clue to get a required estimation of the solution for $\eqref{0.01}$.

Recently, in a series of papers \cite{Baldi1,Berti0,Berti1,Berti4,Feola1,Feola2,montalto1}, Baldi-Berti-Feola-Montalto invented a sophisticated tool to deal with the case $\beta <0$ for some  quasi-linear or fully nonlinear partial differential equations, such as KdV and water wave equation. Take the  fully nonlinear KdV equation
\begin{equation}
\partial_tu+u_{xxx}+f(\omega t,x,u,u_x,u_{xxx})=0, \quad  x\in \mathbb{T}=\mathbb{R}/ 2\pi\mathbb{Z},
\end{equation} as an example.
By Nash-Moser iteration, the linearized homological equation can be seen as $\mathcal{L}v=F$, where
\begin{equation}
\mathcal{L}=\omega \cdot \partial_{\varphi}+(1+a_3(\varphi,x))\partial^3_x+a_1(\varphi,x)\partial_x+a_0(\varphi,x),\quad \varphi \in \mathbb{T}^v.
\end{equation}
It is crucial to estimate the inverse of the linear operator $\mathcal{L}$. Instead of directly reducing the linear operator $\mathcal{L}$ to a diagonal operator, Baldi-Berti-Feola used some sophisticated way to reduce the linear operator $\mathcal{L}$ to a diagonal operator plus a bounded perturbation by a set of regularization procedures. For example, the regularization procedures for the fully nonlinear KdV equation can be summarized as:

$\bullet$ To eliminate the space variable dependence of the coefficients of $\partial_{x}^3$ by a $\varphi$-dependent changes of variable. Then, to eliminate the time dependence of the coefficients of $\partial_{x}^3$ by a quasi-periodic time re-parametrization(See \cite{Berti1} for detail). The linear operator $\mathcal{L}$ is thus reduced to
\begin{equation}
\mathcal{L}_1=\omega \cdot \partial_\varphi +m_3\partial^3_x +b_1(\varphi,x) \partial_x +b_0(\varphi,x),
\end{equation}
where $m_3 \in \mathbb{R}$ is a constant.

$\bullet$ The regularization procedure to dispose the coefficients of $\partial_x$ can be divided into two steps.

The first step  is to use the space variable change $y=x+p(\varphi)$, by which
 the differential operator $\omega \cdot \partial_\varphi$ becomes
\begin{equation}\label{1.02}
\omega \cdot \partial_\varphi+\omega \cdot \partial_\varphi p(\varphi) \partial_y.
\end{equation}
At this time,  the linear operator $\mathcal{L}_1-\omega \cdot \partial_\varphi$ becomes
\begin{equation}
m_3 \partial^3_y+b_1(\varphi,y-p(\varphi)) \partial_y +b_0(\varphi,y-p(\varphi))
\end{equation}
One sees that the coefficients of $\partial_y$ is  $\omega \cdot \partial_\varphi p(\varphi)+ b_1(\varphi,y-p(\varphi))$. The term $\omega \cdot \partial_\varphi p(\varphi)$ in \eqref{1.02} is used to amend the coefficients of $\partial_y$, which guarantees the spatial average of the coefficients of $\partial_y$ is a constant.

On the basis of the first step, one uses some pseudo-differential operator technique to reduce the linear operator $\mathcal{L}_1$  to
\begin{equation}
\mathcal{L}_2=\omega \cdot \partial_\varphi +m_3\partial^3_x +m_1 \partial_x +\mathcal{R},
\end{equation}
where $m_3,m_1 \in \mathbb{R}$, $\mathcal{R}$ is a bounded  remainder.

 However, there are some difficulties to handle quasi-linear or fully nonlinear higher order KdV equations with the regularization method. Consider the fully nonlinear fifth order KdV equation, for example,
\begin{equation}
\partial_tu+\partial^5_xu+f(\omega t,x,u,\partial_xu,\partial^2_xu,\partial^3_xu,\partial^5_xu)=0, \quad  x\in \mathbb{T}.
\end{equation}
The linearized homological equation still be seen as $\mathfrak{L}v=F$, with
\begin{equation}
\mathfrak{L}=\omega \cdot \partial_\varphi +a_5\partial^5_x+ a_3\partial^3_x+a_2\partial^2_x+a_1\partial_x+a_0, \quad \varphi \in \mathbb{T}^v,
\end{equation}
where $a_i$ is the function of $(\varphi,x)$.

Applying the delicate variable change as the abvoe to $\mathfrak{L}$, then, the linear operator $\mathfrak{L}$ is reduced to
 \begin{equation}
 \mathfrak{L}_1=\omega \cdot \partial_\varphi +m_5\partial^5_x+ b_3\partial^3_x+b_2\partial^2_x+b_1\partial_x+b_0,
 \end{equation}
where $m_5 \in \mathbb{R}$ is constant and $b_i$'s are  function of $(\varphi,x)$.

When one tries to reduce the coefficient  $b_3=b_3(\varphi,x)$ of  $\partial^3_{x}$ to constant, a new difficulty arises:
 the variable change $y=x+p(\varphi)$ to amend the coefficients of $\partial_{x}$ can not be applied to the coefficients of $\partial^3_{x}$.  Therefore, a suitable regularization procedure for the higher-order KdV equations seems to be more complicated and maybe need some new  technical method.

In this paper we make attempt to deal with the problem by combing regularization method by Baidi-Berti-Feola  and the  unbounded reduction method by Kuksin \cite{KuKsin3}.  Consider the Hamiltonian quasi-linear fifth order KdV equation of the following form
\begin{equation}\label{main1}
\partial_{t}u=X_H(u),
\end{equation}
with
\begin{equation}\label{main2}
X_H(u)=\partial_{x}\nabla_{L^2}H(t,x,u,u_x,u_{xx}),
\end{equation}
and
\begin{equation}\label{main2.1}
H(t,x,u,u_x,u_{xx})=\int_{\mathbb{T}}-\frac{1}{2}u^2_{xx}+5uu^2_x-\frac{5}{2}u^4+g(\omega t,x,u,u_x,u_{xx})dx.
\end{equation}
The perturbation $g(\omega t,x,u,u_x,u_{xx})$ is unbounded, and quasi-periodic in time, periodic in space, a polynomial with regard to $u,u_x,u_{xx}$. Here and in other places in this paper, $\int_{\mathbb{T}^d}$ is short for the average$ (2\pi)^{-d}\int_{\mathbb{T}^d}$, by a slight abuse of notation.

The primary fifth-order KdV equation without perturbation $g$ is

 \begin{equation}\label{main2.2}
u_t+\partial^5_{x}u+10u\partial_x^3u+20\partial_xu\partial_x^2u+30 u^2 \partial_xu=0,
\end{equation}
which is a special case of the general fifth-order KdV equation (fKdV) of the following

\begin{equation}
 u_t+\alpha\partial^5_{x}+\beta u\partial_x^3u+\gamma\partial_xu\partial_x^2u+\sigma u^2 \partial_xu=0.
 \end{equation}

The special case (\ref{main2.2}) is called the Lax case, which is characterized by $\beta = 2\gamma$ and $\alpha= \frac{3}{10}\gamma^2$. This general fifth-order KdV equation describes motions of long waves in shallow water under gravity. In a one-dimensional nonlinear lattice, it is an important mathematical model with wide application in quantum mechanics and nonlinear optics. Typical examples are widely used in various fields such as solid state physics, plasma physics, fluid physics, and quantum field theory. The relevant research can be found in \cite{Wazwaz1},\cite{Chun C1} and \cite{Lax1}.

 Since our purpose is to dispose some quasi-linear  problem, we set
 $$g=u^3_{xx}+f(\omega t,x)u.$$
  The concrete form of the equation is
\begin{equation}\label{main3}
u_t=-\partial^5_{x}-10u\partial_x^3u-20\partial_xu\partial_x^2u-30u^2 \partial_xu+6\partial^2_xu \partial^5_xu+18\partial^3_xu\partial^4_xu+ \partial_xf(\omega t,x),
\end{equation}
where $x\in \mathbb{T}=\mathbb{R} /2\pi\mathbb{Z}$, and $\partial_xf(\omega t, x)$ is quasi-periodic in time with Diophantine frequency vector, namely
\begin{equation}\label{main3.0}
 \omega =\lambda \overline{\omega},\quad \lambda \in \Pi =[\frac{1}{2},\frac{3}{2}], \quad |\overline{\omega}\cdot \ell |\geq \frac{\alpha_0} {|\ell|^{\tau_0}}, \quad \forall\ell \in \mathbb{Z}^v\backslash \{0\}.
\end{equation}
\par
Clearly, if $\partial_xf(\omega t, x)$ is not identically zero, then $u = 0$ is not a solution of \eqref{main3}. Thus we look for non-trivial solutions $u(\varphi, x)$ of the fifth-order KdV
equation in the analytical space $\mathrm{H}_{s,p}(\mathbb{T}^{v+1})$
\begin{equation}\label{function}
\begin{split}
 \Big\{ u(\omega t, x)=&\sum_{(\ell.k)\in \mathbb{Z}^v \times \mathbb{Z}}u_{\ell.k}e^{\mathbf{i}\ell \omega t x} e^{\mathbf{i}kx},\\
  &\|u\|^2_{s,p}:=\sum_{(\ell.k)\in \mathbb{Z}^v \times \mathbb{Z}} |u_{\ell.k}|^2 e^{2(|\ell|+|k|)s}([\ell]+[k])^{2p} <+\infty,\Big\}
\end{split}
\end{equation}
where
\begin{equation}
|\ell|=|\ell_1|+\cdots +|\ell_v| \quad [\ell]=\max(|\ell|,1).
\end{equation}

The Banach space $\mathrm{H}_{s,p}$ can be extended to the analytic functions defined on $\mathbb{T}^{b}$ with any integer $b>0$. If $u(\varphi)=\sum_{k\in \mathbb{Z}^b}u_ke^{\mathrm{i}k\varphi},\varphi \in \mathbb{T}^b$, we can denote
 $$(\|u\|_{s,p})^2=\sum_{k\in \mathbb{Z}^b} |u_{k}|^2 e^{2(|k|)s}[k]^{2p}. $$
 For notation convenience, when $p=0$, $\|\cdot\|_{s,0}$ is simplified as $\|\cdot\|_{s}$.

 Our main  result is
\par
\begin{thm}\label{main result} Assume that $\omega$ satisfies Diophantine Condition \eqref{main3.0} and assume that  there are constants $ q=q(v)>0$, $\varepsilon_0=\varepsilon(v)>0$ and $s>0$ such that $$\|\partial_xf(\omega t, x)\|_{s,q}\leq \varepsilon$$ with $\varepsilon \in (0,\varepsilon_0)$ and $\varepsilon_0= \varepsilon_0(v)>0$ being small enough. Then there exists a Cantor set $\Pi_{\varepsilon} \subseteq \Pi $ of asymptotically full Lebesgue measure, i.e,
$$|\Pi_{\varepsilon}|\rightarrow 1 \quad as \quad \varepsilon \rightarrow 0, $$
such that for every $\lambda \in \Pi_{\varepsilon}$, the KdV  equation (\ref{main3}) admits a solution $u(t,x) \in C^\infty(\mathbb{R}\times \mathbb{T})$ which is quasi-periodic in time $t$ with frequency $\omega=\lambda\, \bar \omega$.
\end{thm}

Although just only the quasi-linear fifth-order KdV equation is investigated, the method in Theorem \ref{main result} also applies  to other quasi-linear Hamiltonian higher order KdV equations, for example, the seventh order, even to  some other quasi-linear or fully-nonlinear equations.

\section{Functional setting}
\par
In this section, we introduce some notations, definitions and technical tools,  which will be used in section 3, 4, 5.

The phase space of (\ref{main3}) is
\begin{equation}
\mathrm{H}_0^{1}(\mathbb{T})=\{u(x)\in \mathrm{H}_{0,1}(\mathbb{T},\mathbb{R}): \int_{\mathbb{T}}u(x)dx =0 \}
\end{equation}
endowed with non-degenerate  symplectic form
\begin{equation}\label{1.01}
\Omega(u,v)=\int_{\mathbb{T}}(\partial_x^{-1}u)vdx \quad \forall u,v\in\mathrm{H}_0^1(\mathbb{T}),
\end{equation}
where $\partial_x^{-1}u$ is the periodic primitive of $u$ with zero average.
The Hamiltonian vector field $X_H(u)=\partial_{x}\nabla H(u)$ is the unique vector field satisfying the equality
\begin{equation}
dH(u)[h]=(\nabla H(u),h)_{L^2(\mathbb{T})}=\Omega(X_H(u),h), \forall u,h\in \mathrm{H}^1_0(\mathbb{T}),
\end{equation}
where for all $u,v \in L^2(\mathbb{T})$
\begin{equation}
(u,v)_{L^2(\mathbb{T})}=\int_{\mathbb{T}}u(x)v(x)dx=\sum_{j\in\mathbb{Z}}u_jv_{-j},
\end{equation}
\begin{equation}
u(x)=\sum_{j\in\mathbb{Z}}u_je^{\mathbf{i}jx}.\quad v(x)=\sum_{j\in\mathbb{Z}}v_je^{\mathbf{i}jx}.
\end{equation}
Recall Poisson bracket between two Hamiltonians $F,G:\mathrm{H}^1_0(\mathbb{T})\rightarrow \mathbb{R}$ is
\begin{equation}
{F(u),G(u)}=\Omega(X_F,X_G)=\int_{\mathbb{T}}\nabla F(u)\partial_x\nabla G(u)dx,
\end{equation}

  The function in the present paper  is quasi-periodic in the time variables and periodic in the space variable. This function is analytic for these variables in the domain of $\mathbb{T}^{v+1}_s$, where $\mathbb{T}^{v+1}_s$ be the complexified torus with $|\mathfrak{Im} \phi _i|\leq s $. So, the function will be of the form

\begin{equation}\label{function1}
 u(\omega t.x)=\sum_{(\ell.k)\in \mathbb{Z}^v \times \mathbb{Z}}  u_{\ell.k}e^{\mathbf{i}\ell \omega t} e^{\mathbf{i}kx}=\sum_{(\ell.k)\in \mathbb{Z}^v \times \mathbb{Z}}  u_{\ell.k}e^{\mathbf{i}\ell \varphi} e^{\mathbf{i}kx}.
\end{equation}
Now, we define some important norms:

\begin{defi}\label{norms}
In contrast with Banach space $\mathrm{H}_{s,p}(\mathbb{T}^{v+1})$ defined in \eqref{function}, we define a new Banach space $\mathrm{H}^{\mathfrak{s}}_{s,p}(\mathbb{T}^v \times\mathbb{T})$ as
\begin{equation}
\begin{split}
\bigg\{u(\varphi.x)=&\sum_{(\ell.k)\in \mathbb{Z}^v \times \mathbb{Z}}u_{\ell.k} e^{\mathbf{i}\ell \varphi} e^{\mathbf{i}kx}: \\
&(\|u\|^\mathfrak{s}_{s,p})^2=\sum_{(\ell.k)\in \mathbb{Z}^v \times \mathbb{Z}} |u_{\ell.k}|^2 e^{2(|\ell|+|k|)s}[k]^{2p}[\ell]^{2p}<+\infty \bigg\},
\end{split}
\end{equation}
where
$$|\ell|=|\ell_1|+\cdots +|\ell_v|, \quad [\ell]=max(|\ell|,1).$$

For analytic function $u$ defined on $\mathbb{T}^{v+1}_s$, the max norm plays important role in our paper
$$|u|_{s,p}=\sum _{|k|\leq p} \max_{(\varphi,x)\in \mathbb{T}^{v+1}_s }|D^{k} u(\varphi,x)|. $$
\end{defi}

If $E$ is the space $\mathrm{H}^{\mathfrak{s}}_{s,p}$, we  denote $\|f\|^{Lip}_{E}=\|f\|^{\mathfrak{s}£¬Lip}_{s,p}$

As a notation, we denote $a \lessdot b$ as $a \leq C b$, where $C$ is a constant  depending  on  the form of equation, the number $v$ of frequencies, the diophantine exponent $\tau$ in the non-resonance condition. $ a \approx  b$ means $a \leq C_1 b$ and $C_2a \geq  b$.

When we consider a function $f:\Pi \rightarrow E, \omega \mapsto F(\omega)$, where $(E,\|\|_E)$ is the Banach space and $\Pi$ is the subset of $\mathbb{R}$, we can define sup-norm and Lipschitz  semi-norm below.
\begin{defi}\label{ lipschitz norms}
$$\|f\|^{sup}_E =\|f\|^{sup}_{E,\Pi}=\sup_{\lambda \in \Pi}\|f\|_E,\quad \|f\|^{lip}_E= \|f\|^{lip}_{E,\Pi}=\sup _{\lambda_1,\lambda_2 \in \Pi ,\lambda_1\neq \lambda_2 }\frac{\|f(\lambda_1)-f(\lambda_2)\|_E}{|\lambda_1-\lambda_2|}. $$
Then, the Lipschitz norm is
$$\|f\|^{Lip}_{E} = \|f\|^{Lip}_{E,\Pi}=\|f\|^{sup}_{E}+ \|f\|^{lip}_{E}.$$
\end{defi}
If $E$ is the space $\mathrm{H}_{s,p}$, we denote $\|f\|^{Lip}_{E}$ by $\|f\|^{Lip}_{s,p}$. When $E=\mathrm{H}^{\mathfrak{s}}_{s,p}$, we denote $\|f\|^{Lip}_{E}$ by $\|f\|^{\mathfrak{s},Lip}_{s,p}$.
Now, we show  the relationship between Banach space $\mathrm{H}_{s,p}$ and $ \mathrm{H}^{\mathfrak{s}}_{s,p}$.
\begin{lemma}\label{2.21}
\begin{equation}\label{2.15}
\|u\|^\mathfrak{s}_{s,p} \leq  \|u\|_{s,2p} \leq 2^p \|u\|^\mathfrak{s}_{s,2p}.
\end{equation}
\end{lemma}
\begin{proof}
Notice that $2^p [k]^{p}[\ell]^{p}\leq([k]+[\ell])^{2p} \leq 4^p[k]^{2p}[\ell]^{2p}$,  we can get
\begin{equation}
\begin{split}
(\|u\|^{\mathfrak{s}}_{s,p})^2&=\sum_{(\ell.k)\in \mathbb{Z}^v \times \mathbb{Z}} |u_{\ell.k}|^2 e^{2|\ell|s}[\ell]^{2p}e^{2|k|s}[k]^{2p}\\
&\leq \frac{1}{4^p}\sum_{(\ell.k)\in \mathbb{Z}^v \times \mathbb{Z}} |u_{\ell.k}|^2 e^{2(|\ell|+|k|)s}([k]+[\ell])^{4p}\\
&\leq (\|u\|_{s,2p})^2,
\end{split}
\end{equation}
and

\begin{equation}
 \begin{split}
(\|u\|_{s,p})^2&=\sum_{(\ell.k)\in \mathbb{Z}^v \times \mathbb{Z}} |u_{\ell.k}|^2 e^{2(|\ell|+|k|)s}([k]+[\ell])^{2p}\\
&\leq 4^{p}\sum_{(\ell.k)\in \mathbb{Z}^v \times \mathbb{Z}} |u_{\ell.k}|^2 e^{2|\ell|s}[\ell]^{2p}e^{2|k|s}[k]^{2p}\\
&=4^{p}(\|u\|^{\mathfrak{s}}_{s,p})^2.
\end{split}
\end{equation}
Then, the lemma is proved.
\end{proof}
The algebra properties of Banach space $\mathrm{H}_{s,p}$ and $\mathrm{H}^\mathfrak{s}_{s,p}$ are also our concern.
\begin{lemma} For all $p\geq s_0 >\frac{v+1}{2}$, if $h_1, h_2 \in \mathrm{H}_{s,p}(\mathbb{T}^{v+1})$, then $h_1h_2 \in \mathrm{H}_{s,p}(\mathbb{T}^{v+1})$. Also, there are $c(p) \geq c(s_0)$, such that
\begin{equation}\label{2.11}
\|h_1h_2\|_{s,p} \leq c(p) \|h_1\|_{s,p}\|h_2\|_{s,p}.
\end{equation}
If $h_1=h_1(\lambda)$ and $h_2=h_2(\lambda)$ depend in a Lipschitz way on the parameter $\lambda \in \Pi \subset \mathbb{R}$, then
\begin{equation}\label{2.12}
 \|h_1h_2\|^{Lip}_{s,p} \leq c(p) \|h_1\|^{Lip}_{s,p}\|h_2\|^{Lip}_{s,p}.
\end{equation}
\end{lemma}
\begin{proof} \eqref{2.11} is the same as Lemma \ref{estimation1}. The proof of \eqref{2.12} is standard.
\end{proof}

\begin{lemma} \label{2.121}For all $p\geq s_0 > \frac{v}{2}$, if $h_1, h_2 \in \mathrm{H}^\mathfrak{s}_{s,p}(\mathbb{T}^v \times \mathbb{T})$, then $h_1h_2 \in \mathrm{H}^\mathfrak{s}_{s,p}(\mathbb{T}^v\times \mathbb{T})$. Also, there are $c(p) \geq c(s_0)$ such that
\begin{equation}\label{2.111}
\|h_1h_2\|^\mathfrak{s}_{s,p} \leq c(p) \|h_1\|^\mathfrak{s}_{s,p}\|h_2\|^\mathfrak{s}_{s,p}.
\end{equation}
If $h_1=h_1(\lambda)$ and $h_2=h_2(\lambda)$ depend in a Lipschitz way on the parameter $\lambda \in \Pi \subset \mathbb{R}$, then
\begin{equation}\label{2.122}
\|h_1h_2\|^{\mathfrak{s},Lip}_{s,p} \leq c(p) \|h_1\|^{\mathfrak{s},Lip}_{s,p}\|h_2\|^{\mathfrak{s},Lip}_{s,p}.
\end{equation}
\end{lemma}
\begin{proof}
If $h_i \in \mathrm{H}^\mathfrak{s}_{s,p}(\mathbb{T }\times \mathbb{T}^v),i=1,2$, then $h_i=\sum_{k \in \mathbb{Z}} \hat{h}^k_i(\varphi)e^{\mathrm{i}kx}$, with
\begin{equation}
(\|h_i\|^\mathfrak{s}_{s,p})^2=\sum_{k \in \mathbb{Z}} e^{2|k|s}[j]^{2p}\|\hat{h}^k_i(\varphi)\|^2_{s,p}.
\end{equation}
Let $\gamma_{j,k}= \frac{[j-k][k]}{[j]}$, By the Schwarz inequality, we have
\begin{equation}
|\sum_k x_k|^2=\bigg|\sum_k \frac{\gamma^p_{j,k}x_k}{\gamma^p_{j,k}}\bigg|^2 \leq c^2_j \sum_k\gamma^{2p}_{j,k}|x_k|^2, \quad c^2_j=\sum_k\frac{1}{\gamma^{2p}_{j,k}},
\end{equation}
where
\begin{equation}
c^2_j=\sum_k\frac{1}{\gamma^{2p}_{j,k}} \leq \sum_k(\frac{1}{[j-k]}+\frac{1}{[k]})^{2p} \leq 4^p \sum_k\frac{1}{[k]^{2p}}=c^2 < +\infty .
\end{equation}
For the case $s=0$, we have
\begin{equation}
\begin{split}
(\|h_1h_2\|^{^\mathfrak{s}}_{s,p})^2&=\sum_{j}[j]^{2p}\|\sum_{k} \hat{h}^{j-k}_1(\varphi) \hat{h}^k_2(\varphi) \|^2_{s,p} \\
&\leq \sum_{j}[j]^{2p} (\sum_{k} \|\hat{h}^{j-k}_1(\varphi)\hat{h}^k_2(\varphi)\|_{s,p})^2 \\
&\leq c^2 \cdot c(p,v) \sum_{j}[j]^{2p} \sum_{k}\gamma^{2p}_{j,k} \|\hat{h}^{j-k}_1(\varphi) \|^2_{s,p} \|\hat{h}^k_2(\varphi)\|^2_{s,p}\\
&= c_1\sum_{j} \sum_{k}[j-k]^{2p}[k]^{2p} \|\hat{h}^{j-k}_1(\varphi) \|^2_{s,p} \|\hat{h}^k_2(\varphi)\|^2_{s,p}\\
&= c_1(\|h_1\|^{\mathfrak{s}}_{s,p})^2 (\|h_2\|^{^\mathfrak{s}}_{s,p})^2.
\end{split}
\end{equation}
The case $s>0$ is a simple variation.
\end{proof}
\subsection{Matrices with variable}\

 Let $b \in \mathbb{N}$, and consider the exponential basis ${e_i:i \in \mathbb{Z}^b}$ of $L^2(\mathbb{T}^b)$, so that $L^2(\mathbb{T}^b)$ is the vector space ${u=\sum u_i e_i}$, $\sum |u_i|^2< \infty$. Any linear operator $A:L^2(\mathbb{T}^b) \rightarrow L^2(\mathbb{T}^b) $  can be represented by the  infinite dimensional matrix
 $$(A^{i'}_i)_{i,i' \in \mathbb{Z}^b},\quad A^{i'}_i:=( Ae_{i'},e_{i})_{L^2(\mathbb{T}^b)}, \quad Au=\sum_{i,i'} A^i_{i'} u_{i'}e_i.$$
\begin{defi}\label{matrix}
  Consider an infinite dimensional matrix $\mathcal{A}(\varphi)$ of time variables, where $A(\varphi)^{i_2}_{i_1}=(\mathcal{A}e^{\mathrm{i}i_2x},e^{\mathrm{i}i_1x})$. Thus, we define an $(s,p)$-decay Banach space $\mathrm{B}_{s,p}$ as

\begin{equation}
\mathrm{B}_{s,p}:= \bigg\{\mathcal{A}:(|\mathcal{A}|_{s,p})^2 =\sum_{i\in \mathbb{Z}}e^{2|i|s}[i]^{2p} (\sup_{i_1-i_2=i}  \Big\| \mathcal{A}(\varphi) ^{i_2}_{i_1} \Big\|^2_{s,p}) < +\infty \bigg\}.
\end{equation}

 \end{defi}
 So, for parameter  dependent matrices $A:=A(\lambda),\lambda \in \Pi \subseteq \mathbb{R}$, we can also define Lipschitz norms as
 $$ |A|^{sup}_{s,p}= \sup_{\lambda \in \Pi} |A(\lambda)|^{sup}_{s,p},\quad |A|^{lip}_{s,p}=\sup_{\lambda_1 \neq \lambda_2}\frac{|A(\lambda_1)-A(\lambda_2)|_{s,p}}{|\lambda_1-\lambda_2|},$$
 $$|A|^{Lip}_{s,p}=|A|^{sup}_{s,p}+  |A|^{lip}_{s,p}.$$

We now show some properties of $(s,p)$-decay norm.
 \begin{lemma}$\mathbf{(Multiplication \ operator)}$ Let $p=\sum _i p_i(\phi) e_i \in \mathrm{H}_{s,p}$, the multiplication operator $h\rightarrow ph$ is represented by the matrix with variables $T^i_{i'}=p_{i-i'}(\phi)$ and
 $$|T|_{s,p}\leq\|p\|_{s,2p}.$$
 Moreover, if $p=p(\lambda)$ is a Lipschitz family of functions,
 $$|T|^{Lip}_{s,p}\leq \|p\|^{Lip}_{s,2p}.$$
\end{lemma}
\begin{proof}
According to  Definition \ref{matrix}, we see
\begin{equation}
\begin{split}
|T|^2_{s,p}&=\sum_{k\in \mathbb{Z}}\|p_k(\varphi)\|_{s,p}^2 e^{2|k|s}[k]^{2p}\\
&=\sum_{(\ell.k)\in \mathbb{Z}^v \times \mathbb{Z}} |p_{\ell.k}|^2 e^{2|\ell|s}[\ell]^{2p}e^{2|k|s}[k]^{2p}\\
&\leq \sum_{(\ell.k)\in \mathbb{Z}^v \times \mathbb{Z}} |p_{\ell.k}|^2 e^{2(|\ell|+|k|)s}([\ell]+[k])^{4p}\\
&=\|p\|^2_{s,2p}
\end{split}
\end{equation}
Then, the lemma is proved.
\end{proof}
\begin{definition}\label{2.221}
Given a $\mathcal{A} \in \mathrm{B}_{s,p}, h\in \mathrm{H}^{\mathfrak{s}}_{s,p}$, we say that $\mathcal{A}$ is dominated by $h$, and we write $\mathcal{A} \prec h$, if $\|A(\varphi)^{i_2}_{i_1}\|_{s,p}
 \leq \|h(\varphi)_{i_2-i_1}\|_{s,p}  $ for all $i_1,i_2 \in \mathbb{Z}$.
\end{definition}
It can see  that
\begin{equation}
|\mathcal{A}|_{s,p}=\min\{\|h\|_{s,p}:h \in \mathrm{H}_{s,p},\mathcal{A} \prec h \} \quad and \quad \exists h\in\mathrm{H}_{s,p},|\mathcal{A}|_{s,p}=\|h\|_{s,p}
\end{equation}

\begin{lemma} \label{2.2221}For $\mathcal{A}_1,\mathcal{A}_2 \in \mathrm{B}_{s,p}$, we have
\begin{equation}
\mathcal{A}_1\prec h_1,\mathcal{A}_2\prec h_2 \Rightarrow  |\mathcal{A}_1\mathcal{A}_2|_{s,p} \leq C(p)  \|h_1\|^{\mathfrak{s}}_{s,p} \| h_2\|^{\mathfrak{s}}_{s,p}
\end{equation}
\end{lemma}
\begin{proof}
 For all $i_1,i_2 \in \mathbb{Z}$, $i_1-i_2=i$, we have
\begin{equation}
\begin{split}
\|(\mathcal{A}_1\mathcal{A}_2(\varphi))^{i_2}_{i_1}\|_{s,p}&=\|\sum_{k \in \mathbb{Z}}\mathcal{A}_1(\varphi)^k_{i_1} \mathcal{A}_2(\varphi)^{i_2}_k\|_{s,p} \leq \sum_{k \in \mathbb{Z}}\|\mathcal{A}_1(\varphi)^k_{i_1}\|_{s,p} \|\mathcal{A}_2(\varphi)^{i_2}_k\|_{s,p}\\
&\leq \sum_{k \in \mathbb{Z}}\|(h_1(\varphi))_{i_1-k}\|_{s,p} \|h_2(\varphi)_{k-i_2}\|_{s,p}\\
&=\sum_{k \in \mathbb{Z}}\|h_1(\varphi)_{k}\|_{s,p} \|h_2(\varphi)_{i-k}\|_{s,p}   ,
\end{split}
\end{equation}
implying $|\mathcal{A}_1\mathcal{A}_2|_{s,p} \leq C(p) \|h_1\|^{\mathfrak{s}}_{s,p} \| h_2\|^{\mathfrak{s}}_{s,p}$,  following from the proof of  Lemma \ref{2.121} .
\end{proof}

\begin{lemma}\label{2.2}$\mathbf{(Classical \ algebra \  property)}$ For all $p\geq s_0 \geq \frac{v+1}{2}$, if $\mathcal{A},\mathcal{B} \in \mathrm{B}_{s,p}$, then $\mathcal{A}\mathcal{B} \in \mathrm{B}_{s,p}$. Also, there are $c(p) \geq c(s_0)$ such that
\begin{equation}\label{2.13}
|\mathcal{A}\mathcal{B}|_{s,p} \leq c(p) |\mathcal{A}|_{s,p}|\mathcal{B}|_{s,p}.
\end{equation}
If $A=A(\lambda)$ and $B=B(\lambda)$ depend in a Lipschitz way on the parameter $\lambda \in \Pi \subset \mathbb{R}$, then
\begin{equation}\label{2.14}
|\mathcal{A}\mathcal{B}|^{Lip}_{s,p} \leq c(p) |\mathcal{A}|^{Lip}_{s,p}|\mathcal{B}|^{Lip}_{s,p}.
\end{equation}
\end{lemma}
\begin{proof}
We can immediately deduce \eqref{2.13} from Lemma \ref{2.121} and Lemma \ref{2.2221}. The proof of \eqref{2.14} is standard.
\end{proof}
\begin{lemma}For all $p\geq s_0 \geq \frac{v+1}{2}$, if $\mathcal{A} \in \mathrm{B}_{s,p}, h\in \mathrm{H}_{s,2p}$, then $\mathcal{A}h \in \mathrm{H}_{s,p}$. Also, there are $c(p) \geq c(s_0)$ such that
\begin{equation}\label{2.15}
\|\mathcal{A}h\|_{s,p} \leq c(p)|\mathcal{A}|_{s,p}\|h\|_{s,2p}.
\end{equation}
If $A=A(\lambda)$ and $h=h(\lambda)$ depend in a Lipschitz way on the parameter $\lambda \in \Pi \subset \mathbb{R}$, then
\begin{equation}
\|\mathcal{A}h\|^{Lip}_{s,p} \leq c(p)|\mathcal{A}|^{Lip}_{s,p}\|h\|^{Lip}_{s,2p}.
\end{equation}
\end{lemma}
\begin{proof}
From Lemma \ref{2.121} and Lemma \ref{2.2221}, we can immediately get $\|Ah\|^{\mathfrak{s}}_{s,p} \leq c(p)|A|_{s,p}\|h\|^{\mathfrak{s}}_{s,p}$.
To prove \eqref{2.15}, observe that
\begin{equation}
\|Ah\|_{s,p} \leq c(p)\|Ah\|^\mathfrak{s}_{s,p} \leq c_1(p)|A|_{s,p}\|h\|^\mathfrak{s}_{s,p}\leq c_2(p)|A|_{s,p}\|h\|_{s,2p}.
\end{equation}
\end{proof}
\begin{lemma}
Let $\Phi=e^{\Psi}$, with $\psi:=\psi(\lambda)$ depending in a Lipschitz way on the parameter $\lambda \in \Pi \subset \mathbb{R}$, such that $c(p)|\Psi|^{Lip}_{s,p} \leq \frac{1}{2}$. Then $\Phi$ is invertible and, for all $p > \frac{v}{2}$,
\begin{equation}\label{2.31}
|\Phi^{-1}|^{Lip}_{s,p}\leq 2,\quad |\Phi-\mathrm{I}|^{Lip}_{s,p}| \leq C|\Psi_i|^{Lip}_{s,p},\quad |\Phi^{-1}-\mathrm{I}|^{Lip}_{s,p}| \leq C|\Psi|^{Lip}_{s,p}.
\end{equation}
Let $\Phi_i=e^{\Psi_i},i=1,2$, satisfy $c(p)|\Psi_i|^{Lip}_{s,p} \leq \frac{1}{2}$, then
\begin{equation}\label{2.33}
|\Phi_2-\Phi_1|^{Lip}_{s,p} \leq C|\Psi_2-\Psi_1|^{Lip}_{s,p}, \quad |\Phi^{-1}_2-\Phi^{-1}_1|^{Lip}_{s,p}\leq C|\Psi_2-\Psi_1|^{Lip}_{s,p}.
\end{equation}
\end{lemma}
\begin{proof}
 \eqref{2.31} are  from the Taylor's series of $e^{\Psi_i}$ and Lemma \ref{2.2}.
To prove \eqref{2.33}, we see
\begin{equation}
\begin{split}
\Phi_2-\Phi_1&=e^{\psi_2}-e^{\psi_1}=\sum^{\infty}_{n=1}\frac{1}{n!}[\Psi_2^n-\Psi_1^n]  \\
&=(\Psi_2-\Psi_1)\sum^{\infty}_{n=1}\frac{1}{n!}[\Psi_2^{n-1}+\Psi_2^{n-2}\Psi_1+\cdots+\Psi_1^{n-1}],\\
\end{split}
\end{equation}
and
\begin{equation}
\Phi^{-1}_2-\Phi^{-1}_1=\Phi^{-1}_1(\Phi_1-\Phi_2)\Phi^{-1}_2.
\end{equation}
 Then, use  \eqref{2.31}.
 \end{proof}

\subsection{Linear time-dependent operator and Hamiltonian operators }\

In this section, we give some definitions and properties of the linear time-dependent Hamiltonian systems which will be used in following section.
\begin{defi}
A time dependent linear vector field $X(t):\mathrm{H}^1_0(\mathbb{T})\rightarrow \mathrm{H}^1_0(\mathbb{T})$  is HAMILTONIAN if $X(t)=\partial_xG(t)$ for some real linear operator $G(t)$ which is self-adjoint with respect to the  $L^2$ scalar product. The vector product is generated by the quadratic Hamiltonian
\begin{equation}
H(t,h)=\frac{1}{2}(G(t)h,h)_{L^2(\mathbb{T})}=\frac{1}{2}\int_{\mathbb{T}}G(t)[h]hdx , \quad h\in\mathrm{H}^1_0(\mathbb{T}).
\end{equation}

If $G(t)=G(\omega t)$ is quasi-periodic in time, we say that the associate operator $\omega \cdot \partial_\varphi-\partial_x G(\varphi)$ is Hamiltonian.
\end{defi}

\begin{defi}
A linear operator $A:\mathrm{H}^1_0(\mathbb{T})\rightarrow \mathrm{H}^1_0(\mathbb{T})$ is SYMPLECTIC if
\begin{equation}\label{2.21}
\Omega(Au,Av)=\Omega(u,v) ,\quad \forall u,v\in\mathrm{H}^1_0(\mathbb{T}).
\end{equation}
where the symplectic $2$-form $\Omega$ is defined in \eqref{1.01}. Equivalently $A^T \partial^{-1}_xA=\partial^{-1}_x$.

If $A(\varphi),\forall\varphi \in \mathbb{T}^v$, is a family of symplectic maps we say that the operator $A$ defined by $Ah(\varphi,x)=A(\varphi)h(\varphi,x)$, acting on the functions $h:\mathbb{T}^{v+1}\rightarrow \mathbb{R}$, is symplectic.
\end{defi}
 Under a time dependent family of symplectic transformations $u=\Psi(t)v$ the linear Hamiltonian equation
 \begin{equation}
 u_t=\partial_xG(t)u \quad with \quad Hamiltonian \quad H(t,u)=\frac{1}{2}(G(t)u,u)_{L^2}
 \end{equation}
 transforms into the equation
\begin{equation}
v_t=\partial_xE(t)v,\quad E(t)=\Psi(t)^{T}G(t)\Psi(t)-\Psi(t)^{T}\partial^{-1}_x\Psi_t(t)
\end{equation}
with Hamiltonian
\begin{equation}
K(t£¬v)=\frac{1}{2}(G(t)\Psi(t)v,\Psi(t)v)_{L^2}-\frac{1}{2}(\partial^{-1}_x\Psi_t(t)v,\Psi(t)v)_{L^2}.
\end{equation}
Note that $E(T)$ is self-adjoint with respect to the $L^2$ scalar product because $\Psi^T\partial^{-1}_x\Psi_t+\Psi^T_t\partial^{-1}_x\Psi=0$. If the operators $G(t),\Psi(t)$ are quasi-periodic in time. The Hamiltonian operator $\omega\cdot \partial_{\varphi}-\partial_x G(\varphi)$ transforms into the operator $\omega\cdot \partial_{\varphi}-\partial_x E(\varphi)$, which is still Hamiltonian, according to the Definition \ref{2.21}.

 \section{The Regularization of the linearized operator}
In this section, we perform a regularization procedure, which conjugates the linearized operator $\mathcal{L}(u_n)$ defined in \eqref{3.000} to the operator $\mathfrak{L}(u_n)$ defined in \eqref{3.23}, the coefficients of the highest order spatial derivative operator are constant. The method has been used  in  \cite{Baldi1,Berti0,Berti1,Berti4,Feola1,Feola2}.
Our existence proof is based on a modified Newton iteration. The main step concerns the invertibility of the linearized operator
\begin{equation}\label{operator1}
 \mathcal{L}h= \mathcal{L}(\lambda, u, \varepsilon)h = \omega \partial_\varphi h +  a^*_5\partial^5_{x}h + a^*_4\partial^4_{x}h + a^*_3 \partial^3_xh +a^*_2\partial^2_{x}h +a^*_1\partial_{x}h+a^*_0h,
\end{equation}
obtained by linearizing \eqref{main3} at any approximate (or exact) solution $u$. The coefficients
$a_i = a_i (\varphi, x) = a_i (u, \varepsilon)(\varphi, x) $ are periodic functions of $(\varphi, x)$,  depending on $u$ and $\varepsilon$.
Then, we have
\begin{equation}
a^*_5=(1-6\partial^2_xu), \quad a^*_4=(-18\partial^3_xu,) \quad a^*_3=(10u-18\partial^4_xu), \quad
\end{equation}
\begin{equation}
 a^*_2=(20\partial_xu-6\partial^5_xu), \quad a^*_1=(20\partial_x^2u+30u^2), \quad  a^*_0=(10\partial^3_xu+60u\partial_xu).
\end{equation}
\par
In the Hamiltonian case \eqref{main1}, the linearized operator \eqref{operator1} also has the form
\begin{equation}\label{3.000}
 \mathcal{L}h=\omega \partial_\varphi h+\partial_x\{\partial^2_x [(a_2(u))\partial^2_xh]+\partial_x [a_1(u)\partial_xh]+a_0(u)h\},
\end{equation}
where
 $$a_2(u)=1+a(u)=1-6u_{xx}, \quad a_1(u)=10u,  \quad a_0(u)=10u_{xx}+30u^2.$$

 The coefficients $a_{i}$, together with their derivative $\partial_u a_{i}[h]$ with respect to $u$ in the direction $h$, satisfy the following estimates:
 \begin{lemma} For all $p \geq s_0 >\frac{v+1}{2}$,  $\|u\|_{s,p+2}\leq 1$, we have, for $i=0,1,2,$
 \begin{equation}\label{3.00}
 \|a_i(u)\|_{s,p}\leq C \| u\|_{s,p+2},
 \end{equation}
  \begin{equation}\label{3.01}
 \|\partial_{u}a_i(u)[h]\|_{s,p}\leq C \|h\|_{s,p+2}.
 \end{equation}
  Moreover, if $\lambda\mapsto u(\lambda) $ is a Lipschitz family, and satisfying $\|u\|^{Lip}_{s,p+2} \leq 1$, then, we have
 \begin{equation}\label{3.02}
 \|a_i(u)\|^{Lip}_{s,p'} \leq C \|u\|^{Lip}_{s,p'+2},
\end{equation}
\begin{equation}\label{3.03}
\|\partial_{u}a(u)[h]\|^{Lip}_{s,p'}\leq C \|h\|^{Lip}_{s,p'+2}.
 \end{equation}
\end{lemma}
\begin{proof} Notice
 $$\partial_{u}a_2(u)[h]=6h_{xx}, \quad \partial_{u}a_1(u)[h]=10h$$
 and
 $$\partial_{u}a_0(u)[h]=10h_{xx}+60uh_{xx}.$$
 Then, these estimates are straightforward.
\end{proof}

\subsection{Change of space variable}\

 We consider a $\varphi$-dependent family of space variable change of the form
\begin{equation}\label{space change}
y = x +\beta(\varphi, x),
\end{equation}
where $\beta$ is a (small) real analytic function, $2\pi$-periodic in all its arguments. The change of variables \eqref{space change} induces on the space of functions the linear operator
\begin{equation}
(\mathcal{T}h)(\varphi,x)=h(\varphi,x+\beta(\varphi, x)).
\end{equation}
The operator $\mathcal{T}$ is invertible, with inverse
\begin{equation}
(\mathcal{T}^{-1}h)(\varphi,y)=h(\varphi,y+\hat{\beta}(\varphi, y)).
\end{equation}
where $y\rightarrow y+\hat{\beta}(\varphi,y)$ is the inverse of \eqref{space change}, namely
$$x= y+\hat{\beta}(\varphi,y) \Longleftrightarrow y = x +\beta(\varphi, x).$$
\par
In the Hamiltonian case, in order to keep the Hamiltonian structure of linear operator, the operator $\mathcal{T}$ needs a slight change. The modified linear operator is
\begin{eqnarray}\label{3.10}
&&(\mathcal{A}h)(\varphi,x)=(1+\beta_x(\varphi,x))h(\varphi,x+\beta(\varphi, x)),\nonumber \\
&&(\mathcal{A}^{-1}h)(\varphi,y)=(1+\hat{\beta}_y(\varphi,y))h(\varphi,y+\hat{\beta}(\varphi, y)).\nonumber \\
\end{eqnarray}
\begin{rem}
By  \cite[remark 4.1.3]{Berti1}, the modified change of variable and its inverse \eqref{3.10} are  symplectic, for each $\varphi \in \mathbb{T}^v$. Also, both $\mathcal{A}$ and $\mathcal{A}^{-1} $ are maps from $\mathrm{H}^1_{0}$ to $\mathrm{H}^1_{0}$, for each $\varphi \in \mathbb{T}^v$.
\end{rem}

Now, we calculate the conjugate $\mathcal{A}^{-1}\mathcal{L}\mathcal{A}$ of the linearized operator $\mathcal{L}$ in \eqref{operator1}.

 The conjugate $\mathcal{A}^{-1}a\mathcal{A}$ of any multiplication operator $a:h(\varphi,x)\mapsto a(\varphi,x)h(\varphi,x)$ is the  multiplication operator $(\mathcal{T}^{-1}a)$ that maps $v(\varphi,y)\mapsto (\mathcal{T}^{-1}a)v(\varphi,y)$. The conjugate of differential operators are
 \begin{equation}
 \begin{split}
 \mathcal{A}^{-1}\omega \cdot\partial_{\varphi}\mathcal{A}=&\omega \cdot\partial_{\varphi}+\partial_y\big\{\mathcal{T}^{-1}(\omega \cdot \partial_{\varphi}\beta)\big\},\\
\mathcal{A}^{-1}\big\{\partial_{x}a\big\}\mathcal{A}=&\partial_{y}\Big\{\big(\mathcal{T}^{-1}[a(1+\beta_x)]\big)\Big\},\\
  \mathcal{A}^{-1}\Big\{\partial_{x}\big\{\partial_x (a \partial_x)\big\}\Big\}\mathcal{A}=&\partial_y\Big\{\partial_y \big((\mathcal{T}^{-1}[a(1+\beta_x)^3]) \partial_y\big)+\big(\mathcal{T}^{-1}[a_x\cdot \beta_{xx}+a\cdot\partial^3_x\beta]\big)\Big\},\\
 \mathcal{A}^{-1}\Big\{\partial_{x}\{\partial^2_x (a \partial^2_x)\}\Big\}\mathcal{A}=&\partial_y\Big\{\partial^2_y \big((\mathcal{T}^{-1}[a(1+\beta_x)^5]) \partial^2_y\big) \\
 & +\partial_y \big((\mathcal{T}^{-1}[3a_x(1+\beta_x)^2\beta_{xx}+5a(1+\beta_x)^2\partial^3_x\beta]) \partial_y \big)
\\
& + \big(\mathcal{T}^{-1}[a_{xx}\cdot \partial^3_{x}\beta+2a_x\cdot\partial^4_{x}\beta+a\cdot\partial^5_{x} \beta]\big)\Big\},
\end{split}
\end{equation}
where all the coefficients $\{\mathcal{T}^{-1}[..]\}$ are periodic functions of $(\varphi,y)$.

\begin{rem}
we give out some calculation tricks which have been used above.

$(1)$: $\mathcal{A}^{-1}\partial_xg=\partial_y\{\mathcal{T}^{-1}g\}$, since
\begin{equation*}
\begin{split}
\partial_y\{\mathcal{T}^{-1}g\}&=\partial_y g(y+\hat{\beta}(\varphi,y), \varphi )\\
& = (1+\hat{\beta}_y)\cdot \partial_xg (y+\hat{\beta}(\varphi,y), \varphi )\\
&= \mathcal{A}^{-1}\partial_xg .
\end{split}
\end{equation*}

$(2)$: $(1+\hat{\beta}_y(y,\varphi))[\mathcal{T}^{-1}(1+\beta_x(x,\varphi))]=1$.

Using $(1),(2)$, the conjugate of differential operators $\partial_{x}a$ and $\omega \cdot\partial_{\varphi}$ is obvious.
\end{rem}
\begin{rem}
  The calculation of the conjugate of differential operators $\partial_{x}\{\partial^2_x (a \partial^2_x)\}$ and $\partial_{x}\{\partial^2_x (a \partial^2_x)\}$  are slightly more finicky. Let's take  $\big\{\partial_{x}a\big\}$ as an example, we see
  \begin{equation}
  \begin{split}
   \mathcal{A}^{-1}\Big\{\partial_{x}\big\{\partial_x (a \partial_x)\big\}\mathcal{A}h(y,\varphi)\Big\}=&\partial_{y}\Big\{\mathcal{T}^{-1}\big\{\partial_x (a \partial_x)\big\}\mathcal{A}h(y,\varphi)\Big\}\\
   =&\partial_{y}\Big\{\mathcal{T}^{-1}\big\{\partial_x \big(a \partial_x[(1+\beta_x)h(x+\beta,\varphi)]\big)\big\}\Big\}\\
   =&\partial_{y}\Big\{\mathcal{T}^{-1}\big\{\partial_x \big(a(1+\beta_x)^2\partial_yh+a\beta_{xx}h \big)\big\}\Big\}\\
   =&\partial_{y}\Big\{\mathcal{T}^{-1}\big\{a(1+\beta_x)^3\partial^2_yh+[a_x(1+\beta)^2+3a(1+\beta_x)\beta_{xx}]\partial_yh\\
   &+(a_x\beta_{xx}+a\partial^3_x\beta) h\big)\big\}\Big\}\\
   =&\partial_y\Big\{\partial_y \big((\mathcal{T}^{-1}[a(1+\beta_x)^3]) \partial_y\big)+\big(\mathcal{T}^{-1}[a_x\cdot \beta_{xx}+a\cdot\partial^3_x\beta]\big)\Big\},
  \end{split}
  \end{equation}
  since $\mathcal{T}^{-1}[a_x(1+\beta)^2+3a(1+\beta_x)\beta_{xx}]=\partial_y \big\{\mathcal{T}^{-1}[a(1+\beta_x)^3]\big\}$.
\end{rem}

Now, we get
\begin{equation}\label{L1}
\mathcal{L}_1:=\mathcal{A}^{-1}\mathcal{L}\mathcal{A}=\omega \partial_\varphi h+\partial_y\big\{\partial^2_y [b_2(u)\partial^2_yh]+\partial_y[b_1(u)\partial_yh]+b_0(u)h\big\},
\end{equation}
where
\begin{equation}\label{3.101}
\begin{split}
 b_2 =&\mathcal{T}^{-1}[a_2(1+\beta_x)^5],\\
 b_1 =&\mathcal{T}^{-1}[a_1(1+\beta_x)^3+3(1+\beta_x)^2\beta_{xx}\cdot \partial_xa_2+5a_2(1+\beta_x)^2\partial^3_x\beta] , \\
 b_0 =&\mathcal{T}^{-1}[\partial_xa_1\cdot \partial^2_x\beta+a_1\cdot\partial^3_x\beta+\partial^2_xa_2\cdot \partial^3_{x}\beta+2\partial_x a_2\cdot\partial^4_{x}\beta+a_2\cdot\partial^5_{x} \beta\\
 &+\omega \cdot \partial_{\varphi}\beta+a_0(1+\beta_x)]. \\
 \end{split}
\end{equation}
For convenience, we set $b_i=\mathcal{T}^{-1}(u)b^*_i, \ i=0,1$. Now, we look for $\beta(\varphi,x)$ such that  the coefficient $b_2(\varphi,y)$ dose not  depend on $y$, namely
\begin{equation}\label{change1}
 b_2(\varphi,y)=\mathcal{T}^{-1}[(1+a)(1+\beta_x)^5]=1+b(\varphi).
\end{equation}
Since $\mathcal{T}$ only make changes on the space variable,$\mathcal{T}b = b$ for every function
$b(\varphi)$ that is independent on $y$. Hence \eqref{change1} is equivalent to
\begin{equation}
(1+a)(1+\beta_x)^5=1+b(\varphi),
\end{equation}
namely
\begin{equation}\label{solution1}
\beta_x(\varphi,x)=p_0,  \quad p_0(\varphi,x)=(1+b(\varphi))^{\frac{1}{5}}(1+a(\varphi,x))^{-\frac{1}{5}}-1.
\end{equation}
 The equation \eqref{solution1} has a solution $\beta$,  periodic in $x$, if and only if  $\int_{\mathbb{T}}p_0(\varphi,x)dx=0$. This condition uniquely determines
 \begin{equation}
 1+b(\varphi)=(\int_\mathbb{T}(1+a(\varphi,x))^{-\frac{1}{5}}dx)^{-5}.
 \end{equation}
Then, we have a solution (with zero average) of \eqref{solution1}
\begin{equation}
\beta(\varphi,x)=(\partial^{-1}_x p_0)(\varphi,x),
\end{equation}
where $\partial^{-1}_x$ is defined by linearity as
\begin{equation}
 \partial^{-1}_x e^{\mathbf{i}kx}=\frac{e^{\mathbf{i}kx}}{\mathbf{i}k},   \forall k \in \mathbb{Z}\backslash \{0\}, \quad \partial^{-1}_x1=0.
\end{equation}
In other words, $\partial^{-1}_x$ is the primitive of $h$ with zero average in $x$
.  Thus we obtain the operator $\mathcal{L}_1$ in \eqref{L1}, that $b_2(\varphi,x)=1+b(\varphi)$.

 Set $s_0 > \frac{v}{2}$, $\|u\|^{Lip}_{s,p+2s_0+7} \ll\frac{1}{100}$.  We have the following estimates:

\textbf{1. Estimates of $b(\varphi)$ }

We prove $b(\varphi)$ satisfies the following estimates:
\begin{equation}\label{3.11}
\|b\|_{s,p} \leq C \|a(\varphi,x)\|_{s,p} \leq C \|u(\varphi,x)\|_{s,p+2},
\end{equation}
\begin{equation}\label{3.12}
\|b\|^{Lip}_{s,p} \leq C \|a(\varphi,x)\|^{Lip}_{s,p} \leq C \|u(\varphi,x)\|^{Lip}_{s,p+2},
\end{equation}
\begin{equation}\label{3.13}
\|\partial_ub(u)[h]\|^{Lip}_{s,p} \leq C(\|h\|^{Lip}_{s,p+2}\|u\|^{Lip}_{s,p+2}+\|h\|^{Lip}_{s,p+2}).
\end{equation}

Proof of \eqref{3.11} and \eqref{3.12}:  Write $b$ as
 \begin{equation*}
b=\psi(M[g(a)-g(0)])-\psi(0),
\end{equation*}
where
\begin{equation*}
\psi(t)=(1+t)^{-5}, \quad Mh=\int_{\mathbb{T}}h dx \quad  g(t)=(1+t)^{-\frac{1}{5}}.
\end{equation*}
If u  is small enough,  we have
\begin{equation}\label{3.14}
\|b(u)\|_{s,p} \leq C\|M[g(a)-g(0)\|_{s,p} \leq C\|g(a)-g(0)\|_{s,p} \leq C\|a\|_{s,p}.
\end{equation}
Since $u$ is small enough, $\psi(t)$ and $g(t)$ can be well defined by its power series expansion, i.e. $g(t)=1-\frac{1}{5}t+\frac{3}{25}t^2+\cdots$. Hence we have
\begin{equation}
\|g(u)-1\|^{Lip}_{s,p}=\|-\frac{1}{5}u+\frac{3}{25}u^2+\cdots\|^{Lip}_{s,p}\leq C\|u\|^{Lip}_{s,p}
\end{equation}
The first and last inequality of \eqref{3.14} can be proved in such way. The second inequality is a direct result of $\|Mg\|_{s,p} \leq C\|g\|_{s,p}$.

Proof of \eqref{3.13}:  The derivative of $c$ with respect to $u$ in the direction $h$ is
\begin{equation}
\partial_ub(u)[h]=\psi'(M[g(a)-g(0)])M(g'(a)\partial_ua(u)[h] ),
\end{equation}
where
\begin{equation}
\psi'=-5(1+t)^{-6},  \quad  g'=-\frac{1}{5}(1+t)^{-\frac{6}{5}}.
\end{equation}
Using the same method as \eqref{3.14}, by \eqref{3.00} and \eqref{3.01}, we can get \eqref{3.13}.

\textbf{2. Estimates of $\beta(\varphi,x)$ }

Considering Definition \ref{solution1} of $p_0$, suppose $\zeta(t)=(-1+t)^{\frac{1}{5}}$, we see
\begin{equation}
p_0=g(a)\zeta(b)-1.
\end{equation}
Using the same way as \eqref{3.11}, we can get
\begin{equation}
|\beta(\varphi,x)|_{s,p} \lessdot \|\beta(\varphi,x)\|_{s,p+s_0} \lessdot \|p_0(\varphi,x)\|_{s,p+s_0} \lessdot \|u\|_{s,p+s_0+2},
\end{equation}
and
\begin{equation}\label{3.15}
|\beta(\varphi,x)|^{Lip}_{s,p} \lessdot \|u\|^{Lip}_{s,p+s_0+2}.
\end{equation}
 The derivative of $p_0$ with respect to $u$ in the direction $h$ is
\begin{equation}\label{3.16}
\partial_{u}p_0[h]=g(a)\big(\zeta'(b)\partial_ub(u)[h]\big)+\big(g'(a)\partial_ua(u)[h]\big)\zeta(b).
\end{equation}
Use the same way as \eqref{3.14}, the bounds \eqref{3.02}, \eqref{3.03} and \eqref{3.14} imply
\begin{equation}\label{3.17}
\|\partial_{u}\beta[h]\|^{Lip}_{s,p}\lessdot \|\partial_{u}\rho_0[h]\|^{Lip}_{s,p}\lessdot(1+\|u\|^{Lip}_{s,p+2})\|h\|^{Lip}_{s,p+2}.
\end{equation}
The inverse function $y\rightarrow y+\hat{\beta}(\varphi,y)$ is also under our consideration. By Lemma \ref{estimation4}, one gets
\begin{equation}
|\hat{\beta}(\varphi,y)|_{\frac{99s}{100},p} \lessdot |\beta(\varphi,x)|_{s,p} \lessdot \|u\|_{s,p+s_0+2} \leq \frac{1}{100},
\end{equation}
and
\begin{equation}\label{3.171}
|\hat{\beta}(\varphi,y)|^{Lip}_{\frac{99s}{100},p} \lessdot |\beta(\varphi,x)|^{Lip}_{s,p} \lessdot \|u\|^{Lip}_{s,p+s_0+3} \leq \frac{1}{100}.
\end{equation}
Writing explicitly the dependence on $u$, we have $\hat{\beta}(\varphi,y;u)+\beta(\varphi,\hat{\beta}(\varphi,y,u);u)=0$. Differentiating this equality with respect to $u$ in the direction $h$ gives
\begin{equation}
(\partial_u\hat{\beta})[h]=-\mathcal{T}^{-1}(\frac{\partial_u\beta[h]}{1+\beta_x}).
\end{equation}
Applying lemma \ref{estimation2} and Lemma \ref{estimation34} to cope with $\mathcal{T}^{-1}$ , the bounds \eqref{3.15}, \eqref{3.17} and \eqref{3.171} imply
\begin{equation}\label{3.18}
\|(\partial_u\hat{\beta})[h]\|^{Lip}_{\frac{99s}{100},p} \lessdot (1+\|u\|^{Lip}_{s,p+2s_0+3})\|h\|^{Lip}_{s,p+2s_0+3}.
\end{equation}

\textbf{3. Estimates of $ \mathcal{T}$ and $\mathcal{T}^{-1}$ }

By Lemma \ref{estimation2},  Lemma \ref{estimation34} and Lemma \ref{estimation4},  we can get the following estimation:
\begin{equation}
\|\mathcal{T}(u)g\| _{\frac{100s}{101},p} \lessdot \|g\|_{s,p+2s_0},
\end{equation}
\begin{equation}
\|\mathcal{T}(u)g\|^{Lip}_{\frac{100s}{101},p} \lessdot \|g\|^{Lip}_{s,p+2s_0+1},
\end{equation}
\begin{equation}\label{3.19}
\|\mathcal{T}(u)^{-1}g\|_{\frac{99s}{100},p} \lessdot \|g\|_{s,p+2s_0},
\end{equation}
\begin{equation}\label{3.191}
\|\mathcal{T}(u)^{-1}g\|^{Lip}_{\frac{99s}{100},p} \lessdot \|g\|^{Lip}_{s,p+2s_0+1}.
\end{equation}
Since $\mathcal{T}^{-1}(u)g=g(\varphi,y+\hat{\beta}(\varphi,y))$, the derivative of $\mathcal{T}^{-1}(u)g$ in the direction $h$ is the product $\partial_u(\mathcal{T}^{-1}(u)g)=(\mathcal{T}^{-1}g_x)(\partial_u\hat{\beta})[h]$. Applying Lemma \ref{estimation3}, the bounds \eqref{3.19},\eqref{3.191} and \eqref{3.18} imply
\begin{equation}\label{3.192}
\|\partial_u(\mathcal{T}^{-1}(u)g)\|^{Lip}_{\frac{99s}{100},p} \lessdot \|g\|^{Lip}_{s,p+2s_0+4}\|h\|^{Lip}_{s,p+2s_0+3}(1+\|u\|^{Lip}_{s,p+2s_0+3}).
\end{equation}

\textbf{4. Estimates of the coefficients $b_i$ }

Consider the coefficients $b^*_1$, $b^*_0$,  which are given in \eqref{3.101}. We have
\begin{equation}\label{3.1921}
\|b^*_1\|^{Lip}_{s,p} \lessdot \|u\|^{Lip}_{s,p+4},
\end{equation}
\begin{equation}\label{3.1922}
\|b^*_0\|^{Lip}_{s,p} \lessdot \|u\|^{Lip}_{s,p+6}.
\end{equation}

Applying \eqref{3.19},\eqref{3.191} to \eqref{3.1921} and \eqref{3.1922}, for $i=0,1$, we see
\begin{equation}\label{3.193}
\|b_i\|_{\frac{99s}{100},p} \lessdot \|u\|_{s,p+2s_0+6},
\end{equation}
and
\begin{equation}\label{3.1941}
\|b_i\|^{Lip}_{\frac{99s}{100},p} \lessdot \|u\|^{Lip}_{s,p+2s_0+7}.
\end{equation}
Now,  we estimate the derivative of $b_1$ with respect to $u$. Write $b_1$ as $\mathcal{T}^{-1}(u)b^*_1$, where $$b^*_1=a_1(1+\beta_x)^3+3(1+\beta_x)^2\partial^2_x\beta\cdot \partial_xa_2+5a_2(1+\beta_x)^2\partial^3_x\beta.$$
The bounds \eqref{3.02}, \eqref{3.03} and \eqref{3.17} imply
\begin{equation}\label{3.195}
\|\partial_ub^*_1(u)[h]\|^{Lip}_{s,p}\lessdot \|h\|^{Lip}_{s,p+4}(1+\|u\|^{Lip}_{s,p+4}).
\end{equation}
The derivative of $b_1$ in the direction $h$ is
\begin{equation}\label{3.198}
\partial_ub_1(u)[h]=\partial_u(\mathcal{T}(u)^{-1}b^*_1(u))[h]=(\partial_u\mathcal{T}(u)^{-1})(b^*_1(u))[h]+
\mathcal{T}(u)^{-1}(\partial_ub^*_1(u))[h].
\end{equation}
Then, \eqref{3.19}, \eqref{3.191}, \eqref{3.192}, \eqref{3.195} and \eqref{3.198} imply
\begin{equation}\label{3.196}
\|(\partial_u\mathcal{T}(u)^{-1})(b^*_1(u))[h]\|^{Lip}_{\frac{99s}{100},p} \lessdot \|u\|^{Lip}_{s,p+2s_0+6}\|h\|^{Lip}_{s,p+2s_0+3}(1+\|u\|^{Lip}_{s,p+2s_0+3})
\end{equation}
and
\begin{equation}\label{3.197}
\|\mathcal{T}(u)^{-1}(\partial_ub^*_1(u))[h]\|^{Lip}_{\frac{99s}{100},p} \lessdot \|h\|^{Lip}_{s,p+2s_0+5}(1+\|u\|^{Lip}_{s,p+2s_0+3})(1+\|u\|^{Lip}_{s,p+2s_0+5}).
\end{equation}
Ultimately, \eqref{3.195}, \eqref{3.196} and \eqref{3.197} imply that
\begin{equation}\label{3.198}
\|\partial_ub_1(u)[h]\|^{Lip}_{\frac{99s}{100},p} \lessdot \|h\|^{Lip}_{s,p+2s_0+5}(1+\|u\|^{Lip}_{s,p+2s_0+6}).
\end{equation}
By the same way as $b_1$, we can get
\begin{equation}\label{3.199}
\|\partial_ub_0(u)[h]\|^{Lip}_{\frac{99s}{100},p} \lessdot \|h\|^{Lip}_{s,p+2s_0+7}(1+\|u\|^{Lip}_{s,p+2s_0+8}).
\end{equation}

\subsection{Time reparametrization}\

In this section, we will make constant the coefficient of the highest order spatial derivative operator of $\mathcal{L}_1$, by a quasi-periodic reparametrization of time. The change of variables has the form

\begin{equation}
 \varphi \mapsto \varphi+\omega\alpha(\varphi), \quad  \varphi \in \mathbb{T}^v,
\end{equation}
where $\alpha$ is a (small) real analytic function, $2\pi$-periodic in all its arguments.  The induced linear operator on the space of functions is
\begin{equation}\label{a}
(\mathcal{B}h)(\varphi,y)=h(\varphi+\omega\alpha(\varphi)),
\end{equation}
whose inverse is
\begin{equation}\label{a1}
 (\mathcal{B}^{-1}h)(\theta,y)=h(\theta+\omega\hat{\alpha}(\theta)),
\end{equation}
where $\varphi=\theta+\omega\hat{\alpha}(\theta)$ is the inverse of  $\theta=\varphi+\omega\alpha(\varphi)$.
Then, the time derivative operator becomes

\begin{equation}
\mathcal{B}^{-1} \omega \cdot\partial_{\varphi}\mathcal{B}=\xi(\theta) \omega \cdot\partial_{\theta}, \quad \xi(\theta)=
\mathcal{B}^{-1}(1+\omega\partial_{\varphi}\alpha(\varphi)).
\end{equation}
The spatial derivative operator dose not have any change. Thus, see
\begin{equation}
\mathcal{B}^{-1}\mathcal{L}_1\mathcal{B}=\xi(\theta) \omega \cdot\partial_{\theta}+\partial_y\{\partial^2_y ([\mathcal({B} ^{-1}b_2(u))\partial^2_yh])+\partial_y [(\mathcal{B}^{-1}b_1(u))\partial_yh]+[\mathcal{B}^{-1}b_0(u)]h\}.
\end{equation}
We look for $\alpha$ such that the coefficients of the highest order derivatives are proportional, namely
\begin{equation}
[\mathcal{B}^{-1}(1+b)](\theta)=m \xi(\theta)=m [\mathcal{B}^{-1}(1+\omega \cdot\partial_\varphi \alpha)](\theta)
\end{equation}
for some constant $m\in \mathbb{R}$. This is equivalent to require that
\begin{equation}\label{3.20}
1+b(\varphi)=m(1+\omega \cdot\partial_\varphi \alpha(\varphi)).
\end{equation}
Integrating on $\mathbb{T}^v$ determines the value of the constant $m$,
\begin{equation}\label{3.21}
m=\int_{\mathbb{T}^v}(1+b(\varphi))d\varphi.
\end{equation}
We can find the unique solution of \eqref{3.20} with zero average
\begin{equation}\label{3.22}
a(\varphi)=\frac{1}{m}(\omega\cdot\partial_{\varphi})^{-1}(1+b-m)(\varphi),
\end{equation}
where $(\omega\cdot\partial_{\varphi})^{-1}$ is defined by linearity
\begin{equation}
(\omega\cdot\partial_{\varphi})^{-1} e^{\mathrm{i}\ell\varphi}=\frac{e^{\mathrm{i}\ell\varphi}}{\mathrm{i}\omega\cdot \ell},\ell\neq 0, \quad (\omega\cdot\partial_{\varphi})^{-1}1=0.
\end{equation}
With this choice of $\alpha$,  we have
\begin{equation}\label{3.23}
\mathcal{B}^{-1}\mathcal{L}_1\mathcal{B}=\xi(\theta)\mathfrak{L},\quad \mathfrak{L}= \omega \cdot\partial_{\theta}+m \partial^5_{y}+
\partial_{y}\{\partial_y[ (c_1(\theta,y)\partial_y)]+c_0(\theta,y)\},
\end{equation}
where

\begin{equation}\label{3.24}
c_i=\frac{\mathcal{B}^{-1}b_i}{\xi},\quad i=0,1.
\end{equation}

Suppose  $\|u\|^{Lip}_{s,p+4s_0+\tau_0+9} \ll \frac{1}{100}$, we have these estimation below:

\textbf{1. Estimates of $m$ }

The coefficient $m$, defined in\eqref{3.21}, satisfies the following estimates:
\begin{equation} \label{3.25}
|m-1| \leq C \|u\|_{s,p+2}, \quad  \quad |m-1|^{Lip} \leq C\|u\|^{Lip}_{s,p+2},
\end{equation}
\begin{equation}\label{3.26}
|\partial_um(u)[h]| ^{Lip}\leq C\|h\|^{Lip}_{s,p+2}(1+\|u\|^{Lip}_{s,p+2}).
\end{equation}
Using \eqref{3.11}, \eqref{3.12}, \eqref{3.21},
\begin{equation}
|m-1|=\int_{\mathbb{T}^v}|b(\varphi)|d\varphi \leq \|b\|^{Lip}_{s,p} \leq C\|u\|^{Lip}_{s,p+2}
\end{equation}
Similarly we get the Lipschitz part of \eqref{3.25}.  The estimates \eqref{3.26} follows by \eqref{3.13}, since
\begin{equation}
|\partial_um(u)[h]| ^{Lip}\leq \int_{\mathbb{T}^v}|\partial_ub(u)[h]|d\varphi \leq C \|\partial_ub(u)[h]\|^{Lip}_{s,p}.
\end{equation}
\textbf{2. Estimates of $\alpha$ }

 The function $\alpha(\varphi)$, defined in \eqref{3.22}, satisfies
\begin{equation}\label{3.261}
|\alpha|_{s,p} \leq C \alpha^{-1}_0\|u\|_{s,p+s_0+\tau_0+2}.
\end{equation}
\begin{equation}\label{3.262}
|\alpha|^{Lip}_{s,p} \leq C \alpha^{-1}_0\|u\|^{Lip}_{s,p+s_0+\tau_0+2}.
\end{equation}
Remember that $\omega =\lambda \overline{\omega}$, and $ |\overline{\omega}\cdot \ell |\geq \frac{\alpha_0} {|\ell|^\tau_0} $, $ \forall\ell \in \mathbb{Z}^v\backslash \{0\}$. By \eqref{3.11}, \eqref{3.25} and \eqref{6.03}, we see
\begin{equation}
|\alpha|_{s,p} \lessdot \|\alpha\|_{s,p+s_0} \leq C \alpha^{-1}_0\|b(\varphi)+(1-m)\|_{s,p+s_0+\tau_0} \leq C \alpha^{-1}_0\|u\|_{s,p+s_0+\tau_0+2}.
\end{equation}
Providing \eqref{3.261}. Then \eqref{3.262} holds similarly using \eqref{3.12} and $(\omega \cdot \partial_\varphi)^{-1}=\lambda^{-1}(\bar{\omega} \cdot \partial_\varphi)$.
 Differentiating formula \eqref{3.22} with respect to $u$ in the direction $h$ gives
\begin{equation}
\partial_u\alpha(u)[h]=(\lambda\bar{\omega}\cdot\partial_{\varphi})^{-1}(\frac{\partial_ub(u)[h]m-(b(\varphi)+1)\partial_um(u)[h]}{m^2}).
\end{equation}
Then, \eqref{3.12}, \eqref{3.25}, and \eqref{3.26} imply that
\begin{equation}\label{3.28}
\|\partial_u\alpha(u)[h]\|^{Lip}_{s,p}\leq C (\|h\|^{Lip}_{s,p+\tau_0+2}\|u\|^{Lip}_{s,p+\tau_0+2}+\|h\|^{Lip}_{s,p+\tau_0+2})
\end{equation}
For the inverse change of variable \eqref{a1}, by Lemma \ref{estimation4}, we have the following estimates:
\begin{equation}\label{3.29}
|\hat{\alpha}|_{\frac{99s}{100},p} \lessdot |\alpha|_{s,p} \lessdot \|u\|_{s,p+s_0+\tau_0+2}\leq \frac{1}{100},
\end{equation}
\begin{equation}\label{3.291}
|\hat{\alpha}|^{Lip}_{\frac{99s}{100},p} \lessdot  |\alpha|^{Lip}_{s,p+1} \lessdot \|u\|^{Lip}_{s,p+s_0+\tau_0+3} \leq \frac{1}{100}.
\end{equation}
Writing explicitly the dependence on $u$, we have $\hat{\alpha}(\theta;u)+\alpha(\theta+\hat{\alpha}(\theta;u);u)=0$.  Differentiating the equality with respect to $u$ in $h$ gives
\begin{equation}
\partial_u\hat{\alpha}[h]=-\mathcal{B}^{-1}(\frac{\partial_u\alpha(u)[h]}{1+\omega \cdot\partial_\varphi \alpha}).
\end{equation}
Using Lemma \ref{estimation34} to cope with $\mathcal{B}^{-1}$,  \eqref{3.261},\eqref{3.262} and \eqref{3.28} imply
\begin{equation}\label{a3}
\|\partial_u\hat{\alpha}[h]\|^{Lip}_{\frac{99s}{100},p} \lessdot \|u\|^{Lip}_{s,p+\tau_0+2s_0+4}\|h\|^{Lip}_{s,p+\tau_0+2s_0+3}+\|h\|^{Lip}_{s,p+2s_0+\tau_0+3}
\end{equation}

\textbf{3. Estimates of $\mathcal{B}$ and $\mathcal{B}^{-1}$ }

By Lemma \ref{estimation2},  Lemma \ref{estimation34} and Lemma \ref{estimation4}, the transformations $\mathcal{B}(u)$ and $\mathcal{B}^{-1}(u)$, defined in \eqref{a}, satisfy the following estimation:
\begin{equation}\label{3.292}
\|\mathcal{B}(u)g\|_{\frac{100s}{101},p} \lessdot \|g\|_{s,p+2s_0},
\end{equation}
\begin{equation}\label{3.293}
\|\mathcal{B}(u)g\|^{Lip}_{\frac{100s}{101},p} \lessdot \|g\|^{Lip}_{s,p+2s_0+1},
\end{equation}
\begin{equation}\label{3.294}
\|\mathcal{B}^{-1}(u)g\|_{\frac{99s}{100},p} \lessdot \|g\|_{s,p+2s_0},
\end{equation}
\begin{equation}\label{3.295}
\|\mathcal{B}^{-1}(u)g\|^{Lip}_{\frac{99s}{100},p} \lessdot \|g\|^{Lip}_{s,p+2s_0+1}.
\end{equation}
Differentiating $\mathcal{B}^{-1}(u)g$ with respect to $u$ in the direction $h$ gives
\begin{equation}\label{3.2960}
\partial_u(\mathcal{B}^{-1}(u)g[h])=\mathcal{B}^{-1}(u)(\omega \cdot\partial_{\varphi}g)\cdot \partial_u\hat{\alpha}[h]
\end{equation}
Then, the bounds \eqref{a3} and \eqref{3.295} imply
\begin{equation}\label{3.296}
\|\partial_u(\mathcal{B}^{-1}(u)g[h]\|^{Lip}_{\frac{99s}{100},p} \lessdot \|g\|^{Lip}_{s,p+2s_0+2}\|h\|^{Lip}_{s,p+2s_0+\tau_0+3}(1+\|u\|^{Lip}_{s,p+2s_0+\tau_0+4}).
\end{equation}

\textbf{4. Estimates of $\xi(\theta)$ }

The function $\xi$  is defined as $\xi(\theta)=\mathcal{B}^{-1}(1+\omega \cdot\partial_\varphi \alpha)$. Obviously, $\xi(\theta)-1=\mathcal{B}^{-1}( \omega \cdot\partial_\varphi \alpha)$. Then, the bounds \eqref{3.294} and \eqref{3.295} imply
\begin{equation}\label{3.297}
\|\xi-1\|_{\frac{99s}{100},p} \lessdot \|u\|_{s,p+2s_0+2}
\end{equation}
and
\begin{equation}\label{3.298}
\|\xi-1\|^{Lip}_{\frac{99s}{100},p} \lessdot \|u\|^{Lip}_{s,p+2s_0+3}.
\end{equation}
Differentiating $\xi(\theta)$ with respect to $u$ in the direction $h$ gives
\begin{equation}
\partial_u\xi(u)[h]=\partial_u\mathcal{B}(u)^{-1}(\omega \cdot\partial_\varphi \alpha)[h]+\mathcal{B}(u)^{-1}(\omega \cdot\partial_\varphi( \partial_u\alpha[h])).
\end{equation}
By  \eqref{a3}, \eqref{3.296} and \eqref{3.2960}, we get
\begin{equation}\label{3.299}
\|\partial_u\mathcal{B}(u)^{-1}(\omega \cdot\partial_\varphi \alpha)[h]\|^{Lip}_{\frac{99s}{100},p}\lessdot \|u\|^{Lip}_{s,p+2s_0+4}\|h\|^{Lip}_{s,p+\tau_0+2s_0+3}(1+\|u\|^{Lip}_{s,p+2s_0+\tau_0+4}).
\end{equation}
Using \eqref{3.295} and \eqref{3.28}, we see
\begin{equation}\label{3.2991}
\|\mathcal{B}(u)^{-1}(\omega \cdot\partial_\varphi( \partial_u\alpha[h]))\|^{Lip}_{\frac{99s}{100},p}\lessdot \|h\|^{Lip}_{s,p+2s_0+3}(1+\|u\|^{Lip}_{s,p+2s_0+\tau_0+3}).
\end{equation}
Finally, \eqref{3.299} and \eqref{3.2991} imply
\begin{equation}\label{3.29911}
\|\partial_u\xi(u)[h]\|^{Lip}_{\frac{99s}{100},p} \lessdot \|h\|^{Lip}_{s,p+2s_0+\tau_0+3}(1+\|u\|^{Lip}_{s,p+2s_0+\tau_0+4}).
\end{equation}

\textbf{5. Estimates of the coefficients $c_i$ }

The coefficients $c_i$ defined in \eqref{3.24}, for $i=0,1$, satisfy the following estimates:
\begin{equation}\label{3.2992}
\|c_i\|_{\frac{99s}{101},p} \lessdot \|u\|_{s,p+4s_0+\tau_0+6},
\end{equation}\label{3.2993}
\begin{equation}
\|c_i\|^{Lip}_{\frac{99s}{101},p} \lessdot \|u\|^{Lip}_{s,p+4s_0+\tau_0+8}.
\end{equation}
Differentiating $c_i$ with respect to $u$ in the direction $h$ gives
\begin{equation}
\partial_uc_i(u)[h]=\frac{1}{\xi}\partial_u[\mathcal{B}^{-1}b_i]-\frac{1}{\xi^2} \partial_u\xi(u)[h](\mathcal{B}^{-1}b_i).
\end{equation}
Now, we can obtain
\begin{equation}\label{3.2994}
\|\partial_uc_i(u)[h]\|^{Lip}_{\frac{99s}{101},p} \lessdot  \|h\|^{Lip}_{s,p+4s_0+\tau_0+9}(1+\|u\|^{Lip}_{s,p+4s_0+\tau_0+9}).
\end{equation}
The definition of $c_i$ \eqref{3.24}, \eqref{3.193}, \eqref{3.294}, \eqref{3.297} imply  \eqref{3.2992}. Similarly, \eqref{3.1941} \eqref{3.295}, \eqref{3.298} imply \eqref{3.2993}. Finally, \eqref{3.2994} follows from \eqref{3.2993},  \eqref{3.198}, \eqref{3.199}, \eqref{3.295}, \eqref{3.29911} and \eqref{3.28}.

\subsection{Estimates on $\mathfrak{L}$}\

Recall the procedure performed in the previous subsection, we have conjugated the operator $\mathcal{L}$ to $\mathfrak{L}$, that is
\begin{equation}\label{3.30}
\mathfrak{L}=\mathcal{U}^{-1}_1\mathcal{L}\mathcal{U}_2,\quad\mathcal{U}_1=\mathcal{A}\mathcal{B}\xi
,\quad \mathcal{U}_2=\mathcal{A}\mathcal{B}.
\end{equation}
In the following lemma, we summarize the estimates for the linear operator $\mathfrak{L}$ and $\mathcal{U}_1$, $\mathcal{U}_2$, also define constants
\begin{equation}\label{3.302}
p= 2s_0+5,\quad \eta=4s_0+\tau_0+9, \quad k_1=\frac{99}{101},\quad k_2=\frac{10000}{10201}.
\end{equation}
\begin{lemma}\label{3.301}
There exists $ 0 <\varepsilon  \ll \frac{1}{100}$, such that  for all $u(\lambda)$, $h(\lambda)$ are Lipshitz-families, satisfying
\begin{equation}
\|u\|^{Lip}_{s.p+\eta} \leq \varepsilon.
\end{equation}
$\mathbf{(1)}:$ Consider the transformation $\mathcal{U}_i, i=1,2$ ,defined in \eqref{3.30},
we have
\begin{equation}\label{3.31}
\|\mathcal{U}_ih\|_{k_2\hat{s},p'}\lessdot \|h\|_{\hat{s},p'+\eta},
\end{equation}
\begin{equation}\label{3.32}
\|\mathcal{U}^{-1}_ih\|_{k_1\hat{s},p'}\lessdot \|h\|_{\hat{s},p'+\eta}.
\end{equation}
$\mathbf{(2)}:$ The constant coefficient $m$ ,defined in \eqref{3.21}, satisfies
\begin{equation} \label{3.33}
|m-1| \leq C \|u\|_{s,2} \quad  \quad |m-1|^{Lip} \leq C\|u\|^{Lip(\gamma)}_{s,2}
\end{equation}
\begin{equation}\label{3.34}
|\partial_um(u)[h]| ^{Lip}\leq C\|h\|^{Lip}_{s,2}(1+\|u\|^{Lip}_{s,2})
\end{equation}
$\mathbf{(3)}:$ The variable coefficient $c_i$, defined in \eqref{3.24}, satisfies
\begin{equation}\label{3.341}
\|c_i\|_{k_1s,p} \lessdot \|u\|_{s,p+\eta},
\end{equation}\label{3.35}
\begin{equation}
\|c_i\|^{Lip}_{k_1s,p} \lessdot \|u\|^{Lip}_{s,p+\eta},
\end{equation}
\begin{equation}\label{3.36}
\|\partial_uc_i(u)[h]\|^{Lip}_{k_1,p} \lessdot  \|h\|^{Lip}_{s,\eta}(1+\|u\|^{Lip}_{p+\eta}).
\end{equation}
\end{lemma}
\begin{proof}
The detail of these estimates can be found in the previous subsection, we just give a summary here.
\end{proof}
\begin{rem}
The $p'$  in  \eqref{3.31} and \eqref{3.32} is any integer greater than $s_0$, $\hat{s}$ is any positive number smaller than $s$.
\end{rem}

\section{KAM Step}
\subsection{$\varepsilon_m$ approximate unbounded reducibility}\

 In this section, we make a reduction to eliminate the unbounded perturbation of linear operator $\mathfrak{L}$ obtained in \eqref{3.23}. The goal is to conjugate it to a diagonal operator $\mathcal{J}$ plus a sufficient small unbounded remainder $\mathcal{R}$. Before we apply the reducibility scheme, we will make some definition, recall and revise some important lemmas.
\begin{defi}
The equation \eqref{main2} can be denoted as $F(u)=0$, with $u$ quasi-periodic in time and periodic in space. If $\|F(u)\| \leq \varepsilon$ in a suitable Banach space, we say $u$ be  an $\varepsilon$ approximate solution of the equation $F(u)=0$.
\end{defi}

\begin{defi}\label{4.00}
Any linear operator $\mathcal{A}:\mathrm{H}_0^1(\mathbb{T}) \mapsto \mathrm{H}_0^1(\mathbb{T})$ can be represented by the infinite dimensional matrix $(\mathcal{A}(\theta)^{i_2}_{i_1})_{i_1,i_2 \in \mathbb{Z}\mathrm{\backslash \{0\} }}$, where $\mathcal{A}(\theta)^{i_2}_{i_1}=(\mathcal{A} e^{\mathrm{i}i_2x}, e^{\mathrm{i}i_1x})$.
Now, we define a new $(s,p)$-decay Banach space $\mathrm{B}^{\varsigma}_{s,p}$ as
\begin{equation}
\mathrm{B}^{\varsigma}_{s,p}:=\bigg\{\mathcal{A}:(|\mathcal{A}|^{\varsigma}_{s,p})^2 =\sum_{i\in \mathbb{Z}}e^{2|i|s}[i]^{2p} (\sup_{i_1-i_2=i} \Big\| \frac{ \mathcal{A}(\theta) ^{i_2}_{i_1} }{{i_1}^2\cdot {i_2}}\Big \|^2_{s,p})< + \infty \bigg\},
\end{equation}
\end{defi}
\begin{defi}\label{4.01} We define some new $(s,p)$-decay Banach space $\widetilde{\mathrm{B}}_{s,p}$ and $\widehat{\mathrm{B}}_{s,p}$ as
\begin{equation}
\widetilde{\mathrm{B}}_{s,p}:= \bigg\{\mathcal{A}:(\widetilde{|\mathcal{A}|}_{s,p})^2 =\sum_{i\in \mathbb{Z}}e^{2|i|s}[i]^{2p} (\sup_{i_1-i_2=i}  \Big\| \frac{ \mathcal{A}(\theta) ^{i_2}_{i_1} \cdot {i_2}^2}{{i_1}^2}\Big\|^2_{s,p}) < +\infty \bigg\},
\end{equation}
\begin{equation}
\widehat{\mathrm{B}}_{s,p}:= \bigg\{\mathcal{A}:(\widehat{|\mathcal{A}|}_{s,p})^2= \sum_{i\in \mathbb{Z}}e^{2|i|s}[i]^{2p}(\sup_{i_1-i_2=i}  \Big\| \frac{\mathcal{ A}(\theta) ^{i_2}_{i_1}\cdot i_1}{i_2} \Big\|^2_{s,p}) < +\infty  \bigg\}.
\end{equation}
We also denote the Banach space $\mathrm{B}_{s,p}^\varrho$ as
\begin{equation}
\mathrm{B}_{s,p}^\varrho:= \bigg\{\mathcal{A}:|\mathcal{A}|_{s,p}^\varrho= max\{\widehat{|\mathcal{A}|}_{s,p},\widetilde{|\mathcal{A}|}_{s,p},|\mathcal{A}|_{s,p} \} < +\infty  \bigg\}.
\end{equation}
\end{defi}

\begin{lemma} \label{ap1}$\mathbf{(Algebra \  property \ 1)}$.For all $p\geq s_0 > \frac{v+1}{2}$, if $\mathcal{A},\mathcal{B} \in \mathrm{B}_{s,p}^\varrho $, then $\mathcal{A}\mathcal{B} \in  \mathrm{B}_{s,p}^\varrho$. Also, there are $c(p)>0$, Such that
\begin{equation}\label{4.001}
|\mathcal{AB}|^{\varrho}_{s,p} \leq c(p)|\mathcal{A}|^{\varrho}_{s,p}|\mathcal{B}|^\varrho_{s,p}.
\end{equation}
If $\mathcal{A}=\mathcal{A}(\lambda)$ and $\mathcal{B}=\mathcal{B}(\lambda)$ depend in a Lipschitz way on the parameter $\lambda \in \Pi \subset \mathbb{R}$, then
\begin{equation}\label{4.01}
|\mathcal{AB}|^{\varrho,Lip}_{s,p} \leq c(p)|\mathcal{A}|^{\varrho,Lip}_{s,p}|\mathcal{B}|^{\varrho,Lip}_{s,p}
\end{equation}
\end{lemma}
\begin{proof}To prove \eqref{4.001}. we will respectively prove the following three cases.

Case 1: If $\mathcal{A},\mathcal{B} \in \mathrm{B}_{s,p}$, then $\mathcal{A}\mathcal{B} \in \mathrm{B}_{s,p}$ with $|\mathcal{AB}|_{s,p} \leq c(p)|\mathcal{A}|_{s,p}|\mathcal{B}|_{s,p}$.
From Lemma \ref{2.2}, this case is simple.

Case 2: If $\mathcal{A},\mathcal{B} \in \widetilde{\mathrm{B}}_{s,p}$, then $\mathcal{A}\mathcal{B} \in \widetilde{\mathrm{B}}_{s,p}$ with $\widetilde{|\mathcal{AB}|}_{s,p} \leq c(p)\widetilde{|\mathcal{A}|}_{s,p}\widetilde{|\mathcal{B}|}_{s,p}$.
With regard to $\mathcal{A},\mathcal{B} \in \widetilde{\mathrm{B}}_{s,p}$, we can define $\mathcal{Q}_1,\mathcal{Q}_2$, where $(\mathcal{Q}_1)^j_i=\frac{\mathcal{A}^j_ij^2}{i^2},(\mathcal{Q}_2)^j_i=\frac{\mathcal{B}^j_ij^2}{i^2}$. Thus, $\mathcal{A},\mathcal{B}$ can be seen as $\partial_{xx}\mathcal{Q}_1\partial^{-1}_{xx}$ and $\partial_{xx}\mathcal{Q}_2\partial^{-1}_{xx}$, with $|\mathcal{Q}_1|_{s,p}=\widetilde{|\mathcal{A}|}_{s,p}, |\mathcal{Q}_2|_{s,p}=\widetilde{|\mathcal{B}|}_{s,p}$.

Then,
\begin{equation}
\begin{split}
\widetilde{|\mathcal{AB}|}_{s,p}&=\widetilde{|\partial_{xx}\mathcal{Q}_1 \mathcal{Q}_2 \partial^{-1}_{xx}|}_{s,p}=|\mathcal{Q}_1\mathcal{Q}_2|_{s,p}\\
& \leq c(p)|\mathcal{Q}_1|_{s,p}|\mathcal{Q}_2|_{s,p}=c(p)\widetilde{|\mathrm{A}|}_{s,p}\widetilde{|\mathrm{B}|}_{s,p}
\end{split}
\end{equation}

Case 3: If $\mathcal{A},\mathcal{B} \in \widehat{\mathrm{B}}_{s,p}$, then $\mathcal{A}\mathcal{B} \in \widehat{\mathrm{B}}_{s,p}$ with $\widehat{|\mathcal{AB}|}_{s,p} \leq c(p)\widehat{|\mathcal{A}|}_{s,p}\widehat{|\mathcal{B}|}_{s,p}$. The proof of this case is almost the same with Case 2.

Now, \eqref{4.001} is proved. The proof of \eqref{4.01} is standard.
\end{proof}

\begin{lemma}\label{ap2} $\mathbf{(Algebra \  property \ 2)}$ For all $p\geq s_0 > \frac{v+1}{2}$, if $\mathcal{A},\mathcal{C} \in \mathrm{B}_{s,p}^\varrho $, $\mathcal{B} \in \mathrm{B}_{s,p}^\varsigma$, then $\mathcal{A}\mathcal{B}\mathcal{C} \in  \mathrm{B}_{s,p}^\varsigma$. Also, there are $c(p)>0$, Such that
\begin{equation}\label{4.0001}
\quad |\mathcal{ABC}|^{\varsigma}_{s,p} \leq c(p)|\mathcal{A}|^{\varrho}_{s,p}|\mathcal{B}|^\varsigma_{s,p}|\mathcal{C}|^{\varrho}_{s}.
\end{equation}
If $\mathcal{A}=\mathcal{A}(\lambda)$, $\mathcal{B}=\mathcal{B}(\lambda)$ and $\mathcal{C}=\mathcal{C}(\lambda)$ depend in a Lipschitz way on the parameter $\lambda \in \Pi \subset \mathbb{R}$, then
\begin{equation}\label{4.00001}
|\mathcal{ABC}|^{\varsigma,Lip}_{s,p}\leq c(p)|\mathcal{A}|^{\varrho,Lip}_{s,p}|\mathcal{B}|^{\varsigma,Lip}_{s,p}|\mathcal{C}|^{\varrho,Lip}_{s,p}.
\end{equation}
\end{lemma}
\begin{proof}
 From  Definition \ref{matrix} and Definition \ref{4.00}. If $\mathcal{A},\mathcal{C} \in \mathrm{B}^\varrho_{s,p}$, they can be seen as  $\partial_{xx}\mathcal{Q}_1\partial^{-1}_{xx}$ and $ \partial^{-1}_{x}\mathcal{Q}_2 \partial_{x}$, with $|\mathcal{Q}_1|_{s,p} \leq |\mathcal{A}|_{s,p}^{\varsigma}, |\mathcal{Q}_2|_{s,p} \leq |\mathcal{C}|_{s,p}^{\varsigma}$.  If $\mathcal{B} \in \mathrm{B}^{\varsigma}_{s,p}$, it can be seen as $\partial_{xx}\mathcal{Q}\partial_{x}$, with $|\mathcal{Q}|_{s,p}=|\mathcal{B}|_{s,p}^{\varsigma}$.
Thus,
\begin{equation}
\begin{split}
|\mathcal{ABC}|^{\varsigma}_{s,p}&=|\partial_{xx}\mathcal{Q}_1\mathcal{Q}\mathcal{Q}_2\partial_{x}|^{\varsigma}_{s,p}=|\mathcal{Q}_1\mathcal{Q}\mathcal{Q}_2|_{s,p}\\
& \leq c(p)|\mathcal{Q}_1|_{s,p}|\mathcal{Q}|_{s,p}|\mathcal{Q}_2|_{s,p}\leq c(p)|\mathcal{A}|^{\varrho}_{s,p}|\mathcal{B}|^{\varsigma}_{s,p}|\mathcal{C}|^{\varrho}_{s,p}.
\end{split}
\end{equation}
The \eqref{4.00001} is standard.
\end{proof}

\begin{lemma} If $\mathcal{A} \in \mathrm{B}^{\varsigma}_{s,p}, h \in \mathrm{H}_{s,2p}$, then $\mathcal{A}h \in \mathrm{H}_{s,p}$ with
\begin{equation}\label{4.03}
\|\mathcal{A}h\|_{s,p} \leq c(p)|\mathcal{A}|^{\varrho}_{s,p}\|h\|_{s,2p},\quad
\|\mathcal{A}h\|^{Lip}_{s,p} \leq c(p)|\mathcal{A}|^{\varrho,Lip}_{s,p}\|h\|_{s,2p},
\end{equation}
 If $\mathcal{A} \in \mathrm{B}^{\varsigma}_{s,p}, h \in \mathrm{H}_{s,2p+1}$, then $\mathcal{A}h \in \mathrm{H}_{s,p-2}$ with
\begin{equation}\label{4.04}
\|\mathcal{A}h\|_{s,p-2}  \leq c(p)|\mathcal{A}|^{\varsigma}_{s,p}\|h\|_{s,2p+1},\quad
\|\mathcal{A}h\|^{Lip}_{s,p-2} \leq c(p) |\mathcal{A}|^{\varsigma,Lip}_{s,p}\|h\|^{Lip}_{s,2p+1}.
\end{equation}
\end{lemma}
\begin{proof}
\eqref{4.03} is standard.
To prove \eqref{4.04}, by Lemma \ref{2.2} and Lemma \ref{ap2}, we see
\begin{equation}
\begin{split}
\|\mathcal{A}h\|_{s,p-2} &=\|\partial_{xx}\mathcal{Q}\partial_xh\|_{s,p-2} \leq \|\mathcal{Q}\partial_xh\|_{s,p}\\
& \leq |Q|_{s,p}\|\partial_xh\|_{s,2p} \leq |\mathcal{A}|^{\varsigma}_{s,p}\|h\|_{s,2p+1}
\end{split}
\end{equation}
\end{proof}
Now, we recall the classical  Kuksin's lemma.
\begin{lemma}[Kuksin]
Consider the following first order partial differential equation
\begin{equation}
 -\mathbf{i}\omega \cdot \partial_{\theta}u+d u +\mu(\theta)u=p(\theta),\quad \theta\in \mathbb{T}^v,
\end{equation}
for the unknown function $u$ defined on the torus $\mathbb{T}^{v}$, where $\omega =(\omega_1,\cdots,\omega_n) \in \mathbb{R}^v  $ and $d \in \mathbb{R} $. We make the following assumption.
\\

 $\mathbf{Assumption \ A:}$ There are constants $\alpha,\gamma>0$ and $\tau >v$ such that
\begin{equation}
|\langle \ell \cdot\omega \rangle |\geq \frac{\alpha}{|k|^{\tau}},
\end{equation}
\begin{equation}
|\langle \ell \cdot\omega \rangle +d|\geq \frac{\alpha \gamma}{|k|^{\tau}},
\end{equation}
for all $0 \neq \ell \in \mathbb{Z}^v$. Also, $|d|\geq \alpha \gamma$.
\\

 $\mathbf{Assumption \ B:}$ The function $\mu$  is analytic on some  complex strip $D(s)=\{ x:|\mathfrak{Im} x |<s\} \subset \mathbb{C}^n$ around $\mathbb{T}^v$ with mean value zero£º $[\mu]= \int_{\mathbb{T}^v}\mu(\theta) d\theta=0$. Moreover,
\begin{equation}
^{l_1}|\mu|_{s,\tau}\overset{def}{=}\sum_{\ell \in \mathbb{Z}^v}|\mu_k|[\ell]^\tau e^{|\ell|s} \leq C \tilde{\gamma}.
\end{equation}
for $\mu =\sum_{\ell \in \mathbb{Z}^v}\mu_k e^{\mathrm{i}\ell x}$ with some $C>0$.
\\

$\mathbf{Assumption \ C:}$ $p(\theta)$ is analytic on the same complex strip $ D(s)$, and $d \geq \tilde{\gamma}^{1+\beta}$, with some $\beta > 0$.

Then the equation has a unique solution $u(\theta)$ defined in a narrower domain $D(s-\sigma)$, with $0 < \sigma <s$, which satisfies
\begin{equation}\label{4.06}
\|u(\theta)\|_{s-\sigma,s_0} \leq \frac {c{e^{2(5/\sigma)}}^{1/\beta}}{\alpha \tilde{\gamma} \sigma^{2v+\tau+s_0+3}} \|p(\theta)\|_{s,s_0},
\end{equation}
where $c$ depend on $v$ and $\tau$.
\end{lemma}
\begin{proof}The original proof  can be found in \cite{KuKsin1} and \cite{Kappler1}. Totally following the proof by Kuksin \cite{KuKsin1} and \cite{Kappler1} and noting Lemma \ref{estimation2}, \eqref{4.06} is verified.
\end{proof}

 The inverse of $\mathcal{L}(u_n)$ is our main concern. However, the estimate of $\mathcal{L}(0)$,  a diagonal operator,  is straightforward. In order to make the structure of this paper much more simplicity, the initial approximate solution is $u_1$ other than $u_0$.  So, we will estimate the inverse of $\mathcal{L}(0)$  and set the initial parameter .
\begin{lemma}
The linear equation $\mathcal{L}(0)v=F(0)$, for all $\lambda \in \Gamma(u_0)$,
\begin{equation}\label{2.40}
\Gamma(u_0):=\{ \lambda :|\omega \cdot \ell+k^5| \geq \alpha_0 \frac{k^5}{[\ell]^\tau},  \forall \ell \in \mathbb{Z}^v  , k \in \mathbb{Z} \backslash \{0\} \}
\end{equation}
has a  unique solution $v$ with zero average, satisfying $\|v\|^{Lip}_{s,p'} \leq \frac{1}{\alpha^2} \| F(0)\|^{Lip}_{s,p'+2\tau+1}$.
\end{lemma}
\begin{proof}
 The equation $\mathcal{L}(0)v=F(0)$ is equivalent to
 \begin{equation}\label{2.433}
 \omega \cdot \partial_{\varphi}v(\varphi,x)+\partial^3_x v(\varphi,x)=F(\varphi,x),
 \end{equation}
 where
\begin{equation}\label{2.43}
v(\varphi, x)=\sum_{k\in  \mathbb{Z}\backslash \{0\}} v_{k}(\varphi) e^{\mathbf{i}kx}, \quad F(\varphi, x)=\sum_{k\in  \mathbb{Z}\backslash \{0\}}  F_{k}(\varphi) e^{\mathbf{i}kx}.
\end{equation}
Thus, \eqref{2.433} can be transformed to
\begin{equation}
\omega \cdot \partial_{\varphi}v_k(\varphi)+(\mathbf{i}k)^5v_k(\varphi)=F_k(\varphi), \quad k \in \mathbb{Z} \backslash \{0\},
\end{equation}
whose solutions are $v_k(\varphi)=\sum_{\ell \in \mathbb{Z}^v} v_{\ell,k}e^{\mathbf{i} \ell\varphi}$ with coefficients
\begin{equation}
v_{\ell,k}=\frac{F_{\ell,k}}{\mathbf{i}\omega \cdot \ell+\mathbf{i}k^5} \quad \forall \ell \in \mathbb{Z}^v , k \in \mathbb{Z} \backslash \{0\}.
\end{equation}
Using \eqref{2.40}, we have
\begin{equation}\label{2.41}
\|v\|_{s,p'+\tau+1} \leq \frac{1}{\alpha_0}\| F(0)\|_{s,p'+2\tau+1}.
\end{equation}
Applying the operator $\Delta v=v(\lambda_1)-v(\lambda_2)$ to $\mathcal{L}(0)v=F(0)$, we have
\begin{equation}\label{2.42}
\omega \cdot \partial_{\varphi}\Delta v_k(\varphi)+(\mathbf{i}k)^5\Delta v_k(\varphi)=\Delta F_k(\varphi)-\Delta \lambda \cdot \overline{\omega} \cdot \partial_{\varphi} v_k(\varphi).
\end{equation}
Again using \eqref{2.40}, we see
\begin{equation}
\|\Delta v\|_{s,p'} \leq \frac{1}{\alpha_0}\| \Delta F(0)\|_{s,p'+\tau}+\Delta\lambda \cdot\frac{1}{\alpha_0^2} \|F(0)\|_{s,p'+\tau+1}.
\end{equation}
Combining \eqref{2.41} with \eqref{2.42}, one gets
\begin{equation}
\|v\|^{Lip}_{s,p'} \leq \frac{1}{\alpha_0^2} \| F(0)\|^{Lip}_{s,p'+2\tau+1}.
\end{equation}
\end{proof}

\begin{rem}
Obviously, $\| F(0)\|^{Lip}_{s,p'+2\tau+1}=\|\partial_xf(\varphi,x)\|^{Lip}_{s,p'+2\tau+1} \leq \varepsilon $. Set $\varepsilon_1$ as $\frac{1}{\alpha^2}\varepsilon $,  $v$  as $v_1$, $p'$ as $p+\eta$. Then, we have $\|v_1\|^{Lip}_{s,p+\eta} \leq \varepsilon_1$. Since
\begin{equation}
F(v_1)=10v_1\partial_x^3v_1+20\partial_xv_1\partial_x^2v_1+30v_1^2 \partial_xv_1-6\partial^2_xv_1 \partial^5_xv_1-18\partial^3_xv_1\partial^4_xv_1,
\end{equation}
one gets
\begin{equation}
\|F(u_1)\|^{Lip}_{s,p+\eta-5}=\|F(v_1)\|^{Lip}_{s,p+\eta-5}\leq c (\|v_1\|^{Lip}_{s,p+\eta})^2 \leq \varepsilon_1^{\frac{8}{5}}.
\end{equation}
\end{rem}

\textbf{KAM step}: In this section, we will give the outline of the reducibility and show in detail one key step of the KAM iteration. The purpose is to define a transformation operator $\Phi_m$ conjugating $\mathfrak{L}_{m}$, a diagonal operator $\mathcal{J}_m$ plus a $\varepsilon_{m}$ remainder $\mathcal{R}_m$, to $\mathfrak{L}_{m+1}$,  a diagonal operator $\mathcal{J}_{m+1}$ plus a $\varepsilon_{m+1}$ remainder $\mathcal{R}_{m+1}$.

Now, we have already got the regularized linear operator $\mathfrak{L}(u_n)$ at the approximate solution $u_{n}$, which is
 \begin{equation}
\mathfrak{L}= \omega \cdot\partial_{\theta}+m \partial^5_{y}+
\partial_{y}\{\partial_y[ c_1(\theta,y)\partial_y)]+c_0(\theta,y)\}.
 \end{equation}
Now, the linear operator $\mathfrak{L}$  can be denote as
\begin{equation}
\mathfrak{L}=\mathfrak{L}_1=\omega \cdot \partial _{\theta} \textbf{1}+\mathcal{D}+\mathcal{R}=\mathcal{J}+\mathcal{R},
\end{equation}
where
\begin{equation}
 \mathcal{D}=m \partial^5_{y}, \quad \mathcal{R}=\partial_{y}\{\partial_y[ (c_1(\theta,y)\partial_y)]+c_0(\theta,y)\}.
\end{equation}
According to  Definition \ref{4.00}, we see $|R|^{\varsigma}_{s,s_0} \leq \|c_1(\theta,y)\|_{s,p}+\|c_0(\theta,y)\|_{s,p} $.

By Lemma \ref{3.301} and Lemma \ref{estimation3}, the coefficients of perturbation term $\mathcal{R}$ can be divided into $n$ parts, which is
\begin{equation}
c_i(u_{n})=c_i(u_1)+(c_i(u_2)-c_i(u_1))+\cdots (c_i(u_{n})-c_i(u_{n-1})).
\end{equation}
Set $u_{m}-u_{m-1}=v_{m}$. From the following Lemma \ref{4.10},  $(c_i(u_{m})-c_i(u_{m_-1}))$ is as small as $v_{m}$, the function space of $(c_i(u_{m})-c_i(u_{m_-1}))$  can also be controlled by $v_{m}$.  Now, the linear operator $\mathfrak{L}(u_n)$ can be seen as
 \begin{equation}
 \mathfrak{L}= \omega \cdot\partial_{\theta}\mathbf{1}+ \mathcal{D}+\mathcal{K}^1+\cdots +\mathcal{K}^n=\omega \cdot\partial_{\theta}\mathbf{1}+ \mathcal{D}+\mathcal{R}_1+\mathcal{Q}_1
\end{equation}
where
$$\mathcal{R}_1=\mathcal{K}^1, \quad  \mathcal{Q}_1=\sum^n_{i=2}\mathcal{K}^i,$$

$$\mathcal{K}^i=\partial_{y}\{\partial_y[ ((c_1(u_{i})-c_1(u_{i-1})\partial_y)]+(c_0(u_{i})-c_0(u_{i-1})\}.$$

\begin{lemma}\label{4.10}
Assume $\|u_m\|^{Lip}_{s_m,p+\eta} \ll \frac{1}{100}$, for all $m \geq 1 $, we have
\begin{equation}\label{4.101}
|m(u_{m})-m(u_{m-1})|^{Lip} \leq \|v_m\|^{Lip}_{s_m,\eta}.
\end{equation}
For $i=0,1$,  we have
\begin{equation}\label{4.11}
\|c_i(u_{m})-c_i(u_{m-1})\|^{Lip}_{k_1s_m,p} \leq \|v_m\|^{Lip}_{s_m,p+\eta},
\end{equation}
\begin{equation}\label{4.12}
|\mathcal{K}^m|^{\varsigma,Lip}_{k_1s_m,s_0} \leq C \|v_m\|^{Lip}_{s_m,p+\eta}.
\end{equation}
The $s_m$  will be define later.
\end{lemma}
\begin{proof}
The estimates \eqref{4.101} and \eqref{4.11} is a direct result of Lemma \ref{estimation3} and Lemma \ref{3.301}.
Definition \ref{4.00} implies the estimates of  \eqref{4.12}.
\end{proof}

  The purpose of reducibility is to make the reminder of the linear operator $\mathfrak{L}_m$ much more small. If the the reminder of  $\mathfrak{L}_m$ can be divided into $\mathcal{Q}_m$ and $\mathcal{R}_m$ , $\mathcal{R}_m$ lies in a much more general analytical space and $\mathcal{Q}_m$ is much more smaller. We can  consider the homological equation to eliminate $\mathcal{R}_m$. Thus,  the transformation operator $\Psi_m$ can lies in a much more general Banach space.

In order to exhibit the outline of our reducibility, we will give the outline of the one step of reduction. The transformation  $\Psi_m=e^{\Phi_m}$ acting on the operator $\mathfrak{L}_m$:
\begin{equation}
\mathfrak{L}_m=\omega \cdot \partial _{\theta}\textbf{1}+\mathcal{D}_m+\mathcal{R}_m+\mathcal{Q}_m,
\end{equation}
where
$$\mathcal{D}_m=\emph{diag}_{h \in \mathbb{Z}\backslash \{0\}}\mathrm{i}\{d_h +\mu_h(\theta)\},\quad d_h \approx h^5 ,\quad \int_{\mathbb{T}^v} \mu_h(\theta) d\theta=0,$$
\begin{equation}\label{4.13}
\mathcal{Q}_m=\prod^1_{i=m-1}\Psi^{-1}_i[\sum^{n}_{i=m}\mathcal{K}_i]\prod^{m-1}_{i=1}\Psi_i.
\end{equation}
Then, we can  get
\begin{equation}
\begin{split}
 \Psi^{-1}_m  \mathfrak{L}_m\Psi_m =& e^{-\Phi_m}[\omega \cdot\partial_{\varphi}(e^{\Phi_m})+\mathcal{D}_me^{\Phi_m} + \mathcal{R}_me^{\Phi_m}+\mathcal{Q}_me^{\Phi_m}] \\
 =&\omega\cdot\partial_{\varphi}+\mathcal{D}_m+\mathrm{diag}[\mathcal{R}_m] +(\omega \cdot\partial_{\varphi} \Psi_m+[\mathcal{D}_m,\Phi_m]+\mathcal{R}_m-\mathrm{diag}[\mathcal{R}_m]  \\
 &+(e^{-\Phi_m}\mathcal{D}_m e^{\Phi_m}-\mathcal{D}_m-[\mathcal{D}_m,\Phi_m])+(e^{-\Phi_m}\mathcal{R}_me^{\Phi_m}-\mathcal{R}_m)\\
 &+(e^{-\Phi_m} \omega\cdot\partial_{\varphi}\Phi_m e^{\Phi_m}- \omega \cdot\partial_{\varphi} \Phi_m)  \\
 &+\prod^1_{i=m}\Psi^{-1}_i[\mathcal{K}_{m}]\prod^{m}_{i=1}\Psi_i+\prod^1_{i=m}\Psi^{-1}_i[\sum^{n}_{i=m+1}\mathcal{K}_{i}]
 \prod^{m}_{i=1}\Psi_i,
\end{split}
\end{equation}
where $[\mathcal{D}_m,\Phi_m]£»=\mathcal{D}_m \Phi_m - \Phi_m \mathcal{D}_m$.

If we solve the  homological equation
\begin{equation}
\omega \cdot\partial_{\varphi} \Psi_m+[\mathcal{D}_m,\Phi_m]+\mathcal{R}_m=\mathrm{diag}[\mathcal{R}_m],
\end{equation}

 $\mathfrak{L}_{m+1}$ can be denote as
\begin{equation}
\mathfrak{L}_{m+1}=\Psi^{-1} \mathfrak{L}_m\Psi=\omega \cdot \partial _{\varphi}\textbf{1}+\mathcal{D}_{m+1}+\mathcal{R}_{m+1}+\mathcal{Q}_{m+1},
\end{equation}
where
\begin{equation}
\mathcal{D}_{m+1} =\mathcal{D}_{m}+\mathrm{diag}[\mathcal{R}_m],
\end{equation}
\begin{equation}\label{RM}
\mathcal{R}_{m+1}=\mathcal{R}_m^+ + \prod^1_{i=m}\Psi^{-1}_i[\mathcal{K}_{m+1}]\prod^{m}_{i=1}\Psi_i,
\end{equation}
\begin{eqnarray}\label{R+}
\mathcal{R}^+_{m}&=&(e^{-\Phi_m}\mathcal{D}_me^{\Phi_m}-\mathcal{D}_m-[\mathcal{D}_m,\Phi_m])+(e^{-\Phi_m}\mathcal{R}_me^{\Phi_m}-\mathcal{R}_m)
 \nonumber \\
 &&+(e^{-\Phi_m} \omega\cdot\partial_{\varphi}\Phi_m e^{\Phi_m}- \omega \cdot\partial_{\varphi} \Phi_m), \nonumber \\
\end{eqnarray}
\begin{equation}\label{QM}
\mathcal{Q}_{m+1}=\prod^1_{i=m}\Psi^{-1}_i[\sum^{n}_{i=m+2}\mathcal{K}_{i}]\prod^{m}_{i=1}\Psi_i.
\end{equation}

Before we give the iteration lemmas, we need the following iteration constants and domains.

\textbf{iteration parameters}:
Set $n \geq 1, m \geq 1$. Then, $(m,n)$ indicates the $m^{th}$ step KAM reduction for the linear operator $\mathfrak{L}(u_n)$.

$\bullet$ $\varepsilon_1 = \varepsilon$, $\varepsilon_m={\varepsilon^{(\frac{4}{3})}}^{m-1}$, which dominate the size of the perturbation $\mathcal{R}_m$ in KAM iteration, the modified function $v_m$ and $c_i(u_{m-1}+v_m)-c_i(u_{m-1})$.

$\bullet$ $s_n=(\frac{10}{11})^{n-1} s_1$, $s_1=s$, which dominate the width of the $u_n$.\\

$\bullet$ $s'_n=\frac{99}{101}s_n$, $s_1=s$, which dominate the width of the coefficients $c_i(u_n)$ and $\mathcal{R}_n$.\\

$\bullet$ $\sigma_m=\frac{1}{200}s_m $, which serve as a bridge from $s_m$ to $s_{m+1}$.\\

$\bullet$ $\alpha_{mn}=\frac{\alpha_0}{2^m}(1+\frac{1}{2^{n-m}})$, which dominate the measure of parameters removed in the $(m,n)-$ step KAM iteration.\\

$\bullet$ $C_{d,m}=\frac{1}{2}(1+\frac{1}{2^{m+1}})$, which $\frac{1}{2}<C_{d,m} \leq 1$.\\

$\bullet$ $C_{\lambda,m}=(2-\frac{1}{2^m})\varepsilon$, which $ \varepsilon \leq C_{\lambda,m} \leq 2\varepsilon$.\\

$\bullet$ $C_{\mu,m}=c(2-\frac{1}{2^m})\varepsilon$, which $ c \varepsilon \leq C_{\lambda,m} \leq 2c \varepsilon$.\\

Set these parameter
\begin{equation}
0 < \varepsilon \ll \alpha_0 \ll min\{ \frac{1}{100},s\}.
\end{equation}

\textbf{iteration lemmas}:

$\mathbf{[H1]_n}$
Assume that $q \geq p+\eta+2\tau_0+1$, $\partial_xf$ satisfies the assumption of Theorem \ref{main result}. Let $\tau > v+1$, then, for all $n\geq       1$,

$(\mathcal{P}1)_n:$ There exist a function $u_n: \lambda\subseteq \Lambda(u_n) \rightarrow u_n(\lambda)$, with $\|u_n\|^{Lip}_{s_n,p+\eta} \leq 2\varepsilon$, where $\Lambda(u_n)$ are cantor like subset of $\Pi =[\frac{1}{2},\frac{3}{2}]$.

The difference function $v_n=u_n-u_{n-1}$, where, for convenience, $v_0=0$, satisfy
\begin{equation}
 \quad \|v_i\|^{Lip}_{s_i,p+\eta}\leq \varepsilon_{i}.
\end{equation}

$(\mathcal{P}2)_n:$ $\|F(u_n)\|^{Lip}_{s_n,p+\eta-5} \leq \varepsilon^{\frac{6}{5}}_{n+1}.$

$\mathbf{[H2]_m}$

Assume we have get the following operator
\begin{equation}
\mathfrak{L}_m=\omega \cdot \partial _{\theta}\textbf{1}+\mathcal{D}_m+\mathcal{R}_m+\mathcal{Q}_m,
\end{equation}
after $(m-1)^{th}$ step KAM reduction for the linear operator $\mathfrak{L}(u_n)$, which satisfies the following hypothesis:

$(\mathcal{S}1)_m:$
\begin{equation}
\mathcal{D}_m=\emph{diag}_{k \in \mathbb{Z}\backslash \{0\}}\mathrm{i}\{d^m_k(\lambda) +\mu^m_k(\lambda,\theta)\,
\end{equation}
 \begin{equation}
d^{m}_k(u_n)=m(u_n)k^5+r_k^{m}k^3=m(u_n)k^5+r_k^{m-1}k^3+[\mathcal{R}_{m-1}(k,k)],
\end{equation}
defined for all $\lambda \in \Lambda_{m.n}$, where $\Lambda_{0.n}=\Lambda(u_n)$ (is the domain of $u_n$), and, for  $m\leq n$,
\begin{equation}
\Lambda_{m.n}:= \left \{  \lambda \in \Lambda_{m-1.n}:|\ell\cdot\lambda\overline{\omega}+d^{m}_i(u_n)-d^{m}_j(u_n)| \geq \frac{\alpha_{mn}|i^5-j^5|}{[\ell]^\tau}  \right \},
\end{equation}
with
\begin{equation}\label{4.13}
\cdots <d_{-1}(\lambda)<0< d_{1}(\lambda)<\cdots ,\quad |d^m_i(u_n)-d^m_j(u_n)|\geq (m(u_n)+C_{d,m})|i^5-j^5|.
\end{equation}

$d^m_i(u_n)$ is $Lipschitz-continuous$ in $\lambda$, and fulfills the estimate:
\begin{equation}
\sup_{\lambda_1,\lambda_2 \in \Lambda}\frac{d^m_i(\lambda_1)-d^m_i(\lambda_1)}{\lambda_1-\lambda_2} \leq m^{lip}(u_n)i^5+C_{\lambda,m}i^3.
\end{equation}

$\mu_k(\lambda,\theta)$ is  real analytic in $\theta$ and  $Lipschitz-continous$ in $\lambda$  of zero average. It also satisfies
\begin{equation}\label{4.18}
^{l_1}|\mu^m_k|_{s'_m,\tau} \leq C_{\mu,m}k^3, \quad \|\mu^m_k\|^{Lip}_{s'_m,s_0} \leq C_{\lambda,m}k^3.
\end{equation}

$\mathcal{Q}_m$ is defined in \eqref{QM}, and $\mathcal{Q}_{n+1}=0$.

$(\mathcal{S}2)_m:$ The reminder $\mathcal{R}_m$ is $Lipschitz-continous$ in $\lambda$, and satisfies the estimate:
\begin{equation}
|\mathcal{R}_m|^{\varsigma,Lip}_{s'_m,s_0} \leq C \varepsilon_m,  \quad \forall  m\leq n,
\end{equation}
\begin{equation}
|\mathcal{R}_{n+1}|^{\varsigma,Lip}_{s'_n-2\sigma_n,s_0} \leq  \varepsilon^{\frac{4}{3}}_n.
\end{equation}
$C$ is an constant only depend on $v$. Moreover, for $m\geq 1$, we have
\begin{equation}
\mathfrak{L}_{m+1}=\Psi^{-1}_m \mathfrak{L}_{m}\Psi_m, \quad \Psi_m=e^{\Phi_m},
\end{equation}
\begin{equation}\label{4.21}
|\Phi_m|^{\varrho,Lip}_{s'_m-2\sigma_m,s_0} \leq \varepsilon_m^{\frac{5}{6}}.
\end{equation}
\begin{rem} We only make $n-step$ KAM reduction for the linear operator $\mathfrak{L}(u_n)$, conjugating it to $\mathfrak{L}_{n+1}(u_n)$.
\end{rem}
\begin{rem}In the Hamiltonian case $\Phi_m$ is Hamiltonian, the transformation operator
$$\Psi_m=e^{\Phi_m}$$
is a $symplectic$ map. The corresponding operator $\mathfrak{L}_m,\mathcal{R}_m$ are Hamiltonian, then, $\mathcal{R}_{m}(k,k)$ can be guaranteed to be pure imaginary.
\end{rem}

\begin{coro}\label{4.20}
$\forall \lambda \in \Lambda_{n,m}$, the sequence
\begin{equation}
\Omega_m=\prod^m_{i=1}\Psi_i=\Psi_1\circ\Psi_2\circ \cdots \circ\Psi_m
\end{equation}
satisfies the following estimate
\begin{equation}\label{4.22}
|\Omega_m^{-1}-\mathrm{I}|^{\varrho,Lip}_{s'_m-2\sigma_m,s_0}+|\Omega_m-\mathrm{I}|^{\varrho,Lip}_{s'_m-2\sigma_m,s_0} \leq 1.
\end{equation}
\end{coro}
\begin{proof}
\eqref{4.001} and \eqref{4.21} imply \eqref{4.22}.
\end{proof}

Now, we would prove the iteration Lemma $\mathbf{[H2]_m}$.
\begin{proof}For convenience, let $\mathfrak{L},\mathcal{R}$ refer to $\mathfrak{L}_m,\mathcal{R}_m$,   $\mathfrak{L}_+,\mathcal{R}_+$ refer to $\mathfrak{L}_{m+1},\mathcal{R}_{m+1}$.

\textbf{(Part1: Homological equation)}

See $\mathcal{D}_{ii}=\mathrm{i}(d_i+\mu_i(\varphi))$. Multiply $-\mathrm{i}$ on both side of the following equation:
\begin{equation}
\omega \cdot\partial_{\theta} \Phi+[\mathcal{D},\Phi]+\mathcal{R}=\mathrm{diag}[\mathcal{R}].
\end{equation}
Then, the homological equation is equivalent to

\textbf{(1)} $i=j$: $\Phi_{ii}=0$,

\textbf{(2)} $i\neq j$:
\begin{equation}\label{h solution}
-\mathrm{i}\omega\cdot \partial_\theta \Phi_{ij}+(d_i-d_j)\Phi_{ij} +(\mu_i(\theta)-\mu_j(\theta))\Phi_{ij}-\mathrm{i}\mathcal{R}_{ij}=0.
\end{equation}
See $d_{ij}=d_i-d_j,\mu_{ij}=\mu_i(\theta)-\mu_j(\theta), \chi _{ij}=|i^5-j^5|, \iota_{ij}=|i-j|$. By \eqref{4.18}, we have
\begin{equation}
^{l_1}|\mu_{ij}|_{s'_m,\tau} \leq C_{\mu m}(i^3+j^3).
\end{equation}
Now, applying Kuksin's lemma to \eqref{h solution}, we get
\begin{equation}
\|\Phi_{ij}\|_{s'_m-\sigma_m,s_0} \leq \frac{c e^{2(5/\sigma_m)^{1/\beta}}}{\alpha_{mn} \chi_{ij} \sigma_m^{2v+\tau+s_0+3}} \|\mathcal{R}_{ij}\|_{s'_m,s_0}.
\end{equation}
Since
\begin{equation}
\beta \geq \frac{1}{3}, \quad \sigma_m=\frac{1}{200}s'_m =\frac{1}{200}{(\frac{10}{11})}^{m-1}s,
\end{equation}
we get
\begin{equation}
e^{2(5/\sigma_m)^{1/\beta}}={c(s)^{(\frac{11}{10})}}^{3(m-1)} \leq {c(s)^{(\frac{4}{3})}}^{m-1},
\end{equation}
 $c$ is an constant only depending on $s$. Let
\begin{equation}
\varepsilon^{\frac{1}{20}}\leq c(s)^{-1},
\end{equation}
we have
\begin{equation}\label{5.021}
 \|\Phi_{ij}\|_{s'_m-\sigma_m,s_0} \lessdot  \frac{\varepsilon^{-1/20}_m}{\alpha_{mn} \sigma_m^{2n+\tau+s_0+3} \chi_{ij}}\|\mathcal{R}_{ij}\|_{s'_m,s_0}.
\end{equation}
Then, consider the infinite matrices of elements
\begin{equation}\label{5.022}
\frac{\mathcal{R}_{ij}j^2}{i^2(i^4+j^4)}, \quad \frac{\mathcal{R}_{ij}i}{j(i^4+j^4)}, \quad \frac{\mathcal{R}_{ij}}{(i^4+j^4)}.
\end{equation}
Combining \eqref{5.021}, \eqref{5.022} with the Definition \ref{4.00}, we have the estimates of  matrix $\Phi$,
\begin{equation}\label{5.03}
\widehat{|\Phi|}_{s'_m-\sigma_m,s_0} \lessdot \frac{\varepsilon^{-1/20}_m}{\alpha_{mn} \sigma_m^{2v+\tau+s_0+3}}|\mathcal{R}|^\varsigma_{s'_m,s_0},
\end{equation}
\begin{equation}\label{5.04}
\widetilde{|\Phi|}_{s'_m-\sigma_m,s_0} \lessdot \frac{\varepsilon^{-1/20}_m}{\alpha_{mn} \sigma_m^{2v+\tau+s_0+3}}|\mathcal{R}|^\varsigma_{s'_m,s_0},
\end{equation}
\begin{equation}\label{5.041}
|\Phi|_{s'_m-\sigma_m,s_0} \lessdot \frac{\varepsilon^{-1/20}_m}{\alpha_{mn} \sigma_m^{2v+\tau+s_0+3}}|\mathcal{R}|^\varsigma_{s'_m,s_0}.
\end{equation}

Now,we need a bound on the Lipschitz semi-norm of $\Phi$. Given a function $\Phi$ of $\omega=\lambda \bar{\omega}$, set $\Delta\Phi=\Phi(\lambda_1)-\Phi(\lambda_2)$. Then, applying the operator $\Delta$ to the equation \eqref{h solution}, we get
\begin{equation}\label{5.05}
\begin{split}
-\mathbf{i}(\lambda_1 \bar{\omega})\cdot \partial _{\theta} (\Delta \Phi_{ij}) &+d_{ij}  (\lambda_1)(\Delta\Phi_{ij})+\mu_{ij}(\theta,\lambda_1)(\Delta \Phi_{ij})\\
&=\mathbf{i}(\lambda_1 \bar{\omega}-\lambda_2 \bar{\omega}) \partial _{\theta}(\Phi_{ij}(\lambda_2))-(\Delta d_{ij}+\Delta\mu_{ij})\Phi(\lambda_2)-\mathbf{i} \Delta \mathcal{R}_{ij}, \\
\end{split}
\end{equation}
where
\begin{equation}
\begin{split}
|\Delta d_{ij}|&= |(d_i(\lambda_1)-d_j(\lambda_1))-(d_i(\lambda_2)-d_j(\lambda_2))|\\
&\leq|m(\lambda_1)-m(\lambda_2)||i^5-j^5|+|(r_i(\lambda_1)-r_i(\lambda_2))i^3|+|(r_j(\lambda_1)-r_j(\lambda_2))j^3| \\
&\leq c\varepsilon|i^5-j^5||\Delta\lambda|.
\end{split}
\end{equation}
Again applying Kuksin's lemma, we have
\begin{equation}
\|\Delta\Phi_{ij}\|_{s'_m-\frac{3}{2}\sigma_m} \lessdot \frac{\varepsilon^{-1/10}_m}{\alpha^2_{mn} \sigma_m^{4v+2\tau+7}\chi_{ij}}(|\Delta \lambda|\|\mathcal{R}_{ij}\|_{s'_m}+\|\Delta \mathcal{R}\|_{{s'_m}}),
\end{equation}
and
\begin{equation}\label{5.06}
\|\Phi_{ij}\|^{lip}_{s'_m-2\sigma_m,s_0} \lessdot \frac{\varepsilon^{-1/10}_m}{\alpha^2_{mn} \sigma_m^{4n+2\tau+7+s_0}\chi_{ij}}(\|\mathcal{R}_{ij}\|_{s'_m,s_0}+\| \mathcal{R}_{ij}\|^{lip}_{s'_m,s_0}).
\end{equation}
By \eqref{5.022},\eqref{5.03}, \eqref{5.041}, \eqref{5.041},\eqref{5.05} and \eqref{5.06}, we can obtain
\begin{equation}
\widehat{|\Phi|}^{Lip}_{s'_m-2\sigma_m,s_0} \leq \varepsilon_m^{-\frac{1}{7}} |\mathcal{R}|^{\varsigma,Lip}_{s'_m,s_0},
\end{equation}
\begin{equation}
\widetilde{|\Phi|}^{Lip}_{s'_m-2\sigma_m,s_0} \leq \varepsilon_m^{-\frac{1}{7}} |\mathcal{R}|^{\varsigma,Lip}_{s'_m,s_0},
\end{equation}
\begin{equation}
|\Phi|^{Lip}_{s'_m-2\sigma_m,s_0} \leq \varepsilon_m^{-\frac{1}{7}} |\mathcal{R}|^{\varsigma,Lip}_{s'_m,s_0}.
\end{equation}
Finally, we get
\begin{equation}
|\Phi|^{\varrho,Lip}_{s'_m-2\sigma_m,s_0} \leq  \varepsilon^{\frac{5}{6}}_m.
\end{equation}

\textbf{(Part2: New diagonal part)}

We  have already get the new linear operator
\begin{equation}
\mathfrak{L}_+=\omega \cdot \partial_\theta \mathbf{1}+\mathcal{D}_++\mathcal{R}_++\mathcal{Q}_+,
\end{equation}
where
\begin{equation}
\mathcal{D}_+:=\emph{diag}_{i \in \mathbb{N}}\mathbf{i}\{d^+_i +\mu^+_i(\varphi)\},
\end{equation}
\begin{equation}
d^+_i=d_i+\overline{R_{ii}},\quad \mu^+_i(\varphi)=\mu_i(\varphi)+(R_{ii}-\overline{R_{ii}}),\quad \overline{R_{ii}}=\int_{\mathbb{T}^v}R_{ii}(\theta)d\theta.
\end{equation}

Considering $P_1$, by  Lemma \ref{estimation0}, we see
\begin{equation}
\begin{split}
^{l_1}|R_{ii}-\overline{R_{ii}}|_{s'_m-\sigma_m,\tau} &\leq \|R_{ii}\|_{s'_m}\cdot ( \sum_{k\in \mathbb{Z}^v \backslash\{0\}} e^{-2k\sigma_m}|k|^{2\tau})^{\frac{1}{2}}  \\
&\leq (\frac{2\tau}{e})^{\tau} \sigma_m^{-(\frac{v}{2}+\tau)} {(1+e)}^{\frac{v}{2}} \|R_{ii}\|_{s'_m} \\
&\leq c\frac{1}{s_m}^{\frac{v}{2}+\tau}\varepsilon_m i^3 \\
&\leq c_1\frac{\varepsilon}{2^{m+1}}i^3,
\end{split}
\end{equation}
and
\begin{equation}
\begin{split}
^{l_1}|\mu^+_i(\varphi)|_{s'_m-\sigma_m,\tau} &\leq  ^{l_1}|\mu_i(\varphi)|_{s'_m,\tau}+^{l_1}|R_{ii}-\overline{R_{ii}}|_{s'_m-\sigma_m,\tau}  \\
&\leq c_1(2-\frac{1}{2^m})\varepsilon i^3+c_1\frac{\varepsilon}{2^{m+1}}i^3=C_{\mu.m}i^3.\\
\end{split}
\end{equation}

\textbf{(Part3: Estimates of New perturbed terms)}

Consider the new perturbed terms $ \mathcal{R}_+:=\mathcal{H}_1+ \mathcal{H}_2$, where
\begin{equation}
\begin{split}
\mathcal{H}_1 &= (e^{-\Phi}\mathcal{R}e^{\Phi}-\mathcal{R})+(e^{-\Phi}\mathcal{D} e^{\Phi}-\mathcal{D}-[\mathcal{D},\Phi])+(e^{-\Phi} \omega\cdot\partial_{\theta}\Phi e^{\Phi}- \omega \cdot\partial_{\theta} \Phi)  \\
&= \mathcal{P}_1+\mathcal{P}_2+\mathcal{P}_3,  \\
\end{split}
\end{equation}
\begin{equation}
\mathcal{H}_2=\prod^1_{i=m}\Psi^{-1}_i[\mathcal{K}_{m+1}]\prod^{m}_{i=1}\Psi_i.
\end{equation}
Considering $\mathcal{P}_1$, by \cite[Lemma 5.3]{Bambusi1}, we get
\begin{equation}\label{p1}
\begin{split}
|\mathcal{P}_1|^{\varsigma,Lip}_{s'_m-2\sigma_m,s_0} &\lessdot|\mathcal{R}|^{\varsigma,Lip}_{s'_m,s_0} |\Phi|^{\varrho,Lip}_{s'_m-2\sigma_m,s_0} \\
&\leq \varepsilon^{\frac{9}{5}}_m.
\end{split}
\end{equation}
Considering $\mathcal{P}_2$, from the homological equation $[\mathcal{D},\Phi]=-\omega \cdot\partial_{\theta}\Phi-(\mathcal{R}-\mathrm{diag}[\mathcal{R}])$, we have
\begin{equation}
\begin{split}
|[\mathcal{D},\Phi]|^{\varsigma,Lip}_{s'_m-3\sigma_m,s_0} &\leq \frac{c(v)}{\sigma_m}|\Phi|^{\varsigma,Lip}_{s'_m-2\sigma_m,s_0}+|\mathcal{R}|^{\varsigma,Lip}_{s'_m-3\sigma_m,s_0}\\
&\leq \varepsilon^{\frac{2}{3}}_m.
\end{split}
\end{equation}
Moreover,
\begin{equation}
\begin{split}
|[[\mathcal{D},\Phi],\Phi]|^{\varsigma,Lip}_{s'_m-3\sigma_m,s_0} &\lessdot|[\mathcal{D},\Phi]|^{\varsigma,Lip}_{s'_m-3\sigma_m,s_0} |\Phi|^{\varrho,Lip}_{s'_m-3\sigma_m,s_0} \\
&\leq   \varepsilon^{\frac{8}{5}}_m.
\end{split}
\end{equation}
By the following formula
\begin{equation}
e^{-\Phi}\mathcal{D} e^{\Phi}-\mathcal{D}-[\mathcal{D},\Phi]=\int_0^1\int_0^s e^{-s_1\Phi}[[\mathcal{D},\Phi],\Phi]e^{s_1\Phi}ds_1ds,
\end{equation}
we get
\begin{equation}\label{p2}
|\mathcal{P}_2|^{\varsigma,Lip}_{s'_m-3\sigma_m,s_0} \leq \frac{1}{3} \varepsilon^{\frac{4}{3}}_m.
\end{equation}
Considering $\mathcal{P}_3$, by\cite[Lemma4.3]{Bambusi1}, we have
\begin{equation}\label{p3}
\begin{split}
|\mathcal{P}_3|^{\varsigma,Lip}_{s'_m-3\sigma_m,s_0} &\lessdot
|\Phi|^{\varrho,Lip}_{s'_m-3\sigma_m,s_0} |\omega \cdot\partial_{\varphi}\Phi|^{\varsigma,Lip}_{s'_m-3\sigma_m,s_0} \\
&\leq \frac{c}{\sigma_m}(|\Phi|^{\varrho,Lip}_{s'_m-2\sigma_m,s_0})^2 \\
&\leq \frac{c}{\sigma_m} \varepsilon^{\frac{10}{6}}_m.
\end{split}
\end{equation}
The bounds \eqref{p1}, \eqref{p2} and \eqref{p3} imply
\begin{equation}
|\mathcal{H}_1|^{\varsigma,Lip}_{s'_m-2\sigma_m,s_0} \leq  \varepsilon^{\frac{4}{3}}_m.
\end{equation}
As we see in Corollary \ref{4.20},
$$\max\{|\Omega^{-1}_m|^{\varrho,Lip}_{s'_m-2\sigma_m,s_0},|\Omega_m|^{\varrho,Lip}_{s'_m-2\sigma_m,s_0}\} \leq 2,$$
then,

\begin{equation}
\begin{split}
|\mathcal{H}_2|^{\varsigma,Lip}_{s'_m-2\sigma_m,s_0} &\lessdot  |\Omega^{-1}_m|^{\varrho,Lip}_{s'_m-2\sigma_m,s_0}|\Omega_m|^{\varrho,Lip}_{s'_m-2\sigma_m}|\mathcal{K}|^{\varsigma,Lip}_{s'_m,s_0}
 \\
&\leq C\varepsilon^{\frac{4}{3}}_m.
\end{split}
\end{equation}
Finally,
\begin{equation}
|\mathcal{R}^+|^{\varsigma,Lip}_{s'_m-2\sigma_m,s_0} \leq  (C+1) \varepsilon^{\frac{4}{3}}_m.
\end{equation}
\end{proof}

\subsection{The $\varepsilon_{n+1}$ approximate solution of linear equation $F(u_{n})+\mathcal{L}(u_n)v=0$}\

After $n$-step iteration, the linear operator $\mathfrak{L}(u_n)$ has been transformed into
\begin{equation}
\mathfrak{L}_{n+1}=\omega \cdot \partial_\theta\mathbf{1}+\mathcal{D}_{n+1}+\mathcal{R}_{n+1},
\end{equation}
where $\mathcal{R}_{n+1}$ is relatively small linear operator with $\|\mathcal{R}_{n+1}\|^{Lip}_{s'_n-2\sigma_n} \leq \varepsilon_{n+1}$. Now, the main concern is the  invertibility of $\mathcal{J}_{n+1}=\omega \partial_\varphi. \textbf{1}+\mathcal{D}_{n+1}$.
\begin{lemma} \label{6.001}
For all $g \in \mathrm{H}^{\mathfrak{s}}_{s'_n-2\sigma_n,s_0}$  with zero space average and $\lambda \in \Lambda_{n,n} \cap \Gamma(u_n) $,
\begin{equation}
\Gamma(u_n):=\left \{ \lambda :|\lambda\bar{\omega} \cdot \ell+d^{n+1}_k(u_n)| \geq {\alpha_{nn}} \frac{k^5}{[\ell]^\tau},  \forall \ell \in \mathbb{Z}^n , k \in \mathbb{Z} \backslash \{0\} \right \},
\end{equation}
the equation $\mathcal{J}_{n+1} v=g$  has a unique solution $v$ with zero space average and satisfies
\begin{equation}
\|v\|^{\mathfrak{s},Lip}_{s'_n-4\sigma_n,s_0} \leq \varepsilon^{-\frac{1}{6}}_{n+1}\|g\|^{\mathfrak{s},Lip}_{s'_n-2\sigma_n,s_0}.
\end{equation}
\end{lemma}
\begin{proof}
Since $\mathcal{J}$ is a diagonal linear operator and $g \in \mathrm{H}^1_0$, the equation $\mathcal{J} v=g$ can be transformed to
\begin{equation}\label{6.002}
-\mathbf{i}\omega \cdot \partial_{\theta}v_i+d^{n+1}_iv_i+\mu^{n+1}_i(\theta)v_i(\theta)=-\mathbf{i}g_i(\theta), \quad i\in \mathbb{Z}\backslash {0},
\end{equation}
 where $d^{n+1}_i\geq \frac{1}{2} i^5$, $^{l_1}|\mu^{n+1}_i(\varphi)|_{s_n-2\sigma_n,\tau} \leq 2\varepsilon i^3$.

 Applying Kuksin's lemma to \eqref{6.002}, we get
\begin{equation}
\|v_i\|_{s'_n-3\sigma_n,s_0} \leq \frac{c e^{2(5/\sigma_n)^{1/\beta}}}{\alpha_{nn} \chi_i\sigma_n^{\tau+2n+s_0+3}} \|g_i\|_{s'_n-2\sigma_n,s_0}.
\end{equation}
Since
\begin{equation}
\beta \geq \frac{2}{3}, \quad \sigma_n=\frac{1}{200}s_n =\frac{1}{200}{(\frac{10}{11})}^{n-1}s,
\end{equation}
we have
\begin{equation}
e^{2(5/\sigma_n)^{1/\beta}}\leq {c(s)^{(\frac{11}{10})}}^{\frac{3(n-1)}{2}} < {c(s)^{(\frac{4}{3})}}^n,
\end{equation}
$c$ is an constant only depending on $s$. Let $\varepsilon^{\frac{1}{20}}\leq c(s)^{-1}$,

we have
\begin{equation}\label{6.003}
\|v_i\|_{s'_n-3\sigma_n,s_0} \leq \frac{c \varepsilon^{-\frac{1}{20}}_{n+1}}{\alpha_n \sigma_n^{\tau+2v+s_0+3}} \|g_i\|_{s'_n-2\sigma_n,s_0}.
\end{equation}
Applying the operator $\Delta v=v(\lambda_1)-v(\lambda_2)$ to \eqref{6.002}, one gets
\begin{eqnarray}
-\mathbf{i} \lambda_1 \overline{\omega} \dots \partial_{\theta}\Delta v_i&+& d^{n+1}_i(\lambda_1)\Delta v_i+\mu^{n+1}_i(\theta)(\lambda_1) \Delta v_i(\theta)\nonumber\\
&=&-\mathbf{i}\Delta g_i(\varphi)-(\Delta d^{n+1}_i+ \Delta \mu^{n+1}_i) v_i(\lambda_2)+\mathbf{i}\Delta\lambda \cdot \overline{\omega} \partial_{\theta} v_i(\lambda_2)\nonumber\\
\end{eqnarray}
Again applying Kuksin's lemma, we have
\begin{equation}\label{6.004}
\|\Delta v_i\|_{s'_n-4\sigma_n,s_0} \leq \frac{\varepsilon^{-1/10}_{n+1}}{\alpha^2_{nn} \sigma_n^{4n+2\tau+2s_0+7}\chi_{i}} (\|\Delta g_i\|_{s'_n-2\sigma_n}+|\Delta \lambda| \|g_i\|_{s'_n-2\sigma_n,s_0})
\end{equation}
Finally,  \eqref{6.003} and \eqref{6.004} imply
\begin{equation}
\|v\|^{\mathfrak{s},Lip}_{s'_n-4\sigma_n,s_0} \leq \varepsilon^{-\frac{1}{6}}_{n+1}\|g\|^{\mathfrak{s},Lip}_{s'_n-2\sigma_n,s_0}
\end{equation}
\end{proof}

Now,we have conjugated the linearized operator $\mathcal{L}$ to $$\mathfrak{L}_{n+1}=\mathcal{J}_{n+1}+\mathcal{R}_{n+1}=\mathcal{W}^{-1}_1 \mathcal{L } \mathcal{W}_2,$$
where
$$\mathcal{W}_2=\mathcal{A} \mathcal{B} \Omega_n,  \quad   \quad  \mathcal{W}^{-1}_1=\Omega^{-1}_n \frac{1}{\xi(\theta)} \mathcal{B}^{-1} \mathcal{A}^{-1}. $$
Also, we can see  $\mathcal{W}^{\pm1}_i$ are linear maps of the subspace of $\mathrm{H}^{1}_{0}$.

Now, we can prove the first part  of the iteration lemma $\mathbf{[H1]_n}$.
\begin{lemma}For all $\lambda \in \Lambda_{n,n} \cap \Gamma(u_n)$, the linear operator $\mathcal{W}_1\mathcal{J}\mathcal{W}^{-1}_2$ admits a right inverse of $\mathrm{H}^1_0$. More precisely. for all Lipschitz family $F(\lambda) \in \mathrm{H}^1_0$, the function
\begin{equation}
v:=(\mathcal{W}_1\mathcal{J}\mathcal{W}^{-1}_2)^{-1}F:=\mathcal{W}_2 \mathcal{J}^{-1}\mathcal{W}^{-1}_1F
\end{equation}
is a solution of $\mathcal{W}_1\mathcal{J}\mathcal{W}^{-1}_2 v=F$. Morover
\begin{equation}
\|v_{n+1}\|^{Lip}_{s_{n+1},p+\eta} \leq \varepsilon^{-\frac{1}{5}}_{n+1} \|F(u_n)\|^{Lip} _{s_n,p+\eta-5}.
\end{equation}
\end{lemma}
\begin{proof}

We have already get
\begin{equation}
 v_{n+1} =\mathcal{W}_2 \mathcal{J}^{-1}\mathcal{W}^{-1}_1 F(u_n).
\end{equation}
Applying Lemma \ref{3.301} and Corollary \ref{4.20}, we see
\begin{equation}\label{6.05}
\begin{split}
\|\mathcal{W}^{-1}_1 F(u_n)\|^{\mathfrak{s},Lip}_{s'_n-2\sigma_n,s_0}&=\|\Omega^{-1}_n \mathcal{U}^{-1}_2 F(u_n)\|^{\mathfrak{s},Lip}_{s'_n-2\sigma_n,s_0} \\
& \lessdot\|\mathcal{U}^{-1}_2F(u_n)\|^{\mathfrak{s},Lip}_{s'_n,s_0} \\
&\lessdot\|\mathcal{U}^{-1}_2F(u_n)\|^{Lip}_{s'_n,2s_0}\\
& \lessdot\|F(u_n)\|^{s,Lip}_{s_n ,p+\eta-5}
\end{split}
\end{equation}
Using  Lemma \ref{6.001},  one gets
\begin{equation}\label{6.051}
\begin{split}
\|\mathcal{J}^{-1}\mathcal{W}^{-1}_1 F(u_n)\|^{\mathfrak{s},Lip}_{s'_n-4\sigma_n,s_0} &\lessdot \varepsilon^{-\frac{1}{6}}_{n+1}\|\mathcal{W}^{-1}_1 F(u_n)\|^{\mathfrak{s},Lip}_{s'_n-2\sigma_n,s_0}\\
 &\lessdot \varepsilon^{-\frac{1}{6}}_{n+1} \|F(u_n)\|^{\mathfrak{s},Lip} _{s_n,p+\eta-5}.
\end{split}
\end{equation}
With reference to Corollary \ref{4.20} and Lemma \ref{estimation2}, we see
\begin{equation}
\begin{split}
\|\Omega_n \mathcal{J}^{-1}\mathcal{W}^{-1}_1 F(u_n)\|^{Lip}_{s'_n-5\sigma_n,p+2\eta} &\lessdot \frac{1}{\sigma^{2\eta+s_0+5}_n}\|\Omega_n \mathcal{J}^{-1}\mathcal{W}^{-1}_1 F(u_n)\|^{Lip}_{s'_n-4\sigma_n,s_0}\\
&\lessdot \frac{1}{\sigma^{2\eta+s_0+5}_n}\|\Omega_n \mathcal{J}^{-1}\mathcal{W}^{-1}_1 F(u_n)\|^{\mathfrak{s},Lip}_{s'_n-4\sigma_n,s_0}\\
& \lessdot \frac{1}{\sigma^{2\eta+3s_0+5}_n}\|\mathcal{J}^{-1}\mathcal{W}^{-1}_1 F(u_n)\|^{\mathfrak{s},Lip}_{s'_n-4\sigma_n,s_0}\\
& \lessdot
\frac{\varepsilon^{-\frac{1}{6}}_{n+1}}{ \sigma_n^{p+2\eta+s_0}} \|F(u_n)\|^{Lip} _{s_n,p+\eta-5}.
\end{split}
\end{equation}
Using Lemma \ref{3.301} again, we get
\begin{equation}
\begin{split}
\|\mathcal{W}_2 \mathcal{J}^{-1}\mathcal{W}^{-1}_1 F(u_n)\|^{Lip}_{k_1(s'_n-5\sigma_n),p+\eta}&\lessdot \|\Omega_n \mathcal{J}^{-1}\mathcal{W}^{-1}_1 F(u_n)\|^{Lip}_{s'_n-5\sigma_n,p+2\eta}\\
 &\lessdot \frac{\varepsilon^{-\frac{1}{6}}_{n+1}}{\sigma_n^{p+2\eta+5}} \|F(u_n)\|^{Lip}_{s_n,p+\eta-5}.
\end{split}
\end{equation}
Since $k_1(s'_n-5\sigma_n) > s_{n+1}$, we get
\begin{equation}
\|v_{n+1}\|^{Lip}_{s_{n+1},p+\eta} \leq \varepsilon^{-\frac{1}{5}}_{n+1} \|F(u_n)\|^{Lip} _{s_n,p+\eta-5},
\end{equation}
Thus,
\begin{equation}
\|v_{n+1}\|^{Lip}_{s_{n+1},p+\eta} \leq \varepsilon_{n+1}.
\end{equation}
\end{proof}

\subsection{The estimation of $\mathbf{F(u_{n+1})}$ and $ \mathbf{v_{n+1}}$}\

Review the  definition of $F(u)$, which is
$$F(u)=u_t+\partial^5_{x}+10u\partial_x^3u+20\partial_xu\partial_x^2u+30u^2 \partial_xu-6\partial^2_xu \partial^5_xu-18\partial^3_xu\partial^4_xu- \partial_xf(\omega t,x). $$
\begin{lemma}\label{7.00} Assume $u_{n+1}=u_{n}+v$, that $v$ is the solution of $\mathcal{W}_1\mathcal{J}\mathcal{W}^{-1}_2 v=F(u_n)$. Then,
\begin{equation}
\begin{split}
F(u_{n+1})=&\mathcal{W}_1\mathcal{R}\mathcal{W}^{-1}_2(v)\\
&+10v\partial_x^3v+20\partial_xv\partial_x^2v+30v^2 \partial_xv-6\partial^2_xv \partial^5_xv-18\partial^3_xv\partial^4_xv+60u_{n}v\partial_xv.
\end{split}
\end{equation}
\end{lemma}
\begin{proof}
\begin{equation}\label{4.24}
\begin{split}
F(u_{n+1})=&F(u_{n})+\mathcal{L}(u_n)v \\
&-10v\partial_x^3v+20\partial_xv\partial_x^2v-30v^2 \partial_xv-6\partial^2_xv \partial^5_xv-18\partial^3_xv\partial^4_xv+60u_{n}v\partial_xv  \\
=&F(u_n)+\mathcal{W}_1\mathcal{J} \mathcal{W}^{-1}_2(v)+\mathcal{W}_1\mathcal{R}\mathcal{W}^{-1}_2(v) \\
&+10v\partial_x^3v+20\partial_xv\partial_x^2v-30v^2 \partial_xv-6\partial^2_xv \partial^5_xv-18\partial^3_xv\partial^4_xv+60u_{n}v\partial_xv  \\
=&\mathcal{W}_1\mathcal{R}\mathcal{W}^{-1}_2(v) \\
&+10v\partial_x^3v+20\partial_xv\partial_x^2v-30v^2 \partial_xv-6\partial^2_xv \partial^5_xv-18\partial^3_xv\partial^4_xv+60u_{n}v\partial_xv \\
\end{split}
\end{equation}
\end{proof}
 Now, the whole necessary estimates has been prepared. We will prove the last piece $(\mathcal{P}2)_n$ of iteration Lemma $\mathbf{[H1]_n}$.
\begin{proof}
Consider the formula \eqref{4.24} for $F(u_n)$, an approximate of $\mathcal{W}_1\mathcal{R}\mathcal{W}^{-1}_2(v_{n+1})$ is the our main concern. Since $\mathcal{W}^{-1}_2v_{n+1}=\mathcal{J}^{-1}\mathcal{W}^{-1}_1 F(u_n)$, by \eqref{6.051}, we see
\begin{equation}
\|\mathcal{W}^{-1}_2v_{n+1}\|^{\mathfrak{s},Lip}_{s'_n-4\sigma_n,s_0} =\|\mathcal{J}^{-1}\mathcal{W}^{-1}_1 F(u_n)\|^{\mathfrak{s},Lip}_{s'_n-4\sigma_n,s_0} \lessdot \varepsilon^{-\frac{1}{6}}_{n+1} \|F(u_n)\|^{Lip} _{s_n,p+\eta-5}.
\end{equation}
Then, by Lemma \ref{estimation2}, we have
\begin{equation}
\|\mathcal{W}^{-1}_2v_{n+1}\|^{Lip}_{s'_n-5\sigma_n,2s_0+1} \lessdot \frac{1}{\sigma^{s_0+1}_n} \|\mathcal{W}^{-1}_2v_{n+1}\|^{Lip}_{s'_n-4\sigma_n,s_0} \lessdot \frac{1}{\sigma^{s_0+1}_n} \|\mathcal{W}^{-1}_2v_{n+1}\|^{\mathfrak{s},Lip}_{s'_n-4\sigma_n,s_0}.
\end{equation}
Using \eqref{4.04}, we have
\begin{equation}
\begin{split}
\| \mathcal{R} W^{-1}_2v_{n+1}\|^{Lip}_{s'_n-5\sigma_n,s_0-2 } &\lessdot \varepsilon_{n+1} \|\mathcal{W}^{-1}_2v_{n+1}\|^{Lip}_{s'_n-5\sigma_n,2s_0+1}\\
&\lessdot \frac{\varepsilon^{\frac{5}{6}}_{n+1}}{\sigma_n^{s_0+1}}\|F(u_n)\|^{Lip} _{s_n,p+\eta-5}\\
\end{split}
\end{equation}
By Corollary \ref{4.20} and Lemma \ref{estimation2}, one gets
\begin{equation}
\begin{split}
\|\Omega_n \mathcal{R}\mathcal{ W}^{-1}_2v_{n+1}\|^{Lip}_{s'_n-7\sigma_n,p+2\eta }& \lessdot \frac{1}{\sigma^{s_0+2\eta+5}_n} \|\Omega_n \mathcal{R}\mathcal{ W}^{-1}_2v_{n+1}\|^{Lip}_{s'_n-6\sigma_n,s_0}\\
&\lessdot \frac{1}{\sigma^{s_0+2\eta+5}_n} \|\Omega_n\mathcal{R} \mathcal{W}^{-1}_2v_{n+1}\|^{\mathfrak{s},Lip}_{s'_n-6\sigma_n,s_0 }\\
 &\lessdot \frac{1}{\sigma^{s_0+2\eta+5}_n} \|\mathcal{R} \mathcal{W}^{-1}_2v_{n+1}\|^{\mathfrak{s},Lip}_{s'_n-6\sigma_n,s_0 }\\
 &\lessdot \frac{1}{\sigma^{s_0+2\eta+5}_n} \|\mathcal{R} \mathcal{W}^{-1}_2v_{n+1}\|^{Lip}_{s'_n-6\sigma_n,2s_0 }\\
 &\lessdot \frac{1}{\sigma^{2s_0+2\eta+7}_n} \|\mathcal{R} \mathcal{W}^{-1}_2v_{n+1}\|^{Lip}_{s'_n-5\sigma_n,s_0-2 }\\
&\lessdot \frac{\varepsilon^{\frac{5}{6}}_{n+1}}{\sigma^{3s_0+2\eta+8}}\|F(u_n)\|^{Lip} _{s_n,p+\eta-5}.
\end{split}
\end{equation}

By Lemma \ref{3.301}, we see
\begin{equation}\label{7.002}
\|\mathcal{W}_1 \mathcal{R} \mathcal{W}^{-1}_2v_{n+1}\|_{k_1(s'_n-7\sigma_n),p+\eta } \lessdot \frac{\varepsilon^{\frac{5}{6}}_{n+1}}{\sigma_n^{3s_0+2\eta+8}} \|F(u_n)\|^{Lip} _{s_n,p+\eta-5}.
\end{equation}
Obviously, $k_1(s'_n-7\sigma_n)> s_{n+1}$. Finally, Combining \eqref{7.002} with the estimate of the rest part of $F(u_{n+1})$, we get
\begin{equation}
\begin{split}
\|F(u_{n+1})\|^{Lip}_{s_{n+1},p+\eta-5} &\lessdot(\frac{\varepsilon^{\frac{5}{6}}_{n+1}}{\sigma^{4s_0+2\eta+8}_n} \|F(u_n)\|^{Lip} _{s_n,p+\eta-5}+(\|v_{n+1}\|^{Lip}_{s_{n+1},p+\eta})^2 )\\
&\leq \varepsilon^{\frac{6}{5}}_{n+2}.
\end{split}
\end{equation}
\end{proof}
\section{Measure estimation}
For notational convenience, we extend the eigenvalues $d^{m}_i(u_n)$, which are defined for $i \in \mathbb{Z}\backslash\{0\} $,to $i\in \mathbb{Z}$, where $d^{m}_i(u_n)=0,i=0$.

Set $\Theta_{mn}=\bigcup_{i,j,\ell}R^m_{ij\ell}(u_n), i,j \neq 0 $, where
\begin{equation}
 R^m_{ij\ell}(u_n)=\{ \lambda \in \Pi:|\ell\cdot\lambda\overline{\omega}+d^{m}_i(u_n)-d^{m}_j(u_n)| \leq \frac{\alpha_{mn}|i^5-j^5|}{[\ell]^\tau} ,i \neq j ,m \leq n \}.
\end{equation}

Set $\Gamma_{n}=\bigcup_{i,\ell}R^{n+1}_{i,0,\ell}(u_n)=\bigcup_{i,\ell}R_{i\ell}(u_n)$, where
\begin{equation}
 R_{i\ell}(u_n)=\{ \lambda \in \Pi:|\ell\cdot\lambda\overline{\omega}+d^{n+1}_i(u_n)| \leq \frac{\alpha_{nn}|i^5|}{[\ell]^\tau}  \}.
\end{equation}

 Although $\Theta_{mn}$ and $\Gamma_{n}$  seem different, the following two lemmas can applying them both.

\begin{lemma}\label{lemma5.1}
If $R^m_{ij\ell}(u_n)\neq\emptyset $, then $|i^5-j^5|\leq 8|\lambda\overline{\omega}\cdot \ell|$
\end{lemma}
\begin{proof}If $R^m_{ij\ell}(u_n)\neq\emptyset $, then there exists $\lambda \in \Pi$, such that $$|\lambda\overline{\omega}\cdot \ell + d^m_i(u_n)-d^m_j(u_n)| < \frac{\alpha_{mn}|i^5-j^5|}{[\ell]^\tau} .$$
Therefor,
\begin{equation}
|d^m_i-d^m_j|<\frac{\alpha_{mn}|i^5-j^5|}{|\ell|^\tau}+2|\overline{\omega}\cdot l|.
\end{equation}
Moreover, by \eqref{4.13}, for $\varepsilon$ small enough,
\begin{equation}
|d^m_i-d^m_j|  \geq \frac{1}{2}|i^5-j^5|.
\end{equation}
Since $\alpha_{mn} \leq \alpha_0$, we see
\begin{equation}
2|\overline{\omega}\cdot \ell|\geq(\frac{1}{2}-\frac{\alpha_{0}}{[\ell]^\tau})|i^5-j^5| \geq \frac{1}{4} |i^5-j^5|,
\end{equation}
which prove the lemma.
\end{proof}

\begin{lemma} $|R^m_{ij\ell}(u_n)| \leq 9 \frac{\alpha_{mn}}{[\ell]^\tau} $.
\end{lemma}
\begin{proof}
Consider the function $\phi(\lambda):\Pi \mapsto C$ defined by
\begin{equation}
\phi(\lambda)=\lambda\overline{\omega}\cdot \ell + d^m_i(u_n)-d^m_j(u_n),
\end{equation}
where
\begin{equation}
|d^m_i(u_n)-d^m_j(u_n)|^{lip} \leq m^{lip}(u_n)|i^5-j^5|+C_{\lambda,m}|i^3-j^3| \leq C\varepsilon|i^5-j^5|.
\end{equation}
Since $\varepsilon$ is small enough, for any $ \lambda_1,\lambda_2 \in \Pi $, we see
\begin{equation}
\begin{split}
|\phi(\lambda_1)-\phi(\lambda_2)|& \geq (\lambda_1-\lambda_2)(|\overline{\omega}\cdot \ell|-|d^m_i(u_n)-d^m_j(u_n)|^{lip})\\
 &\geq(\frac{1}{8}- C\varepsilon_0)|i^5-j^5||\lambda_1-\lambda_2| \\
 &\geq \frac{|i^5-j^5|}{9}|\lambda_1-\lambda_2|.
\end{split}
\end{equation}
Then, one gets
\begin{equation}
|R^m_{ij\ell}(u_n)| \leq  \frac{\alpha_{mn} |i^5-j^5|}{[\ell]^\tau} \frac{9}{|i^5-j^5|} \leq \frac{9 \alpha_{mn} }{[\ell]^\tau}.
\end{equation}
\end{proof}
\begin{lemma}\label{lemma5.3}
Let $u_n(\lambda),u_{n-1}(\lambda)$ be Lipschitz families of analytic function, defined for $\lambda \in \Upsilon$. Then, for $v >0$, $\forall \lambda \in \Lambda_{v,n}\cap \Lambda_{v,n-1}$,
\begin{equation}
|\mathcal{R}_v(u_n)-\mathcal{R}_v(u_{n-1})|^\varsigma_{s'_v,s_0} \leq C \varepsilon_v\varepsilon_n.
\end{equation}
\end{lemma}
\begin{proof}
Obviously, for $v=1$, we have
\begin{equation}
\mathcal{R}_1(u_n)-\mathcal{R}_1(u_{n-1})=\mathcal{K}_1-\mathcal{K}_1=0.
\end{equation}
By induction method, for $v \leq m$, we have
\begin{equation}\label{5.17}
\begin{split}
&|(d^v_i(u_n)-d^v_j(u_{n}))-(d^v_i(u_{n-1})-d^v_j(u_{n-1}))|\\
\leq \text{  }& (m(u_n)-m(u_{n-1}))|i^5-j^5|+|(r_i(u_{n})-r_i(u_{n-1}))|i^3+|(r_j(u_{n})-r_j(u_{n-1}))|j^3  \\
\leq \text{  }& c\varepsilon_n|i^5-j^5|+c\sum_{k \leq v-1}|R_k(u_n)-R_k(u_{n-1})|^\varsigma_{s'_k,s_0}|i^3+j^3|\\
\leq \text{  }& c_1\varepsilon_n|i^5-j^5|
\end{split}
\end{equation}
and
\begin{equation}\label{5.18}
\begin{split}
&\|(\mu^v_i(u_n)-\mu^v_j(u_{n}))-(\mu^v_i(u_{n-1})-\mu^v_j(u_{n-1}))\|_{s'_v,s_0} \\
\leq \text{  }& c \sum_{k \leq v-1}|R_k(u_n)-R_k(u_{n-1})|^\varsigma_{s'_v,s_0}|i^3+j^3|\\
\leq \text{  }& c_2 \varepsilon_n|i^5-j^5|
\end{split}
\end{equation}

Now, we consider  $v=m+1$. Set $\Delta\Phi_{ij}=\Phi_{ij}(u_n)-\Phi_{ij}(u_{n-1})$ , and applying the operator $\Delta$ to \eqref{h solution}, we get
\begin{equation}\label{5.19}
\begin{split}
 -\mathbf{i} \omega \cdot  \partial_{\theta} ( \Delta\Phi^m_{ij}) &+ d_{ij}(u_n)(\Delta\Phi^m_{ij})+\mu_{ij}(u_n)(\Delta \Phi^m_{ij})  \\
&= -(\Delta d_{ij}+\Delta\mu_{ij})\Phi^m_{ij}(u_{n-1})-\mathbf{i}\Delta R^m_{ij}. \\
\end{split}
\end{equation}
Applying Kuksin's lemma to \eqref{5.19} again, we have
\begin{equation}
\|\Delta\Phi^m_{ij}\|_{s'_m-2\sigma_m,s_0} \lessdot \frac{\varepsilon^{-1/10}_m}{\alpha^2_{mn} \sigma_m^{4n+2\tau+s_0+7}\chi_{ij}}(\varepsilon_n\|R^m_{ij}\|_{s'_m,s_0}+\|\Delta R_{ij}^m\|_{{s'_m,s_0}}),
\end{equation}
which indicates
\begin{equation}
\|\Delta \Phi_m|^\varrho_{s'_m-2\sigma_m,s_0} \leq   \varepsilon_n  \varepsilon^{\frac{5}{6}}_m.
\end{equation}
Recall the definition of $\mathcal{R}_{m+1}$, we can get
\begin{eqnarray}
\Delta \mathcal{R}_{m+1}=\Delta \mathcal{P}_{1}+\Delta \mathcal{P}_{2}+\Delta \mathcal{P}_{3}+\Delta \mathcal{H}_{m+1}.
\end{eqnarray}
For notation convenience, we make a notation
\begin{equation}
|\Phi_m|^{\varrho}_{s'_m-2\sigma_m,s_0}=\max\{|\Phi(u_n)|^\varrho_{s'_m-2\sigma_m,s_0}, |\Phi(u_{n-1})|^\varrho_{s'_m-2\sigma_m,s_0}\}.
\end{equation}

By \eqref{2.33}, we see
\begin{equation}\label{5.20}
|\Delta\Psi_m|^{\varrho}_{s'_m-2\sigma_m,s_0} \leq C|\Delta\Phi_m|^{varrho}_{s'_m-2\sigma_m,s_0} \leq C \varepsilon_n\varepsilon^{\frac{5}{6}}_m,
\end{equation}
and
\begin{equation}\label{5.21}
|\Delta\Psi^{-1}_m|^{\varrho}_{s'_m-2\sigma_m,s_0} \leq C|\Delta\Phi_m|^{varrho}_{s'_m-2\sigma_m,s_0} \leq C \varepsilon_n \varepsilon^{\frac{5}{6}}_m.
\end{equation}
The detail of the estimation of $\Delta R_{m+1}$ can be divided into several parts.

\textbf{Part 1:} Firstly, we consider $\Delta \mathcal{H}_{m+1}$. Since $\Delta[\mathcal{K}_{m+1}]=0$, then, we get
\begin{equation}
\Delta \mathcal{H}_{m+1}=\Delta(\prod^1_{i=m}\Psi^{-1}_i)[\mathcal{K}_{m+1}]\prod^{m}_{i=1}\Psi_i +\prod^1_{i=m}\Psi^{-1}_i[\mathcal{K}_{m+1}]\Delta (\prod^{m}_{i=1}\Psi_i). \nonumber \\
\end{equation}
Using \eqref{2.33}, \eqref{5.20} and \eqref{5.21}, we have
\begin{equation}\label{H}
\begin{split}
|\Delta \mathcal{H}_{m+1}|_{s'_{m+1},s_0}^{\varsigma} \leq &|[\Delta\Omega^{-1}_m][\mathcal{K}_{m+1}][\Omega_m]|^{\varsigma}_{s'_{m+1},s_0}+|\Omega^{-1}_m][\mathcal{K}_{m+1}][\Delta\Omega_m]|^{\varsigma}_{s'_{m+1},s_0}\\
\lessdot& |\Delta(\prod^1_{i=m}\Psi^{-1}_i)|_{s'_m-2\sigma_m,s_0}^\varrho|[\mathcal{K}_{m+1}]|_{s'_{m+1},s_0}^\varsigma|\prod^{m}_{i=1}\Psi_i |_{s'_m-2\sigma_m,s_0}^\varrho\\
&+|\prod^1_{i=m}\Psi^{-1}_i|_{s'_m-2\sigma_m,s_0}^\varrho|[\mathcal{K}_{m+1}]|_{s'_{m+1},s_0}^\varsigma |\Delta (\prod^{m}_{i=1}\Psi_i)|_{s'_m-2\sigma_m,s_0}^\varrho \\
\leq & C \varepsilon_{m+1} \varepsilon_n.
\end{split}
\end{equation}

\par

\textbf{Part 2:} Considering $\Delta \mathcal{P}_1$, we have
\begin{equation}
\Delta \mathcal{P}_1=(e^{-\Phi_m}(\Delta\mathcal{R}_m)e^{\Phi_m}-\Delta\mathcal{R}_m)+(\Delta e^{-\Phi_m})\mathcal{R}_me^{\Phi_m}+ e^{-\Phi_m}\mathcal{R}_m(\Delta e^{\Phi_m}).
\end{equation}
Using \cite[lemma5.3]{Bambusi1} and \eqref{2.33}, we have
\begin{equation}
\begin{split}
|e^{-\Phi_m}(\Delta\mathcal{R}_m)e^{\Phi_m}-\Delta\mathcal{R}_m|^{\varsigma}_{s'_m-2\sigma_m,s_0} &\lessdot |\Delta\mathcal{R}_m|^{\varsigma}_{s'_m-2\sigma_m}|\Phi_m|^{\varrho}_{s'_m,s_0} \\
&\leq C \varepsilon_m \varepsilon_n\cdot  \varepsilon^{\frac{5}{6}}_m,
\end{split}
\end{equation}
and
\begin{equation}
\begin{split}
|(\Delta e^{-\Phi_m})\mathcal{R}_me^{\Phi_m}+ e^{-\Phi_m}\mathcal{R}_m(\Delta e^{\Phi_m})|^{\varsigma}_{s'_m-2\sigma_m,s_0} &\lessdot |\mathcal{R}_m|^{\varsigma}_{s'_m}|\Delta\Phi_m|^{\varrho}_{s'_m-2\sigma_m,s_0}\nonumber \\
&\leq C \varepsilon_m \cdot \varepsilon_n \varepsilon^{\frac{5}{6}}_m.
\end{split}
\end{equation}
Then, we have
\begin{equation}\label{P1}
|\Delta \mathcal{P}_1|^{\varsigma}_{s'_m-2\sigma_m,s_0} \lessdot \varepsilon^{\frac{11}{6}}_m \cdot \varepsilon_n \leq \frac{1}{3} \varepsilon^{\frac{4}{3}}_m \cdot \varepsilon_n.
\end{equation}
\par

\textbf{Part 3:} Consider $\Delta \mathcal{P}_2$, we have
\begin{equation}
\Delta \mathcal{P}_2=\int_0^1\int_0^s \Delta[e^{-s_1\Phi_m}[[\mathcal{D}_m,\Phi_m],\Phi_m]e^{s_1\Phi_m}]ds_1ds,
\end{equation}
where
\begin{equation}\label{5.30}
\begin{split}
\Delta[e^{-s_1\Phi_m}[[\mathcal{D}_m,\Phi_m],\Phi_m]e^{s_1\Phi_m}]=&(\Delta e^{-s_1\Phi_m})[[\mathcal{D}_m,\Phi_m],\Phi_m]e^{s_1\Phi_m}\\
&+ e^{-s_1\Phi_m}[[\mathcal{D}_m,\Phi_m],\Phi_m](\Delta e^{s_1\Phi_m}) \\
&+e^{-s_1\Phi_m} \Delta[[\mathcal{D}_m,\Phi_m],\Phi_m]e^{s_1\Phi_m}.
\end{split}
\end{equation}
From the homological equation \eqref{h solution}, we see $\Delta[\mathcal{D},\Phi]=-\omega \cdot\partial_{\varphi}\Delta\Phi-(\Delta\mathcal{R}-\Delta\mathrm{diag}[\mathcal{R}])$. Then, we get
\begin{equation}\label{5.31}
\begin{split}
|\Delta[\mathcal{D}_m,\Phi_m]|^{\varsigma}_{s'_m-3\sigma_m,s_0} &\leq \frac{c(v)}{\sigma_m}\varepsilon_n\cdot \varepsilon^\frac{5}{6}_m+c\varepsilon_n\cdot\varepsilon_m\nonumber \\
&\leq \varepsilon_n \varepsilon^{\frac{2}{3}}_m,
\end{split}
\end{equation}
and
\begin{equation}\label{5.32}
\begin{split}
|\Delta[[\mathcal{D}_m,\Phi_m],\Phi_m]|^{\varsigma}_{s'_m-3\sigma_m,s_0} &\lessdot |\Delta[\mathcal{D}_m,\Phi_m]|^{\varsigma}_{s'_m-3\sigma_m,s_0}|\Phi_m|^{\varrho}_{s'_m-2\sigma_m,s_0}\\
&\quad +|[\mathcal{D}_m,\Phi_m]|^{\varsigma}_{s'_m-3\sigma_m,s_0}|\Delta\Phi_m|^{\varrho}_{s'_m-2\sigma_m,s_0} \\
&\lessdot \varepsilon_n\varepsilon^{\frac{2}{3}}_m \cdot \varepsilon^{\frac{5}{6}}_m+\varepsilon^{\frac{2}{3}}_m \cdot \varepsilon_n \varepsilon^{\frac{5}{6}}_m. \\
\end{split}
\end{equation}
By \eqref{2.33} , we see
\begin{equation}\label{5.33}
\begin{split}
&|(\Delta e^{-s_1\Phi_m})[[\mathcal{D}_m,\Phi_m],\Phi_m]e^{s_1\Phi_m}+ e^{-s_1\Phi_m}[[\mathcal{D}_m,\Phi_m],\Phi_m](\Delta e^{s_1\Phi_m})|^{\varsigma}_{s'_m-2\sigma_m,s_0}\\
\lessdot \text{  }& |\Delta\Phi_m|^{\varrho}_{s'_m-2\sigma_m,s_0} |[[\mathcal{D}_m,\Phi_m],\Phi_m]|^{\varsigma}_{s'_m-3\sigma_m,s_0} \\
\lessdot \text{  }& \varepsilon^{\frac{5}{6}}_m \varepsilon_n\cdot \varepsilon^{\frac{5}{3}}_m.
\end{split}
\end{equation}
The bounds \eqref{5.30}, \eqref{5.32} and \eqref{5.33} imply
\begin{equation}\label{P2}
|\Delta \mathcal{P}_2|^{\varsigma}_{s'_m-3\sigma_m,s_0} \leq \frac{1}{3}\varepsilon^{\frac{4}{3}}_m \varepsilon_n.
\end{equation}
\par
\textbf{Part 4:} For $\Delta\mathcal{P}_3$, we see
\begin{equation}
\begin{split}
\Delta\mathcal{P}_3=&(e^{-\Phi_m} (\omega\cdot\partial_{\varphi} \Delta \Phi_m) e^{\Phi_m}- \omega \cdot\partial_{\varphi} \Delta \Phi_m)+(\Delta(e^{-\Phi_m}) \omega\cdot\partial_{\varphi} \Phi_m e^{\Phi_m}) \nonumber \\
&+(e^{-\Phi_m} \omega\cdot\partial_{\varphi} \Phi_m \Delta(e^{\Phi_m})).
\end{split}
\end{equation}
Using \cite[lemma5.3]{Bambusi1} and  \eqref{2.33} again, one gets
\begin{equation}
\begin{split}
|e^{-\Phi_m}(\omega\cdot\partial_{\varphi} \Delta \Phi_m) e^{\Phi_m}- \omega \cdot\partial_{\varphi} \Delta \Phi_m|^{\varsigma}_{s'_m-3\sigma_m,s_0}&\lessdot |\Phi_m|^{\varrho}_{s'_m-3\sigma_m,s_0} |\omega \cdot\partial_{\varphi}\Delta\Phi_m|^{\varsigma}_{s'_m-3\sigma_m} \\
&\leq \frac{c}{\sigma_m}|\Delta \Phi|^{\varrho}_{s'_m-2\sigma_m,s_0}|\Phi|^{\varrho,Lip}_{s'_m-2\sigma_m,s_0} \\
&\leq \frac{c}{\sigma_m}\varepsilon^{\frac{5}{6}}_m\varepsilon_n \cdot \varepsilon^{\frac{5}{6}}_m,
\end{split}
\end{equation}

and
\begin{equation}
\begin{split}
  &|\Delta(e^{-\Phi_m}) \omega\cdot\partial_{\varphi} \Phi_m e^{\Phi_m}+e^{-\Phi_m} \omega\cdot\partial_{\varphi} \Phi_m \Delta(e^{\Phi_m})|^{\varsigma}_{s'_m-3\sigma_m,s_0} \\
  \lessdot \text{  }&  C|\Delta\Phi_m|^{\varrho}_{s'_m-3\sigma_m,s_0} |\omega \cdot\partial_{\varphi}\Phi_m|^{\varsigma}_{s'_m-3\sigma_m,s_0}\\
\leq \text{  }&  \frac{c}{\sigma_m}|\Delta \Phi|^{\varrho}_{s'_m-2\sigma_m,s_0}|\Phi|^{\varrho,Lip}_{s'_m-2\sigma_m,s_0}\\
\leq \text{  }&  \frac{c}{\sigma_m}\varepsilon^{\frac{5}{6}}_m\varepsilon_n \cdot \varepsilon^{\frac{5}{6}}_m.
\end{split}
\end{equation}

 We can get
\begin{equation}\label{P3}
|\Delta \mathcal{P}_3|^{\varsigma}_{s'_m-3\sigma_m,s_0} \leq \frac{1}{3}\varepsilon^{\frac{4}{3}}_m \varepsilon_n.
\end{equation}
Finally, \eqref{H},\eqref{P1}, \eqref{P2} and \eqref{P3} imply
\begin{equation}
|\Delta \mathcal{R}_{m+1}|^\varsigma_{s'_{m+1},s_0} \leq C \varepsilon_{m+1} \varepsilon_n.
\end{equation}
Then, the lemma is proved.
\end{proof}

\begin{lemma}If $\varepsilon_0$ is small enough,  for any $m\leq n-1$ and $\ell$ satisfying
$$[\ell]^\tau \leq \varepsilon^{-\frac{3}{4}}_n\leq\frac {1} {\varepsilon_n 2^n},$$
 we have
$$ R^m_{ ij\ell}(u_n) \subseteq R^m_{ij\ell}(u_{n-1}).$$
\end{lemma}
\begin{proof} By \eqref{5.17}, for any $i,j \in \mathbb{Z}$, we have
\begin{equation}
|(d^m_i-d^m_j)(u_n)-(d^m_i-d^m_j)(u_{n-1})| \leq c|i^5-j^5|\varepsilon_n.
\end{equation}
Then,
\begin{equation}
\begin{split}
|\lambda\overline{\omega} \cdot \ell+(d^m_i-d^m_j)(u_n)| \geq& |\lambda\overline{\omega} \cdot \ell+(d^m_i-d^m_j)(u_{n-1})| \\
&-|(d^m_i-d^m_j)(u_n)-(d^m_i-d^m_j)(u_{n-1})|  \\
\geq & \frac{\alpha_0}{2^m}(1+\frac{1}{2^{n-1-m}})\frac{|i^5-j^5|}{[\ell]^\tau} - c|i^5-j^5|\varepsilon_n  \\
\geq &\frac{\alpha_0}{2^m}(1+\frac{1}{2^{n-1-m}})\frac{|i^5-j^5|}{[\ell]^\tau} - \alpha_0\frac{1}{2^n}\frac{|i^5-j^5|}{[\ell]^\tau} \\
\geq &\frac{\alpha_0}{2^m}(1+\frac{1}{2^{n-m}})\frac{|i^5-j^5|}{[\ell]^\tau} \\
 =&\alpha_{nm}\frac{|i^5-j^5|}{[\ell]^\tau} .
\end{split}
\end{equation}

\end{proof}
\begin{thm}The cantor like set $\Pi_{\varepsilon} \in \Pi $   is asymptotically full Lebesgue measure, i.e.
$$| \Pi \backslash \Pi_{\varepsilon}| \leq C \alpha_0 .$$
\end{thm}
\begin{proof} We see
$$\Pi \backslash \Pi_{\varepsilon}=(\bigcup\Gamma_n)\bigcup(\bigcup_{m,n}\Theta_{mn}),\quad  \forall m\leq n .$$
Obviously, we have $|\bigcup\Gamma_n|\leq C\alpha_0$.
Consider the set $\bigcup\Theta_{mn}$ in a different view. Set
$$\Lambda_m=\bigcup_{n\geq m}\Theta_{mn} =\bigcup_{i,j,l,n} R^m_{ijk}(u_n) ,$$
where $\Lambda_m$ is set removed from the $m$-step reduction for all $\mathfrak{L}(u_n)$,$n \geq m$.
By Lemma \ref{lemma5.1}, $R^m_{ij\ell}(u_n)\neq\emptyset $ are confined in the ball $i^4+j^4 \leq 16|\overline{\omega}||\ell|$. Then, we have
\begin{equation}
\begin{split}
|\Theta_{mm}|&=|\bigcup_{i,j,\ell} R^m_{ij\ell}(u_m)| \leq \sum_{\ell \in \mathbb{Z}^v} \sum_{i^4+j^4\leq  16|\overline{\omega}||\ell|} |R^m_{ij\ell}(u_m)|  \\
&\leq C \sum_{\ell \in \mathbb{Z}^v} \frac{\alpha_{mm}}{[\ell]^\tau}[\ell]^{\frac{1}{2}} \leq C {\alpha_{mm}}=C \frac{\alpha_{0}}{2^{m-1}},
\end{split}
\end{equation}
\begin{equation}
\begin{split}
|\Theta_{m,n}\backslash \Theta_{m,n-1}| &\leq C\sum_{[\ell]^\tau \geq  \varepsilon^{-\frac{3}{4}}_n, |i|,|j| \leq C |\ell|^{1/4}} \frac {\alpha_{mn}}{[\ell]^\tau}  \\
&\leq C\sum_{[\ell]^\tau \geq   \varepsilon^{-\frac{3}{4}}_n }\frac{\alpha_{mn}}{[\ell]^\tau}[\ell]^{\frac{1}{2}}  \\
&\leq C \frac{\alpha_0}{2^{m-1}}\varepsilon^{\frac{3}{4\tau}(\tau-\frac{1}{2}-v)}_n \\
&\leq C \frac{\alpha_0}{2^{m-1}}\varepsilon_n^{\frac{3}{8\tau}}.
\end{split}
\end{equation}
Then, we can get  $|\Lambda_m| \leq C \frac{\alpha_0}{2^{m-1}}$. The lemma is proved.
\end{proof}

\section{Technical lemmas }
Suppose the function in this paper real analytic on $\mathbb{T}^n_s$, and $s_0$ an integer greater than $\frac{n}{2}$.
\begin{defi}\label{6.00}
Let $p$ be an integer, the max norm of $D^pu$ on $\mathbb{T}^n_s$ is
$$|D^pu|_s=\sum_{\alpha\in \mathbb{Z}^n,|\alpha|=p}|D^\alpha u|_s.$$

\end{defi}

\begin{lemma}[\cite{Bogolyubov1}]\label{estimation0}
For $\sigma>0$ and $v>0$, the following inequalities holds :
\begin{equation}
\sum_{k \in \mathbb{Z}^n}e^{-2|k|\sigma} \leq \big(\frac{v}{e}\big)^v\frac{1}{\sigma^{v+n}}(1+e)^n
\end{equation}
\begin{equation}
\sum_{k \in \mathbb{Z}^n}e^{-2|k|\sigma}|k|^v \leq \big(\frac{v}{e}\big)^v\frac{1}{\sigma^{v+n}}(1+e)^n
\end{equation}
\end{lemma}
\begin{proof}The proof can be found in \cite[p22]{Bogolyubov1}.
\end{proof}

\begin{lemma}[Appendix A. \cite{KuKsin4}] \label{estimation1}
If $s \geq 0$ and $p > \frac{n}{2}$, then $\|uv(x)\|_{s,p}\leq c\|u(x)\|_{s,p}\|v(x)\|_{s,p}$ with a finite constant $c$  depending  on $p$ and $n$.
\end{lemma}
\begin{proof}For $n=1$, the detail of the proof can be found in  \cite{KuKsin4}. $n>1$ is a simple variation.
\end{proof}

\begin{lemma}\label{estimation2} For $u$ be analytic on $\mathbb{T}^n_s$, we have the following inequalities,
 \begin{equation}\label{6.01}
 \|u\|_{s-\sigma,p+\nu}\leq \frac {c(\nu)}{\sigma^{\nu}}\|u\|_{s,p},
 \end{equation}
and
  \begin{equation}\label{6.03}
 |u|_{s,p-s_0} \lessdot \|u\|_{s,p} \lessdot |u|_{s,p+s_0}.
 \end{equation}
 \end{lemma}
 \begin{proof}To prove \eqref{6.01}, we see

$$\|u\|^2_{s-\sigma,p+\nu}=\sum_{k\in \mathbb{Z}^n}|u_k|^2 e^{2|k|(s-\sigma)}[k]^{2(p+\nu)}=\sum_{k\in \mathbb{Z}^n}|u_k|^2 e^{2|k|s}[k]^{2p}(e^{-2|k|\sigma}[k]^{2\nu}).$$
Since $e^{-|k|\sigma}[k]^{\nu} \leq  e^{-\nu}(\frac{\nu}{\sigma})^{\nu}$, we get

$$\|u\|_{s-\sigma,p+\nu}\leq \frac {c(\nu)}{\sigma^{\nu}}\|u\|_{s.p}, \quad \quad c(\nu)=e^{-\nu}{\nu}^{\nu}.$$

Considering \eqref{6.03},  since $u$ is analytic on $\mathbb{T}^n_s$,  the Fourier coefficients $u_k$ satisfy
 $$|u_k|\leq |u|_s e^{-|k|s}.$$

   $D^\alpha u $ is also an analytic function on $\mathbb{T}^n_s$, the Fourier coefficients $(D^\alpha u)_k=u_k (\mathbf{i}k)^\alpha$, $ k^\alpha=k_1^{\alpha_1}\cdots k_n^{\alpha_n}$, satisfy
    $$|(D^\alpha u)_k| =|u_k| |(\mathbf{i}k)^\alpha| \leq |D^\alpha u|_s e^{-|k|s}.$$

   If $|\alpha|=p+s_0$,  by Definition \eqref{6.00}, we have
  $$|u_k| |k|^{p+s_0}= \sum_{|\alpha|=p+s_0} |u_k||k^{\alpha}| \leq |D^{p+s_0}u|_s e^{-|k|s}. $$
  Now, we have
  \begin{equation*}
  \begin{split}
\|u\|^2_{s,p} &=\sum_{k\in \mathbb{Z}^n}|u_k|^2 e^{2|k|s}[k]^{2p}   \\
&=  |u_0|^2 + \sum_{k\in \mathbb{Z}^n\setminus\{0\}}|u_k|^2 e^{2|k|s}|k|^{2p}    \\
&\leq |u|^2_s+\sum_{k\in \mathbb{Z}^n\setminus\{0\}}|D^{p+s_0} u|^2_s |k|^{-2s_0} \\
&= |u|^2_s+|D^{p+s_0} u|^2_s\sum_{k\in \mathbb{Z}^n\setminus\{0\}} |k|^{-2s_0} \nonumber \\
&\leq c |u|^2_{s,p+s_0}.
\end{split}
\end{equation*}.

To prove the left part of  \eqref{6.03}, we see
\begin{equation}
\begin{split}
|u|_{s,p-s_0} &\leq c_1 \Big\{\sum_{k\in \mathbb{Z}^n}|u_k|e^{|k|s}[k]^{p-s_0}\Big\}\\
&\leq c_1\big(\sum_{k\in \mathbb{Z}^n}|u_k|^2e^{2|k|s}[k]^{2p}\big)^{\frac{1}{2}}\big(\sum_{k\in \mathbb{Z}^n}[k]^{-2s_0}\big)^{\frac{1}{2}}\\
&\leq c \|u\|_{s,p}
\end{split}
\end{equation}
\end{proof}

\begin{lemma}[\cite{mcleod}] \label{mean value}
If $f$ is analytic from the segment joining $z_1$ and $z_0$ defined on $\mathbb{C}^n$ to $\mathbb{C}^n$. Then, there are point $w_1,w_2,\cdots w_{2n}$ on the segment such that
\begin{equation}
f(z_1)-f(z_0)=(z_1-z_0)(\lambda_1f'(w_1)+\lambda_2f'(w_2)+\cdots+\lambda_{2n}f'(w_{2n})),
\end{equation}
where $\lambda_i \geq 0$ and $\sum^{2n}_{i=1}\lambda_i=1$.
\end{lemma}
\begin{proof}
The detail of the proof can be found in \cite{mcleod}.
\end{proof}
\par

\begin{lemma}\label{estimation34}
$\mathbf{(Change \ of \  variable)}$ Let $f$ be a real analytic function on $\mathbb{T}^n_s$, with $|f|_{s.p} \leq \frac{1}{100}$.
Then, there are an constant $C$ depending on $n$ and $p$, such that

$(\mathbf{i}).$  If $u$ is  a real analytic function on $\mathbb{T}^n_s$,  $u(x+f(x))$ is also a real analytic function on $\mathbb{T}^n_{\frac{100s}{101}}$ and satisfies
  \begin{equation}\label{6.51}
   |u(x+f(x))|_{\frac{100s}{101},p} \leq C|u(x)|_{s,p}.
  \end{equation}

$(\mathbf{ii}).$ Considering another analytic function $g$ on $\mathbb{T}^n_s$ with  $|g|_{s.p+s_0} \leq \frac{1}{100}$,  we have
  \begin{equation}\label{6.52}
  |u(x+f(x))-u(x+g(x))|_{\frac{100s}{101},p} \leq C|u(x)|_{s,p+1} |f(x)-g(x)|_{s,p}.
  \end{equation}

$(\mathbf{iii}).$ Suppose  $u=u_\lambda$, $f=f_\lambda$ depend in a Lipschtiz way by a parameter $\lambda \in \pi \subset \mathbb{R}$, and $|f_\lambda|_{s,p+s_0} \leq \frac{1}{100}$, for all $\lambda$. Then, we have
  \begin{equation}\label{6.53}
  |u(x+f(x))|_{\frac{100s}{101},p} ^{Lip}\leq
   C |u(x)|_{s.p+1}^{Lip}(1+|f(x)|^{Lip}_{s,p}).
   \end{equation}

   \end{lemma}
   \begin{proof}
    For symbolic simplicity, the following calculations only consider $n=1$ in form. With regard to general situation, there is no difference.

   $(\mathbf{i})$ Set $v(x)=x+f(x)$. Clearly, $v(x)$ is real valued on $\mathbb{R}$.
   Now, we would prove $v(\mathbb{T}_{\frac{100s}{101}}) \subseteq \mathbb{T}_s.$

   Set $x=a+\mathbf{i}b$,  by Lemma \ref{mean value},  we have
   \begin{equation}\label{6.54}
   \begin{split}
  \mathfrak{Im}(v(x))&=\mathfrak{Im}(v(x)-v(a))\\
   &=\mathfrak{Im}(\mathbf{i}b(\lambda_1(1+f'(w_1))+\lambda_2(1+f'(w_2))))\nonumber\\
   &= b \mathfrak{Re}(1+\lambda_1(f'(w_1))+\lambda_2(f'(w_2))))\nonumber\\
   &= b(1+\mathfrak{Re}(\lambda_1(f'(w_1))+\lambda_2(f'(w_2)))).
  \end{split}
   \end{equation}
  Since $ |f'(w_i)|_s \leq |f|_{s,p}\leq \frac{1}{100}$, one gets
  \begin{equation}\label{6.541}
  |\mathfrak{Im}(v(z))|\leq s, \quad for \ all \ z\in \mathbb{T}_{\frac{100s}{101}},
  \end{equation}
  that is  equivalent to $v(\mathbb{T}_{\frac{100s}{101}}) \subseteq \mathbb{T}_s.$ Now, we would compare $u(x)$ with $u(x+f(x))$.

Clearly, we can  see
   \begin{equation}\label{6.41}
   |u\circ v (x)|_{\frac{100s}{101}} \leq |u(x)|_s.
    \end{equation}

Differentiating $u\circ v(x)$, one gets
\begin{equation}
D(u \circ v)(x)=(Du)(v(x))[1+Df(x)].
\end{equation}
By \eqref{6.541}, we have
   \begin{equation}\label{6.42}
   |D(u\circ v)(x)|_{\frac{100s}{101}} \leq |Du(x)|_s + |Du(x)|_s |Df(x)|_{\frac{100s}{101}}.
   \end{equation}

   $(\mathbf{i})$ is proved for $p=0,1$. Considering the general situation, by the "chain rule", the $m^{th}$ $ Fr\acute{e}chet $ derivative of the composition of functions $(u\circ v)(x)$ is
  \begin{equation}\label{6.421}
  D^m(u\circ v)(x) = \sum_ {k=1} ^m  \sum _{j_1+\cdots+ j_k = m} C_{kj}(D^ku)(v(x))[D^{j_1}v(x) \cdots D^{j_k}v(x)].
  \end{equation}
  where $j_1, . . . , j_k \geq1$, and $C_{kj}$ are constants depending on $k, j_1, . . . , j_k$ \cite[p 147]{Hamilton1}.

  For $j_i =1$, $|Dv(x)|_{s}=|1+Df(x)|_{s}<2$.

  For $ 2\leq j_i \leq m$, $|D^{j_i}v|_{s}=|D^{j_i}f|_{s}\leq |f|_{s,m} \leq \frac{1}{100}$ .

 Collecting all the term in the sum, we have
   \begin{equation}\label{6.43}
    |D^m(u\circ v)(x)|_{\frac{100s}{101}} \leq C \sum_ {k=1} ^m  |D^ku(y)|_{s}.
   \end{equation}
 Finally, \eqref{6.41}, \eqref{6.42} and \eqref{6.43} imply
   \begin{equation}\label{6.431}
   |u(x+f(x))|_{\frac{100s}{101},p} \leq C |u|_{s,p}.
   \end{equation}

    $(\mathbf{ii})$ Set $x+f(x)$ as $v_1$, $x+g(x)$ as $v_2$, we see
     \begin{equation}\label{6.44}
     \begin{split}
    |(u\circ v_1-u\circ v_2)(x)|_{\frac{100s}{101}} &\leq |Du(x)|_s|v_1(x)-v_2(x)|_{\frac{100s}{101}}\\
    &=|Du(x)|_s|f(x)-g(x)|_{\frac{100s}{101}}
    \end{split}
    \end{equation}

    Differentiating the left of  inequality \eqref{6.44}, we have
    \begin{equation}\label{6.45}
    \begin{split}
    &D(u\circ v_1-u\circ v_2)(x)\\
    =&(Du)(v_1(x))+(Du)(v_1(x))Df(x)-(Du)(v_2(x))-(Du)(v_2(x))Dg(x)\\
    =&(Du)(v_1(x))-(Du)(v_2(x))+((Du)(v_1(x))-(Du)(v_2(x)))Df(x)\\
    &+(Du)(v_2(x))(Df(x)-Dg(x)).
    \end{split}
    \end{equation}
    Then,  we can get
    \begin{equation}\label{6.46}
    \begin{split}
    &|D(u\circ v_1-u\circ v_2)(x)|_{\frac{100s}{101}}\\
    =&|(Du)(v_1(x))-(Du)(v_2(x))|_{\frac{100s}{101}}+|(Du)(v_1(x))-(Du)(v_2(x))|_{\frac{100s}{101}}|Df(x)|_{\frac{100s}{101}} \\
    &+|(Du)(v_2(x))|_{\frac{100s}{101}}|Df(x)-Dg(x)|_{\frac{100s}{101}}\\
    \leq&|D^2u(x)|_s|f(x)-g(x)|_{\frac{100s}{101}}+|D^2u(x)|_s|f(x)-g(x)|_{\frac{100s}{101}}|Df(x)|_{\frac{100s}{101}}\\ &+|Du(x)|_{s}|D(f(x)-g(x))|_{\frac{100s}{101}}\\
    \leq&|u|_{s,2}|f-g|_{{\frac{100s}{101}},1}.
    \end{split}
    \end{equation}

    $(\mathbf{ii})$ is proved for $p=0,1$. For the general form $D^m(u\circ v_1-u\circ v_2)(x)$,  we have
    \begin{equation}\label{6.99}
    \begin{split}
    D^m(u\circ v_1-u\circ v_2)(x)=&\sum_ {k=1} ^m  \sum _{j_1+\cdots+ j_k = m} C_{kj}\Big\{(D^ku)\circ v_1[D^{j_1}v_1(x) \cdots D^{j_k}v_1(x)]\\
    &-(D^ku)\circ v_2 [D^{j_1}v_2(x) \cdots D^{j_k}v_2(x)]\Big\}\\
    =&\sum_ {k=1} ^m  \sum _{j_1+\cdots+ j_k = m} C_{kj}\Big\{(D^ku)\circ (v_1-v_2)[D^{j_1}v_1(x)\cdots D^{j_k}v_1(x)]\\ &+(D^ku)\circ v_2[(D^{j_1}(f-g)(x)\cdot D^{j_2}v_1(x) \cdots D^{j_k}v_1(x)]+\cdots\\
    &+ (D^ku)\circ v_2[(D^{j_1}(v_2)(x)\cdots D^{j_{k-1}}v_2(x) \cdot D^{j_k}(f-g)(x)] \Big\}
    \end{split}
    \end{equation}

   From identity \eqref{6.99}, we  see
    \begin{equation}
    \begin{split}
    &|D^m(u\circ v_1-u\circ v_2)(x)|_{\frac{100s}{101}} \\
    \leq & C|u|_{s,m+1}|f-g|_{\frac{100s}{101},m}(1+|f|_{\frac{100s}{101},m}+|g|_{\frac{100s}{101},m}) \nonumber\\
    \leq & C|u|_{s,m+1}|f-g|_{\frac{100s}{101},m},
    \end{split}
    \end{equation}
    because$|f|_{s,m},|g|_{s,m} \leq \frac{1}{100}$. (Some repeated details has been omitted, that are almost the same as \eqref{6.45} and \eqref{6.46}.)

    $(\mathbf{iii})$  Let $\lambda_1,\lambda_2 \in \Pi$, $u_1=u_{\lambda_1},u_2=u_{\lambda_2}$, $x+f_{\lambda_1}(x)=v_1(x),x+f_{\lambda_2}(x)=v_2(x)$. Using \eqref{6.51} and \eqref{6.52}, one gets
    \begin{equation}\label{chazhi}
    \begin{split}
    |u_2\circ v_2-u_1\circ v_1|_{\frac{100s}{101},p}
    \leq&|u_2\circ v_2-u_2\circ v_1 |_{\frac{100s}{101},p}+|u_2\circ v_1-u_1\circ v_1|_{\frac{100s}{101},p} \\
    \leq& C(|u_2|_{s.p+1} |v_2-v_1|_{\frac{100s}{101},p} \\
    & \ +|u_2-u_1|_{s.p}(1+|v_1|_{\frac{100s}{101},p})).
    \end{split}
    \end{equation}

 Finally, \eqref{6.53} follows from \eqref{6.51} and \eqref{chazhi}.
   \end{proof}
\begin{lemma}\label{estimation3}
Let $p>0$, $\eta > 0$, $0<k<1 $, $u(\lambda),h(\lambda)$ be a Lipschitz family of function with  $\|h\|^{Lip}_{s,p+\eta} \leq 1$. $F$ be a $C^1$-map satisfying
\begin{equation}
\|F(u)\|^{Lip}_{ks,p} \lessdot \|u\|^{Lip}_{s,p+\eta}.
\end{equation}
\begin{equation}
\|\partial_uF(u)[h]\|^{Lip}_{ks,p} \lessdot \|h\|^{Lip}_{s,p+\eta}(1+ \|u\|^{Lip}_{s,p+\eta}).
\end{equation}
Then,
\begin{equation}
\|F(u+h)-F(u)\|^{Lip}_{ks,p} \lessdot \|h\|^{Lip}_{s,p+\eta}(1+ \|u\|^{Lip}_{s,p+\eta}).
\end{equation}
\end{lemma}
\begin{proof}Since $F(u)$ be a $C^1$-map, we see
\begin{equation*}
F(u+h)-F(u)=\int^1_0\partial_{(u+th)}F(u+th)[h]dt.
\end{equation*}
Then, we have
\begin{equation*}
\begin{split}
\|F(u+h)-F(u)\|^{Lip}_{ks,p}&\leq \int^1_0\|\partial_{(u+th)}F(u+th)[h]\|^{Lip}_{ks,p}dt \\
&\lessdot \|h\|^{Lip}_{s,p+\eta}(1+\int^1_0\|u+th\|^{Lip}_{s,p+\eta})dt \\
&\leq\|h\|^{Lip}_{s,p+\eta}(1+\|u\|^{Lip}_{s,p+\eta}+\|h\|^{Lip}_{s,p+\eta})\\
&\lessdot\|h\|^{Lip}_{s,p+\eta}(1+\|u\|^{Lip}_{s,p+\eta}),
\end{split}
\end{equation*}
because $\|h\|^{Lip}_{s,p+\eta} \leq 1$.
\end{proof}

\begin{lemma}\label{estimation4}
$\mathbf{(The \ implicit \ function)}$ Let p be analytic on $\mathbb{T}^n_s$, with $|p|_{s.m} \leq \frac{1}{100}$. Set  $f(x)=x+p(x)$. Then:

  $(\mathbf{i})$ $f$ is invertible, its inverse is $f^{-1}(y)=g(y)=y+q(y)$, where $q$ be real analytic on $\mathbb{T}^n_{\frac{99}{100} s}$, and satisfying
$$|q|_{\frac{99}{100}s.m} \leq C|p|_{s.m},$$
where the constant C depends on $n$ and $m$.

  $(\mathbf{ii})$ Moreover, suppose that $p=p_\lambda$ depends in a Lipschtiz way by a parameter $\lambda \in \Upsilon \subset \mathbb{R}$, and  $|p_\lambda|_{s,m} \leq \frac{1}{100}$, for all $\lambda$. Then $q=q_\lambda$ is also Lipschitz in $\lambda$, and
$$|q|^{Lip}_{\frac{99}{100}s,m} \leq C |p|^{Lip}_{s,m+1}.$$
The constant C depends on $n$ and $m$.
\end{lemma}

\begin{proof} For symbolic simplicity, the following calculations only consider $n=1$ in form. With regard to general situation, there is no difference.

$(\mathbf{i})$  If we restrict $x$ to $\mathbb{R}$, by \cite[Lemma B.4]{Baldi1}, $f(x)$ is a homeomorphism from $\mathbb{R}$ to $\mathbb{R}$.  Considering $f(x)$ defined on $\mathbb{T}_s$, by Lemma \eqref{mean value},  $f(x)$ is a  one to one mapping from $\mathbb{T}_s$ to  $f(\mathbb{T}_s)$.

Now, we would prove $\mathbb{T}_{\frac{99s}{100}} \subseteq f(\mathbb{T}_s)$.

$\mathbf{(1)}$: Set $x=a+\mathrm{i}b$, by \eqref{6.54}, we can see
 \begin{equation*}
 \begin{split}
  \mathfrak{Im}(f(x))&=\mathfrak{Im}(f(x)-f(a))\\
  &= b+b\mathfrak{Re}(\lambda_1(Dp(w_1))+\lambda_2(Dp(w_2)))).
 \end{split}
   \end{equation*}
  Since $ |Dp(w_i)|_s \leq |p|_{s,m} \leq \frac{1}{100}$,  we have
  \begin{equation}\label{7.1}
  |\mathfrak{Im}(z+f(z))|\geq \frac{99s}{100}, \quad for \ all \ z\in \mathbb{T}_{s}.
  \end{equation}

 $\mathbf{(2)}$: Let $q(y)=g(y)-y$. Since $p(x)$ is periodic, $p(x+2\pi m)=p(x)$. Then,
 \begin{equation}
 \begin{split}
 f (x + 2\pi m) &= x+2\pi m +p(x+2\pi m)\\
 &=x+2\pi m+ p(x)\\
 &=f (x) + 2\pi m
 \end{split}
 \end{equation}
 for all $ m \in \mathbb{Z}^n$.
 Note $g$ is the inverse of $f$, applying $g$ to this equality gives
  $$g\circ f (x + 2\pi m) =g(f(x) + 2\pi m),$$
  where
  $$g \circ f(x+2\pi m)=x+ 2\pi m=g(y)+2\pi m.$$
  On the other hand, $g(f(x)+ 2\pi m)=g(y+2\pi m)$. This means that $q(y)$ is periodic.

   From $\mathbf{(1)},\mathbf{(2)}$,  $g(y)$ are well defined on  $\mathbb{T}^n_{\frac{99}{100}s}$. Now, we would compare $p(x)$ with $q(y)$.

By Neumann series, the matrix $Df(x) = I + Dp(x)$ is invertible, where $(Df(x))^{-1}=\sum^{\infty}_{n=0}(-Dp(x))^n$.  Thus,
 \begin{equation}\label{6.62}
Dq(f(x))=\sum^{\infty}_{n=1}(-Dp(x))^n, \ for \ all \ x\in \mathbb{T}_s.
 \end{equation}
Since $|Dp(x)|_s < |p|_{s,m} \leq \frac{1}{100}$, we see
  \begin{equation}\label{6.61}
   |Dq(f(x))|  \leq  \frac{100} {99}|Dp(x)| \leq \frac{1}{99}, \ for \ all \ x\in \mathbb{T}_s.
 \end{equation}

 The identity $ f (g(y)) = y$ gives
 \begin{equation}\label{6.72}
 q(y)=-p(y+q(y)),\quad  y\in \mathbb{T}^n_{\frac{99}{100}s}.
 \end{equation}

From \eqref{7.1} and \eqref{6.61}, $|Dq|_{\frac{99s}{100}} \leq \frac{1}{99}$. Similarity with \eqref{6.54} and \eqref{6.541}, we have
$$|q|_{\frac{99}{100}s}<|p|_s.$$

Clearly, $|Dq|_{\frac{99}{100}s}\leq C |Dp|_s.$  $(\mathbf{i})$ is proved for $m=0,1$.

Considering the  general situation, suppose $|q|_{\frac{99}{100}s,h} \leq C(h)|p|_{s.h}$, for $h < m$.

Apply \eqref{6.421} to $f\circ g$: since $f(g(y))=y, D^m(f\circ g)=0$ for all $m\geq 2$. Separate $k=1$ from $k \geq 2$ in the sum \eqref{6.421} and solve for $D^mg$,
\begin{equation}
D^mg(y)=-Dg(y)\sum_ {k=2} ^m  \sum _{j_1+\cdots+ j_k = m}C_{kj}(D^kf)(g(y))[D^{j_1}g(y) \cdots D^{j_k}g(y)].
\end{equation}
$D^mg=D^mq$ and $D^kf=D^kp$, because $k,m \leq 2$. Since $k \geq 2$, it is $1 \leq j_i \leq m-1$ for all $i=1,...,k$, because there are at least two $j_1,j_2$, each of them $\geq 1$, and $\sum j_i=m$.

 For $k=m$, one has $j_i=1$ for all $i=1,...,m$, and the corresponding term in the sum is estimated
\begin{equation}
|(D^mp)\circ g[Dg,...,Dg]|_{\frac{99s}{100}}\leq |D^mp|_{s}|Dg|_{\frac{99s}{100}}^m \leq C|D^mp|_{s},
\end{equation}
because $|Dg|_{\frac{99s}{100}}=|I+Dq|_{\frac{99s}{100}} <2$.

For $2\leq k \leq m-1$, at least one among $j_1,...,j_k                                       $ is $\geq 2$(otherwise $k=m$). Let $\ell$ be the number of indices $j_i$ that are $\geq 2$, so that $1\leq \ell \leq k$. It remains to estimate
\begin{equation}
\sum^{m-1}_{k=2} \sum^k_{\ell=1} \sum_{\sigma_1+\cdots \sigma_{\ell}=m-k+\ell}C_{k\ell\sigma}(D^kp)(g(y))[Dg(y)]^{k-\ell}[D^{\sigma_1}q(y) \cdots D^{\sigma_{\ell}}q(y)],
\end{equation}
where indices $j_i \geq 2 $ have been renamed $\sigma_1,...,\sigma_{\ell}$, the number of indices $j_i=1$ is $k-\ell$, and $D^{\sigma_i}g=D^{\sigma_i}q$ because $\sigma_i \geq 2$. Every factor $Dg$ is estimated by $|Dg|_{\frac{99s}{100}} < 2$. For the remaining factors
\begin{equation}
\begin{split}
|(D^kp)\circ g[D^{\sigma_1}q(y) \cdots D^{\sigma_{\ell}}q(y)]|_{\frac{99s}{100}} &\leq |D^kp|_{s}|[D^{\sigma_1}q(y) \cdots D^{\sigma_{\ell}}q(y)]|_{\frac{99s}{100}}\\
& \leq C |D^kp|_{s},
\end{split}
\end{equation}
because $|D^{\sigma_i}q(y)|_{\frac{99s}{100}} \leq |q|_{\frac{99s}{100},m-1} \leq C |p|_{s,m-1} \leq C$.

Collecting all the terms in the sum, we can proved that
 \begin{equation*}
  |D^mq|_{\frac{99s}{100}} \leq C(m) |p|_{s,m}
 \end{equation*}
Finally, we have
  \begin{equation}\label{6.75}
 |q|_{\frac{99s}{100},m} \leq C(m) |p|_{s,m}.
  \end{equation}

  $(\mathbf{ii})$ For the Lipschitz norm, we have
  $$q_\lambda(y)+p_\lambda(g_\lambda(y))=0,\quad \forall \lambda \in \Pi,y\in \mathbb{T}^n_{\frac{99}{100}s}.$$

   Let $\lambda_1,\lambda_2 \in \Pi $, $q_1=q_{\lambda_1},q_2=q_{\lambda_2}$, and so on, then
   \begin{equation}\label{6.73}
   q_1-q_2=(p_2\circ g_2-p_1\circ g_2)+(p_1\circ g_2-p_1\circ g_1).
   \end{equation}

   Since $|Dp_\lambda|< \frac{1}{100}$, for all $\lambda $, $g_{\lambda}(y)$ are well defined on  $\mathbb{T}_{\frac{99s}{100}}$.  From \eqref{6.54} and \eqref{6.44}, we have
   $$|q_1-q_2|_{\frac{99s}{100}} \leq C|p_2-p_1|_s+|Dp_1|_s|q_2-q_1|_\frac{99s}{100},$$
   and $(1-|Dp_1|_s)|q_1-q_2|_{\frac{99s}{100}} \leq C |p_2-p_1|_s$. Thus, we can get
  \begin{equation}
   |q_1-q_2|_{\frac{99s}{100}} \leq C |p_2-p_1|_s.
   \end{equation}

   Differentiating \eqref{6.73}, by \eqref{6.44} and \eqref{6.46}, we have
    \begin{eqnarray*}
    |Dq_1-Dq_2|_{\frac{99s}{100}}\leq C|p_2-p_1|_{s,1}+|D^2p_1|_{s}|q_2-q_1|_{\frac{99s}{100}}|Dg_2|_{_{\frac{99s}{100}}}+|Dp_1|_sDq_1-Dq_2|_{\frac{99s}{100}}
    \end{eqnarray*}
    and $(1-|Dp_1|_s)|Dq_1-Dq_2|_{\frac{99s}{100}}\leq C|p_2-p_1|_{s,1}+|p_1|_{s,2}|p_2-p_1|_{s}|Dg_2|_{\frac{99s}{100}}.$ Thus, we can get

   \begin{equation}
   |Dq_1-Dq_2|_{\frac{99s}{100}}\leq C |p_2-p_1|_{s,1}(1+|p_1|_{s,2})
   \end{equation}

   $(\mathbf{ii})$ is proved for $m=0,1$. Considering general situation, suppose $|q_1-q_2|_{\frac{99s}{100},h} \leq C(h) |p_2-p_1|_{s,h}(1+|p_1|_{s,h+1})$,  for all $h < m $.

   The estimates of $D^m(q_1-q_2)$ can be divined into the following two parts.

   $\mathbf{(A)}$: Considering $D^m(p_1\circ g_2-p_1\circ g_1)$, we have
   \begin{equation}
   \begin{split}
   D^m(p_1\circ g_2-p_1\circ g_1)=& \sum_ {k=1} ^m  \sum _{j_1+\cdots+ j_k = m} C_{kj}\Big\{(D^kp_1)\circ g_2[D^{j_1}g_2(y) \cdots D^{j_k}g_2(y)]\\
   &-(D^kp_1)\circ  g_1 [D^{j_1}g_1(x) \cdots D^{j_k}g_1(y)]\Big\}.\\
   \end{split}
   \end{equation}

   For $k=1$ one has $j_1=m$, and the corresponding term in the sum is estimated
   \begin{equation}
   \begin{split}
   |(Dp_1)\circ g_2 \cdot  D^mq_2-(Dp_1) \circ g_1 \cdot  D^mq_1|_{\frac{99s}{100}}\leq & |(Dp_1)\circ g_2 \cdot D^mq_2-(Dp_1)\circ g_2\cdot D^mq_1|_{\frac{99s}{100}}\\
    +&|(Dp_1)\circ g_2 \cdot D^mq_1-(Dp_1)\circ g_1 \cdot D^mq_1|_{\frac{99s}{100}}\\
    \leq&|Dp_1|_s|D^m(q_2-q_1)|_{\frac{99s}{100}} \\
   +&|Dp_1|_{s} |q_2-q_1|_{\frac{99s}{100}} |D^mq_1|_{\frac{99s}{100}}
   \end{split}
   \end{equation}

   For $k \geq 2$, one has $j_i< m$. It remains to estimate
    \begin{equation}
    \begin{split}
    \sum_ {k=2} ^m  \sum _{j_1+\cdots+ j_k = m} C_{kj}&\Big\{[(D^kp_1)\circ g_2-(D^kp_1)\circ g_1][D^{j_1}g_2(y)\cdots D^{j_k}g_2(y)]\\ &+(D^kp_1)\circ g_1[D^{j_1}(g_2-g_1)(y)\cdot D^{j_2}g_2(y) \cdots D^{j_k}g_2(y)]+\cdots\\
    &+ (D^kp_1)\circ g_1[(D^{j_1}(g_1)(y)\cdots D^{j_{k-1}}g_1(y) \cdot D^{j_k}(g_2-g_1)(y)] \Big\}
    \end{split}
    \end{equation}
   Every factor $|Dg|_{\frac{99s}{100}} < 2$, and  $|(D^kp_1)\circ g_2-(D^kp_1)\circ g_1|_{\frac{99s}{100}} \leq  |D^{k+1}p_1|_s|q_2(y)-q_1(y)|_s$.  For the remaining factors,
   \begin{equation}
   \begin{split}
   |(D^kp_1)\circ g_1[(D^j_1g_1\cdots D^{j_i}(g_2-g_1) \cdots D^{j_{k}}g_2]|_{\frac{99s}{100}} &\leq C |D^kp_1|_s|D^{j_i}(q_2-q_1)(x)|_{\frac{99s}{100}}\\
   &\leq C |D^kp_1|_s ||q_2-q_1|_{\frac{99s}{100},m}.
   \end{split}
   \end{equation}

   $\mathbf{(B)}$:Considering $D^m(p_2\circ g_2-p_1\circ g_2)$, by \eqref{6.43}, we have
   \begin{equation}
   |D^m(p_2\circ g_2-p_1\circ g_2)|_{\frac{99s}{100}} \leq C|p_2-p_1|_{s,m}.
   \end{equation}
   Collecting all the terms above in the sum, we can see that
   \begin{equation}
   \begin{split}
   |D^mq_1-D^mq_2|_{\frac{99s}{100},m} \leq & C(|p_2-p_1|_{s,m}+|p_1|_{s,m+1}|q_2-q_1|_{\frac{99s}{100},m-1})\\
   &+|Dp_1|_s|D^mq_1-D^mq_2|_{\frac{99s}{100}}.
   \end{split}
   \end{equation}
 Since$|Dp_1|_s \leq |p_1|_{s,m} \leq \frac{1}{100}$, we can see
   \begin{equation}
   |D^mq_1-D^mq_2|_{\frac{99s}{100},m} \leq C (|p_2-p_1|_{s,m}+|p_1|_{s,m+1}|p_2-p_1|_{\frac{99s}{100},m-1}),
   \end{equation}
  and
   \begin{equation}\label{chazhi2}
   |q_1-q_2|_{\frac{99s}{100},m} \leq C |p_2-p_1|_{s,m}(1+|p_1|_{s,m+1}).
  \end{equation}

  Finally, the bounds  \eqref{6.75} and \eqref{chazhi2} imply $|q|^{Lip}_{\frac{99s}{100},m} \leq C |p|^{Lip}_{s,m+1}$
 \end{proof}

\section*{Acknowledgement}
The first author would like thank to W. Jian, Y. Shi for comments on earlier version of this manuscript. Special thanks is to  H. Cong for many useful discussion.


\begin{thebibliography}{99}


\bibitem{Baldi1}Baldi,P: Periodic solutions of fully nonlinear autonomous equations of Benjami-Ono type[J]. Ann. I. H. Poincar$\acute{\mbox{e}}$ (C) Anal. Non Lin$\acute{\mbox{e}}$aire, 2013, 30(1): 33-77.
 \bibitem{Berti1}Baldi,P. Berti,M. Montalto,R: KAM for quasi-linear and fully nonlinear forced perturbations of Airy equation[J]. Mathematische Annalen, 2014, 359(1): 1-66.

\bibitem{Berti3}Baldi,P. Berti,M. Montalto,R: KAM for autonomous quasi-linear perturbations of KdV[J].  Ann. I. H. Poincar$\acute{\mbox{e}}$ (C) Anal. Non Lin$\acute{\mbox{e}}$aire, 2016, 33(6): 1589-1638.

\bibitem{Berti4}Baldi,P. Berti,M. Montalto,R: KAM for autonomous quasi-linear perturbations of mKdV[J]. Bollettino dell'Unione Matematica Italiana, 2016, 9(2): 143-188.
\bibitem{Bambusi1}Bambusi,D. Graffi,S: Time quasi-periodic unbounded perturbations of schr$\ddot{\mbox{o}}$dinger operators and KAM methods [J]. Communications in Mathematical Physics, 2001, 219(2): 465-480.
\bibitem{Berti0} Berti,M. Montalto,R: Quasi-periodic standing wave solutions of gravity-capillary water waves. {\it arXiv1602.02411}, 2016.
\bibitem{Bogolyubov1}Bogolyubov,N,N. Mitropolskii,Yu,A. Samoilenko,A,M: Methods of accelerated convergence in nonlinear mechanics[M]. Springer, 1976.

\bibitem{Bourgain1}Bourgain,J: Green's function estimates for lattice Schr$\ddot{\mbox{o}}$dinger operators and applications [M]. Annals of Mahematics Studies 158, Princeton University Press, 2005.
\bibitem{Cong1}Cong,H. Mi,L. Yuan,X: Positive quasi-periodic solutions to Lotka-Volterra system[J]. Science China Mathematics, 2010, 53(5): 1151-1160.
\bibitem{Chun C1}Chun,C: Solitons and periodic solutions for the fifth-order KdV equation with the Exp-function method[J]. Applied Mathematics Letters, 2008, 372(16): 2760-2766.
\bibitem{Feola1}Feola,R. Procesi,M: Quasi-periodic solutions for fully nonlinear forced reversible Schr$\ddot{\mbox{o}}$dinger equations [J]. Journal of Differential Equations, 2014, 259(7): 3389-3447.
\bibitem{Feola2}Feola,R: KAM for quasi-linear forced hamiltonian NLS. {\it arXiv:1602.01341}, 2016.
\bibitem{Kappler1}Kappeler,T. P$\ddot{\mbox{o}}$schel,J: KdV and KAM[M]. Springer, 2003.
\bibitem{KuKsin0}Kuksin,S: Nearly integrable infinite-dimensional hamiltonian systems[M]. Springer, 1993.
\bibitem{KuKsin1}Kuksin,S. On small-denominators equations with large variable coefficients[J]. Zeitschrift f$\ddot{\mbox{u}}$r angewandte Mathematik und Physik, 1997, 48(2): 262-271.
\bibitem{KuKsin2}Kuksin,S: A KAM theorem for equations of the Korteweg-De Vries Type[J]. Rev. Math. Math phys, 1998, 10(3): 1-64.
\bibitem{KuKsin3}Kuksin,S: Analysis of Hamiltonian PDEs[M]. Oxford University Press, 2000.

\bibitem{KuKsin4}Kuksin,S. P$\ddot{\mbox{o}}$schel,J: Invariant Cantor Manifolds of quasi-periodic oscillations for a nonlinear Schr$\ddot{\mbox{o}}$dinger equation[J]. Annals of Mathematics, 1996, 143(1): 149-179.

\bibitem{Lax1}Lax,P: Periodic solutions of the KdV equation[J]. Siam Review, 2004, 18(3): 438-462.
\bibitem{Liu1}Liu,J. Yuan.X: Spectrum for quantum duffing oscillator and small-divisor equation with large-variable coefficient[J]. Communications on Pure \& Applied Mathematics, 2010, 63(9): 1145-1172.
\bibitem{Liu2}Liu,J: Yuan.X: A KAM theorem for hamiltonian partial differential equations with unbounded perturbations[J]. Communications in Mathematical Physics, 2011, 307(3): 629-673.
 \bibitem{mcleod}Mcleod,R: Mean value theorems for vector valued functions[J]. Proceedings of the Edinburgh Mathematical Society, 1965, 14(3): 197-209.

\bibitem{montalto1}Montalto,R: Quasi-periodic solutions of forced Kirchhoff equation[J]. Nonlinear. Differ. Equ. Appl, 2017, 24(9).
\bibitem{Poschel1}P$\ddot{\mbox{o}}$schel,J: A KAM-theorem for some nonlinear partial differential equations[J]. Ann. Sci. Norm. Sup. Pisa Cl. Sci, 1996, 23(4): 119-148.
\bibitem{Poschel2}P$\ddot{\mbox{o}}$schel,J: Quasi-periodic solutions for a nonlinear wave equation[J]. Commentarii Mathematici Helvetici, 1996, 71(1): 269-296.
\bibitem{Wayne1}Wayne,E: Periodic and quasi-periodic solutions of nonlinear wave equations via KAM theory[J]. Communications in Mathematical Physics, 1990, 127(3): 479-528.
 \bibitem{Wazwaz1}Wazwaz,A: The extended tanh method for new solitons solutions for many forms of the fifth-order KdV equations[J]. Applied Mathematics \& Computation, 2007, 184(2): 1002-1014.
 \bibitem{Yuan1}Yuan,X. Zhang,K: A reduction theorem for time dependent Schr$\ddot{\mbox{o}}$dinger operator with finite differentiable unbounded perturbation[J]. Journal of Mathematical Physics, 2013, 54(5): 465-480.

\bibitem{Zehnder1}Zehnder,E: Generalized implicit function theorems with applications to some small divisor problems, I[J]. Communications on Pure \& Applied Mathematics, 1975, 28(1): 91-140.
\bibitem{Zehnder1}Zehnder,E: Generalized implicit function theorems with applications to some small divisor problems, II[J]. Communications on Pure \& Applied Mathematics, 1976, 29(1): 49-111.
\bibitem{Zhang1}Zhang,J. Gao,M. Yuan,X: KAM tori for reversible partial differential equations[J]. Nonlinearity, 2011, 24(4): 1189-1228.
 \end{thebibliography}
\end{document}